\documentclass[oneside,reqno]{amsart}
\usepackage[usenames]{color}

\newcommand{\indic}[1]{\id_{\{#1\}}}

\usepackage{amsmath,amsfonts,amssymb,graphics,amsthm,pgfplots}
\usepackage[foot]{amsaddr}

\pgfplotsset{every axis/.append style={
    axis x line=middle,    
    axis y line=middle,    
    axis line style={->},  
    xlabel={$\gamma$},     
    ylabel={$\mu$},        
    x label style={at={(axis description cs:0.5,-0.05)},anchor=north},
    y label style={at={(axis description cs:-0.1,.46)},anchor=south},              }}
\tikzset{>=stealth}

\usepackage{hyperref}
\usepackage{mathtools}
\usepackage{a4wide}
\usepackage{bbm}
\usepackage{enumerate}

\newtheorem{theorem}{Theorem}[section]
\newtheorem{lemma}[theorem]{Lemma}
\newtheorem{proposition}[theorem]{Proposition}
\newtheorem{corollary}[theorem]{Corollary}
\newtheorem{assumption}[theorem]{Assumption}
\theoremstyle{remark}
\newtheorem{remark}[theorem]{Remark}

\newcommand{\PP}{\mathbb{P}}
\newcommand{\EE}{\mathbb{E}}
\newcommand{\id}{\mathbbm 1}
\newcommand{\invsigma}{\sigma^{-1}}
\newcommand{\ZZ}{\mathbb{Z}}

\DeclareMathOperator*{\argmax}{arg\,max}

\begin{document}

\title[Localisation in the Bouchaud--Anderson model]{Localisation in the Bouchaud--Anderson model}

\author{Stephen Muirhead$^1$}
\address{$^1$Department of Mathematics, University College London (Currently: Mathematical Institute, University of Oxford)}
\email{muirhead@maths.ox.ac.uk}
\author{Richard Pymar$^2$}
\address{$^2$Department of Mathematics, University College London (Currently: Department of Economics, Mathematics and Statistics - Birkbeck)}
\email{r.pymar@bbk.ac.uk}

\begin{abstract}
It is well-known that both random branching and trapping mechanisms can induce localisation phenomena in random walks; the prototypical examples being the parabolic Anderson and Bouchaud trap models respectively. Our aim is to investigate how these localisation phenomena interact in a hybrid model combining the dynamics of the parabolic Anderson and Bouchaud trap models. Under certain natural assumptions, we show that the localisation effects due to random branching and trapping mechanisms tend to (i) mutually reinforce, and~(ii) induce a local correlation in the random fields (the `fit and stable' hypothesis of population dynamics).
\end{abstract}


\thanks{Both authors were supported by the Leverhulme Research Grant RPG-2012-608 held by Nadia Sidorova. The first author was also partial supported by the Engineering \& Physical Sciences Research Council (EPSRC) Fellowship EP/M002896/1 held by Dmitry Belyaev. We would like to thank Nadia Sidorova for many invaluable suggestions, and Franziska Flegel, Onur G\"{u}n, Renato Soares dos Santos and two anonymous referees for helpful comments.}

\date{\today}

\maketitle

\noindent \textbf{Minor revision to published version.}
This is an updated version of \cite{MP16} containing the following minor revisions:
\begin{itemize}
\item A typo in the statement of Proposition \ref{prop:thm4.1} has been corrected;
\item The coupling used in Section \ref{sec:extremal} has been slightly modified to fix a gap in its original statement; we thank Renato Soares dos Santos for pointing this out to us;
\item A slight correction has been made to the proof of Proposition \ref{prop:quickpath2}; and
\item The bibliography has been updated.
\end{itemize}

\section{Introduction}
\label{sec:intro}
\subsection{The Bouchaud--Anderson model}
This paper studies a certain random walk model on~$\mathbb{Z}^d$ that is a hybrid of the well-known \textit{parabolic Anderson} (PAM) and \textit{Bouchaud trap} (BTM) models. To introduce this model, first recall the PAM, which describes the evolution of a diffusive particle in a random potential field (or, equivalently, a random branching environment; see below). Precisely, the PAM is the Cauchy problem on the lattice $\mathbb{Z}^d$
\begin{align} \label{eq:PAM}
\frac{\partial u(t, z)}{\partial t} &= (\Delta + \xi) \, u(t, z)\,,  & (t, z) \in [0, \infty) \times \ZZ^d \,;\\
\nonumber u(0, z) &= \id_{\{0\}}(z)\,,  & z \in \ZZ^d\,;
\end{align}
where $\xi = \{\xi(z)\}_{z \in \mathbb{Z}^d}$ is a collection of independent identically distributed (i.i.d.) random variables known as the (random) \textit{potential field} and $\Delta$ is the \textit{discrete Laplacian} defined by $(\Delta f)(z) = \sum_{|y - z| = 1} (2d)^{-1} (f(y) - f(z)) $, where $|\cdot |$ denotes the $\ell_1$-norm. For a large class of potential field distributions,\footnote{More specifically, those satisfying a certain integrability condition on the upper-tail; see \cite{Gartner90}.} equation \eqref{eq:PAM} has a unique non-negative solution defined for all time $t$. For general background information on the PAM, including its origins in the statistical physics literature and its interpretation in terms of a system of branching diffusive particles, see~\cite{Gartner90}.

Recall also the BTM, which describes the evolution of a diffusive particle in a random trapping landscape. Precisely, the BTM is the continuous-time Markov chain on $\ZZ^d$ defined by the jump rates
\begin{align}
\label{eq:BTMjumps}
w_{z \to y} := \begin{cases}
(2d \sigma(z))^{-1}\,,  & \text{if } |y-z|=1\, , \\
0\,, &  \text{otherwise}\, ,
 \end{cases}
 \end{align}
where $\sigma = \{\sigma(z)\}_{z \in \ZZ^d }$ is a collection of strictly-positive i.i.d.\ random variables known as the (random) \textit{trapping landscape}. Remark that the density of the BTM satisfies the equation
\begin{align}
\label{eq:BTM}
\frac{\partial u(t, z)}{\partial t} &= \Delta \invsigma \, u(t, z)\,,  & (t, z) \in [0, \infty) \times \ZZ^d \, ,
\end{align}
where, for clarity, we stress that the operator $\Delta \invsigma$ acts as
\[  (\Delta \invsigma f)(z) = \sum_{|y-z| = 1} (2d \sigma(y))^{-1} f(y) - \invsigma(z) f(z) \, .  \]
For general background information on the BTM, including its origins in the study of spin-glasses dynamics and its broad utility as a simple model for a variety of trapping behaviour, see \cite{BenArous06}.

The PAM and BTM are of great interest in the theory of random processes because they exhibit \textit{intermittency}, that is, unlike other commonly studied models of diffusion, their long-term behaviour cannot, in general, be described with a simple averaging principle (see \cite{Gartner90} and \cite{BenArous06} for a general overview of the PAM and BTM respectively.) Instead, extremes in the respective random environments may create concentration effects, which can result in the eventual \textit{localisation} of the solution to equations \eqref{eq:PAM} and \eqref{eq:BTM} respectively over long periods of time. In the most extreme cases, the solution localises on just a few sites.

Our aim is to study how the localisation phenomena in the PAM and the BTM interact. To do this, we consider the Cauchy problem on the lattice $\ZZ^d$
\begin{align} \label{eq:BAM}
\frac{\partial u(t, z)}{\partial t} &= (\Delta \sigma^{-1} + \xi) u(t, z) \, , & (t, z) \in [0, \infty) \times \mathbb{Z}^d\,;\\
\nonumber u(0, z) &= \id_{\{0\}}(z) \, , & z \in \mathbb{Z}^d\,;
\end{align}
derived by replacing the discrete Laplacian in equation \eqref{eq:PAM} with the generator of the BTM in equation \eqref{eq:BTM}. We refer to equation \eqref{eq:BAM} as the \textit{Bouchaud--Anderson model} (BAM). 

By analogy with the PAM (see \cite{Gartner90}, Section 1.2), the solution to equation \eqref{eq:BAM} has a natural interpretation as the expected number of particles in a system of continuously-branching diffusive particles on the lattice $\ZZ^d$ specified by:
\begin{itemize}
\item \textit{Initialisation}: A single particle at the origin;
\item \textit{Branching}: The local branching rate for a particle at a site $z$ is given by $\xi(z)$;
\item \textit{Trapping}: Each particle evolves as an independent BTM, that is, the waiting time at each visit to a site $z$ is independent and distributed exponentially with mean $\sigma(z)$, with the subsequent site chosen uniformly from among the nearest neighbours.
\end{itemize}
This interpretation can be formalised in the Feynman-Kac representation of the solution to~\eqref{eq:BAM}:
\begin{align}
\label{eq:fk}
u(t, z) := \EE_{0} \left[ \exp \left\{ \int_0^t \xi(X_s) ds)  \right\} \id_{\{X_t = z\}} \right]\,,
\end{align}
where $X$ is the BTM and, for $z\in \ZZ^d$, $\EE_{z}$ denotes the expectation over $X$ given that $X_0=z$.
As we shall see, the interaction between the random branching and trapping mechanisms makes the localisation behaviour of the BAM highly non-trivial.

\subsection{Localisation in the PAM and BTM}
The PAM and BTM are said to \textit{localise} if, as $t \to \infty$, the solution of equations \eqref{eq:PAM} and~\eqref{eq:BTM} respectively are eventually concentrated on a small number of sites with overwhelming probability, i.e.\ if there exists a (random) \textit{localisation set} $\Gamma_t$ such that, as $t \to \infty$, $|\Gamma_t| = t^{o(1)}$ and
\begin{align}
\label{eq:local}
\frac{\sum_{z \in \Gamma_t} u(t, z)}{U(t)} \to 1 \, \qquad \text{in probability}\,,
\end{align}
where $U(t) := \sum_{z \in \ZZ^d} u(t, z)$ is the total mass of the solution (in the BTM, this is identically one); see Section \ref{subsec:notation} for the definition of the asymptotic notation used here and throughout the paper.

Naturally, the primary measure of the strength of localisation in the PAM and BTM is the cardinality of the localisation set $\Gamma_t$. As such, the most extreme form of localisation is \textit{complete localisation}, which occurs if the total mass is eventually concentrated at just one site, i.e.\ if $\Gamma_t$ can be chosen in equation \eqref{eq:local} such that $|\Gamma_t| = 1$. A finer measure of the strength of localisation is the \textit{radius of influence}, which measures the extent to which localisation sites themselves are determined by purely local features of the random environment. More precisely, the radius of influence $\rho$ is the smallest integer for which the localisation sites can be determined by maximising a functional on $\ZZ^d$ that depends on the random environments only through their values in balls of radius $\rho$ around each site.

Broadly speaking, localisation in the PAM and BTM is generated by the structure-forming effects of extremes in the respective random environment. If these extremes are both sufficiently pronounced and sufficiently regular, over long periods of time the model will come to adopt the structure present in the environment, with localisation the most extreme manifestation of this. Naturally then, the strength of localisation in the PAM and BTM should depend on (i) the asymptotic rate of decay, and (ii) the regularity of the upper-tail of the random variables $\xi(0)$ and $\sigma(0)$. In this context, it is convenient to restrict $\xi(0)$ and $\sigma(0)$ to be strictly-positive and to characterise these random variables by their \textit{exponential tail decay rate} function
$$ g_\xi(x) := -\log (\PP(\xi(0) > x))  \quad \text{and} \quad  g_\sigma(x) := -\log (\PP(\sigma(0) >  x)) $$
for then (i) and (ii) translate to the asymptotic growth and regularity of the non-decreasing functions $g_\xi$ and $g_\sigma$.

We briefly outline some known results on localisation in the PAM and BTM. For simplicity, we shall assume all necessary regularity conditions without further specification.

\subsubsection{Localisation in the parabolic Anderson model}
The conditions under which the PAM completely localises in the sense of equation \eqref{eq:local} has been the subject of intense and ongoing research over the last 25 years. The current understanding is that double-exponential tail decay ($g_\xi(x) \approx e^x$) forms the boundary of the complete localisation universality class. More precisely, it is conjectured that the PAM exhibits complete localisation as long as $\log g_\xi(x) \ll x$. This has been proven (in \cite{Konig09}) in the extremal\footnote{This case is extremal in the sense that if $g_\xi(x) \sim \gamma \log x$ for $\gamma > d$ or $\gamma = d = 1$ then the solution to equation~\eqref{eq:PAM} `blows-up' in finite time, see \cite{Gartner90}.} case of Pareto-like tail decay ($g_\xi(x) \sim \gamma \log x$, for $\gamma > d$), and more recently (in \cite{Sidorova12} and \cite{Fiodorov13}) in the case of Weibull-like tail decay ($g_\xi(x) \sim x^\gamma$). On the other hand, if $\log g_\xi(x) \gg x$, then complete localisation is known not to hold (see \cite{Gartner07}). What occurs in the interface regime of double-exponential tail decay ($\log g_\xi(x) \sim  c x$, for $c > 0$) is not currently well-understood.

As for the radius of influence of the potential field, $\rho_\text{PAM}$, in the case of Pareto-like tail decay it has been shown (see \cite{Konig09}) that $\rho_\text{PAM} = 0$, in other words, the localisation site can be determined by maximising a functional that depends on the potential field $\xi$ only through its value at individual lattice sites, with interactions between neighbouring lattice sites having no influence on localisation. On the other hand, in the case of Weibull-like tail decay ($g_\xi(x) \sim x^\gamma$), the radius of influence has been shown (see \cite{Fiodorov13}) to be $\rho_\text{PAM} = \left[ (\gamma - 1)/2 \right]^+$, where $[x]$ and $x^+$ denote the integer and positive parts of $x$ respectively. Clearly this implies that  $\rho_\text{PAM} = 0$ if and only if $\gamma < 3$, and also that $\rho_\text{PAM} \to \infty$ in the $\gamma \to \infty$ limit.

\subsubsection{Localisation in the Bouchaud trap model}
The study of localisation in the Bouchaud trap model has also received considerable attention over the last 10 years. A notable feature of the BTM is that localisation can only occur in dimension one. In higher dimensions, the traps either have negligible effect in the limit (if the tail is integrable, by virtue of the law of large numbers), or are visited in such a way that their overall effect is spatially-homogeneous (see \cite{FIN02} and \cite{BenArous06} for a proof of this result in the case of Pareto-like tail decay, although the result is thought to hold more generally for arbitrary non-integrable tail decay).

On the other hand, it is known that in dimension one, Pareto-like tail decay ($g_\sigma(x) \sim c \log x$, $c > 0$) forms the boundary of the localisation universality class. More precisely, if $\log x = O(g_\sigma(x))$, it is known that the BTM does not localise in the sense of equation \eqref{eq:local} (although it does localise in a certain weaker sense; see, e.g.\ \cite{FIN02}). On the other hand, it was proven in \cite{Muirhead14} that for sub-Pareto tail decay ($g_\sigma(x) \ll \log x)$, the BTM localises on exactly two-sites in the limit, with a radius of influence (i.e.\ of the trapping landscape) equal to $0$.

\subsection{Overview of our results}
\label{sec:over}
Before detailing our results in full, we first provide a brief overview to highlight salient features; this section is for exposition only, and is not intended to be mathematically rigorous. A complete description of our results follows in Section \ref{sec:full} below.

In this initial study of localisation in the BAM, we focus on the case where both potential distribution $\xi(0)$ and trap distribution $\sigma(0)$ have Weibull tail decay
\[   \PP(\xi(0) > x) = e^{-x^\gamma} \quad \text{and} \quad \PP(\sigma(0) > x) = e^{-x^\mu} \,  \qquad \gamma, \mu > 0 \, .\]
Our results also hold in the $\gamma, \mu \to 0$ limit (with some caveats; see Section \ref{sec:setup}). As we shall see, the BAM with Weibull tail decay turns out to be a natural regime to study, since the interaction between the potential field and trapping landscape exhibits certain phase transitions in $(\gamma, \mu)$.

\subsubsection{Complete localisation} Our first main result establishes the complete localisation of the BAM across the entire regime (see Theorem \ref{thm:main1} below).
\begin{theorem}
\label{thm:ex1}
There exists a (random) site $Z_t$ such that, as $t \to \infty$, 
\[ \frac{u(t, Z_t)}{U(t)} \to 1 \qquad \text{in probability} \, .\]
\end{theorem}
That the BAM completely localises for some $(\gamma, \mu)$ is expected, since the PAM with Weibull potential also exhibits complete localisation. More surprising, however, is that complete localisation occurs regardless of the presence of very large traps, even in dimension one, since \textit{a priori} it might be thought that large traps would draw probability mass away from the localisation site. 

\subsubsection{Mutual reinforcement of localisation effects due to the PAM and the BTM}
Since complete localisation holds in the entire regime, in order to probe the interaction between the potential field and trapping landscape we need a finer measure of localisation. Such a measure is provided by the \textit{radius of influence} $\rho$, which as described above is the smallest integer for which the localisation site $Z_t$ can be determined by maximising a functional on $\mathbb{Z}^d$ that depends on $\xi$ and $\sigma$ only through their values in balls of radius $\rho$ around each site. Our second main result is to determine the radius of influence $\rho$, and to prove its optimality (see Theorem \ref{thm:main1} and part \eqref{thm:main3a} of Theorem \ref{thm:main3} below).
\begin{theorem}
\label{thm:exp2}
The radius of influence is
\[ \rho := \left[ \frac{\gamma - 1}{2} \frac{\mu}{\mu + 1} + \frac12  \right]^+ \, . \]
\end{theorem}
Note that $\rho$ is a decreasing function of the strength of both the potential field and trapping landscape (i.e.\ an increasing function of $\gamma$ and $\mu$), in other words, the localisation effects due to the PAM and BTM are \emph{mutually reinforcing}.

\subsubsection{Reducibility of the BAM to the PAM}
We next ask whether the BAM is `reducible' to the PAM. There are actually two distinct notions of reducibility that are relevant. Strong reducibility describes the situation in which the trapping landscape $\sigma$ plays no role in determining the localisation site $Z_t$, and the macroscopic behaviour of the system is adequately approximated by the PAM with potential $\xi$. Weak reducibility describes the situation in which all necessary information to determine $Z_t$ is contained in the `net growth rate' $\eta := \xi - \invsigma$, and moreover, the macroscopic behaviour of the BTM is adequately approximated by the PAM with potential replaced with $\eta$. The term `net growth rate' comes from the interpretation of the BAM as a trapped, branching random walk (see Section \ref{sec:con} below). Our third main result is to determine the regimes in which the BAM is strongly and weakly reducibility to the PAM (see parts \eqref{thm:main3c} and \eqref{thm:main3d} of Theorem \ref{thm:main3} below). These regimes are depicted in Figure~\ref{fig3}.

\begin{theorem}
The BAM is strongly reducible to the PAM if and only if $\gamma < 1$. The BAM is weakly reducible to the PAM if and only if $\rho = 0$ and $\gamma\ge1$.
\end{theorem}

\begin{figure}[t]
\centering
\begin{tikzpicture}[scale = 0.91]
 \begin{axis}[xmin=0,xmax=5,ymin=-0,ymax=6]
  \addplot [thick, domain=2:5] (\x, {1/(\x-2)});
  \addplot[black, dashed, thick] coordinates{(1, 0) (1, 6)};  
 \end{axis}
\end{tikzpicture}
\caption{Partition of the parameter space of the BAM according to the whether the BAM is `strongly reducible' to the PAM with the usual potential $\xi$ (left of the dashed line) or `weakly reducible' to the PAM with the potential replaced with the `net growth rate' $\eta$ (left of the bold curve). The boundary curve is $\mu = 1/(\gamma - 2)$.}
\label{fig3}
\end{figure}
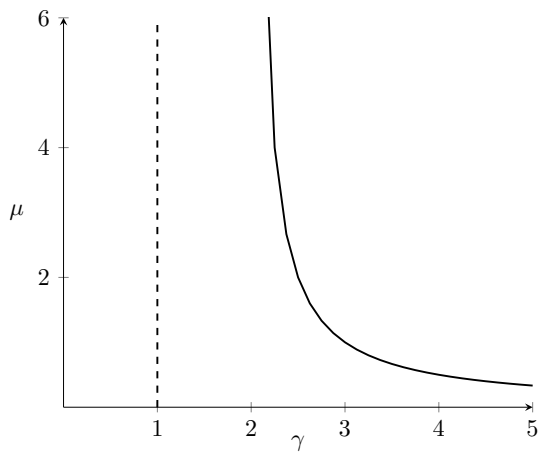

\subsubsection{Local correlation between the potential field and trapping landscape: The `fit and stable' hypothesis}
Our final result is to establish the \textit{local correlation} between the potential field and trapping landscape (where `local' is from the perspective of the localisation site); this is the so-called `fit and stable' hypothesis that has been predicted numerically in the mathematical biology literature (see, e.g., \cite{Brotto14}), but never rigorously confirmed (see Section \ref{sec:con} below). Interestingly, the correlation that we observe is \textit{positive} at the localisation site, but \textit{negative} away from the localisation site, providing an unexpected extension to the `fit and stable' hypothesis.

To describe this correlation, we shall need to define a second, possibly smaller, radius of influence
\[\rho_\xi := \left[ \frac{ \gamma - 1 }{2} \frac{\mu} {\mu + 1}\right]^+ \in \{ \rho - 1, \rho\}  \, ,\]
which is the the smallest integer for which the localisation site $Z_t$ can be determined by a maximising a functional on $\mathbb{Z}^d$ that depends on $\xi$ only through its values in balls of radius $\rho_\xi$ around each site (note, the functional must still depend on $\sigma$ through balls of radius at least $\rho$). For simplicity, we exclude here the `interface cases', i.e.\ the points of discontinuity of $\rho_\xi$.

\begin{theorem}
Assume that $\gamma \ge 1$, so that the BAM is not strongly reducible to the PAM. Let $Z_t$ denote the site of complete localisation. Then, as $t \to \infty$ eventually almost surely: (i) the random variables $\xi(Z_t)$ and $\sigma(Z_t)$ are positively correlated; and (ii) for all $z$ such that $0 < |z - Z_t| \le \rho_\xi$, the random variables $\xi(z)$ and $\sigma(z)$ are negatively correlated.
\end{theorem}
In Theorem \ref{thm:main2} below we make explicit the nature of this correlation, as well as providing a full description of the localisation site, determining its asymptotic distance from the origin, the local profile of the potential field and trapping landscape, and its ageing behaviour.

\subsection{Methods and techniques}
Our approach to proving localisation in the BAM is loosely based on existing techniques to prove localisation in the PAM (see, e.g., \cite{Astrauskas07, Gartner07, Sidorova12}), although the complex interaction between the potential field and the trapping landscape means that these techniques cannot be trivially adapted. Instead, the presence of the trapping landscape requires the development of existing techniques on two main fronts. 

First, proving localisation in the BAM requires the development of the spectral theory of operators of the form $\Delta \invsigma + \xi$, including path expansions and Feynman-Kac representations for the principal eigenvalue and eigenfunction respectively. To our knowledge this theory has not appeared in the literature before, and may be of independent interest, including in the study of position-dependent mass Schr\"{o}dinger operators (see Section \ref{sec:con} below). In the particular case of the BAM with Weibull tails, we also extend existing techniques to establish the max-class of local eigenvalues; this is necessary in order to extract extra information about the local correlation in the potential field and trapping landscape.

Second, in order to analyse the `screening effect' of heavy traps, standard percolation estimates are insufficient: in dimension one, because of the geometry; in dimensions higher than one, because of complex dependencies between the potential field, the trapping landscape, and the localisation site $Z_t$. In dimension one we analyse heavy traps using coarse graining methods; in higher dimensions, we implement new ideas that allow us to apply percolation estimates in the presence of the dependencies.

In addition, our methods provide a new approach to working with `cluster expansions'. Although these expansions have appeared in the literature before (see, e.g., \cite{Astrauskas07, Fiodorov13}), the standard approach has been to access them via resolvent formalism. Our techniques provides a purely probabilistic approach to `cluster expansion', which avoids many of the technicalities of the resolvent formalism. One application would be a simpler, purely probabilistic proof of the localisation results on the PAM found in \cite{Fiodorov13}.

\subsection{Connections to the literature}
\label{sec:con}
Although this is the first work to consider the BAM, there are clear connections between the BAM and other models in the literature. First, the BAM can be interpreted as the thermodynamic limit of a particle system with random branching and trapping mechanisms (given, respectively, by the potential field $\xi$ and the trapping landscape $\sigma$). In the probability literature there have been several other analyses of models combining random branching and trapping mechanisms -- in particular, trapping mechanisms given by asymmetric transition probabilities \cite{Gantert14} and random conductances \cite{Wolff13} -- although these have not considered the localisation properties of the model, focusing instead on the growth of the total population.

Similar models have also appeared in the mathematical biological literature, where they find an application in the study of population dynamics. Here the branching and trapping rates are recast as the \textit{fitness (`adaptedness')} and \textit{stability (`adaptability')} respectively of individual states (e.g.\ geographic locations, genetic configurations etc.). While the literature contains several models which allow for randomness in either the fitness \cite{Kingman78, Park10} or stability \cite{Ishii89, Leigh70, Taddei97}, most relevant is~\cite{Brotto14} which considers a model in which \textit{both} these characteristics vary. Indeed, the model considered in \cite{Brotto14} is essentially identical to the BAM, except it is defined in a domain without any geometry: when an individual's state changes, the fitness and stability are re-sampled according to their respective distributions.\footnote{A second minor difference is that the population size is kept constant by the deletion of a uniformly chosen individual at each replication event.} The primary observation in \cite{Brotto14} (obtained numerically) is the tendency of populations to concentrate on states which are both fit and stable: the `fit and stable' hypothesis. Our results provide the first rigorous analysis of this phenomenon. Indeed, our results actually suggest a refinement of the hypothesis (for our model with geometry): that populations concentrate on states which are fit and stable, \textit{but also} for which neighbouring states are both fit and \textit{unstable}.

Second, operators of the form $\Delta \invsigma + \xi$ have important applications in quantum mechanics, since their eigenvalues give the energy levels of a particle whose effective mass is position-dependent (see, e.g., \cite{Chen04, Eleuch12, vonRoos83}). To make the connection, consider the position-dependent mass Schr\"{o}dinger equation for a particle with effective mass $\sigma$ in a potential field $\xi$. This equation has a Hamiltonian of general form (see \cite{vonRoos83})
\[ \frac{1}{2} \left( \sigma^{-\alpha} \nabla \sigma^{-\beta} \nabla \sigma^{-\gamma} +  \sigma^{-\gamma} \nabla \sigma^{-\beta} \nabla \sigma^{-\alpha} \right)  + \xi \, , \qquad \alpha, \beta, \gamma \ge 0, \ \alpha + \beta + \gamma = 1 \, .\]
Although there is no canonical choice for $\alpha, \beta, \gamma$, in the discrete setting a natural restriction is $\beta = 0$, which avoids symmetry breaking in the definition of $\nabla$. Specialising to the case $\alpha = \gamma = 1/2$ gives the Hamiltonian 
\begin{align}
\label{eq:sh}
\sigma^{-\frac{1}{2}} \Delta \sigma^{-\frac{1}{2}} + \xi  = \sigma^{-\frac{1}{2}} \left(  \Delta \sigma^{-1} + \xi \right) \sigma^{\frac{1}{2}} \, .
\end{align}
We remark that the Hamiltonian in \eqref{eq:sh} is the `symmetrised' form of the operator $\Delta \sigma^{-1} + \xi$, and hence has equivalent spectral theory. In Section \ref{sec:gentheory} we develop general theory for operators of the form $\Delta \invsigma + \xi$, including deriving path expansions and Feynman-Kac representations for the principal eigenvalue and eigenfunction respectively. This section is entirely self-contained, and is completely deterministic, and we expect that it will be of independent interest.

Third, there are connections between the BAM and the PAM in the case where the potential field distribution $\xi(0)$ is allowed to take on highly negative (or even infinitely negative) values, which may be interpreted as `traps'. Previous work has noted the minimal influence of such `traps' in $d \ge 2$ (see, e.g.\ \cite[Section 2.4]{Gartner90}), essentially due to percolation estimates, an observation that finds echoes in our results and methods. However, there are clear differences between this model and the BAM, primarily due to the fact that the traps in the BAM may coexist with sites of high potential; this coexistence underlies the phenomena of mutual reinforcement and correlation that we observe in the BAM. On the other hand, in dimension one the effect of highly negative potential values in the PAM is significant (see~\cite{Biskup01}). Indeed, since such sites cannot be avoided, their effect is to `screen' off the growth that would otherwise occur from sites of high potential, and so the asymptotic growth of the solution depends heavily on the relationship between the upper and lower tails of $\xi(0)$. Again, this is reminiscent of our results in dimension one, which are only valid if the trap distribution decays sufficiently fast to ensure `screening' effects are negligible.

\subsection{The formal set-up for the paper}
\label{sec:setup}
For the rest of the paper, we make the following assumptions on the potential field $\xi$ and the trapping landscape $\sigma$: 
\begin{assumption}[Assumption on the potential field distribution]
\label{assump:xi}
$ $ \\
The random variable $\xi(0)$ is strictly-positive and satisfies
\[\bar F_\xi(x) = e^{-x^\gamma} \, ,\]
for some $\gamma > 0$, where $\bar F_\xi(x) := 1-F_\xi(x):=\PP(\xi(0) > x)$.
\end{assumption}

\begin{assumption}[Assumptions on the trap distribution]
\label{assump:sigma}
$ $\\
The random variable $\sigma(0)$ satisfies:
\begin{enumerate}[(a)]
\item \textbf{No quick sites}: The quantity
\[ \delta_\sigma := \rm{essinf} \, \sigma(0) \]
is strictly positive;
\item \textbf{Regularity}: The quantity
\[ \mu :=  \lim_{x \to \infty} \, \frac{\log g_\sigma(x)}{\log x} \]
exists and is finite. 
If $\mu > 0$, then $\sigma(0)$ has a continuous density function $f_\sigma(x)$ with a Weibull upper-tail, i.e.\ for sufficiently large $x$,
\[ \bar F_\sigma(x) = \exp \{ -x^\mu\}  \, ,\]
where $\bar F_\sigma(x) := 1-F_\sigma(x) := \PP(\sigma(0) > x)$.
If $\mu = 0$, then $\sigma(0)$ has a continuous density function $f_\sigma(x)$, with the property that
\[ f_\sigma(a_x) \sim f_\sigma(b_x)  \]
for any $a_x, b_x \to \infty$ such that $a_x \sim b_x$ (see Section \ref{subsec:notation} for the asymptotic notation). In both cases, the lower-tail of $f_\sigma$ satisfies, as $x \to 0$, 
\[ f_\sigma(x + \delta_\sigma) = o(e^{-1/x})  \, . \]
\end{enumerate}
Furthermore, if $d=1$, then additionally $\sigma(0)$ satisfies the following two extra conditions:
\begin{enumerate}[(a)]
\setcounter{enumi}{2}
\item \textbf{Sufficiently fast tail decay}: As $x \to \infty$ eventually, for some $\varepsilon > 0$,
\begin{align*}
g_\sigma(x) > ( 1 + \varepsilon) \log \log x  \, ;\ 
\end{align*} 
\item \textbf{Regularity}: There exists a $c \in (1, \infty]$ such that
\[ \lim_{x \to \infty} \frac{ g_\sigma(x)}{\log \log x} = c  \, ,\  \]
with the convergence eventually monotone in the case $c = \infty$.
\end{enumerate}
\end{assumption}

We wish to briefly comment on the nature of the above assumptions on $\xi(0)$ and $\sigma(0)$. First, we claim that the BAM with Weibull potential field is a natural regime in which to observe the interaction between the localisation effects in the PAM and the BTM. If the potential field is any stronger (indeed if $\gamma < 1$), the BAM is strongly reducible to the PAM\footnote{Note however that, because of Assumption \ref{assump:sigma}, this conclusion does not apply in dimension one if the trapping landscape is sufficiently strong.}. On the other hand, if the potential field is any weaker, the effect of the trapping landscape, while present, is harder to measure. To see why, recall that the PAM with Weibull potential field has been shown to completely localise with a certain finite radius of influence $\rho_\text{PAM}$; it is on the level of this radius that we measure the impact of the trapping landscape $\sigma$. Since $\rho_\text{PAM} \to \infty$ in the $\gamma \to \infty$ limit, the effect of changes to $\rho_\text{PAM}$ become harder to quantify, and we leave this study to future work.

Second, the regularity assumption on $\xi(0)$ is imposed mainly for simplicity; weaker regularity assumptions (like those found in \cite{Astrauskas07} and \cite{Astrauskas08} for instance) are possible, although they introduce certain technical difficulties that we wish to avoid. Finally, note that equivalent results for the BAM with Pareto-like potential field can be naturally deduced by considering our results in the $\gamma \to 0$ limit.

Turning to the assumptions on $\sigma(0)$, first note that the quantity $\mu$ measures the `Weibull-ness' of the upper-tail of $\sigma(0)$, with the case $\mu = 0$ corresponding to a stronger-than-Weibull trapping landscape. For simplicity, we have chosen not to consider weaker-than-Weibull trapping landscapes in this paper; equivalent results can be naturally deduced by considering our results in the $\mu \to \infty$ limit. As with $\xi(0)$, the regularity assumptions on $\sigma(0)$ are certainly not optimal for our results to hold; they are chosen mainly for simplicity. On the other hand, our assumption that $\sigma(0)$ is bounded away from zero is essential. Indeed we expect that the nature of the localisation behaviour will change if `quick' sites are present. Finally, the additional tail decay assumption in dimension one is also essential, and our results and methods break down completely without it. Note, however, that this condition is only violated for trap distributions with extremely heavy tails, such as if $\sigma(0)$ is a \emph{log-Pareto} random variable.

\subsection{Full description of our results}
\label{sec:full}

Here we describe our results in full, expanding on the exposition given in Section \ref{sec:over}. In order to state our results explicitly, we shall need to introduce some notation. Recall the parameter $\mu \in [0, \infty)$ from Assumption \ref{assump:sigma}, which describes the `Weibull-like' decay parameter of the upper-tail of $\sigma(0)$. Recall also from Section \ref{sec:over} the radius of influence
\[ \rho := \left[ \frac{\gamma - 1}{2} \frac{\mu}{\mu + 1} + \frac12  \right]^+ \]
and the, possibly smaller, radius of influence of the potential field $\xi$,
\[  \rho_\xi := \left[ \frac{ \gamma - 1 }{2} \frac{\mu} {\mu + 1}\right]^+ \in \{ \rho - 1, \rho\}  \le \rho \, .\]
The relationship between $\rho$ and $\rho_\xi$ is depicted in  Figure~\ref{fig1}; we defer further discussion on $\rho$ and $\rho_\xi$ to Remark \ref{rem:radius}.

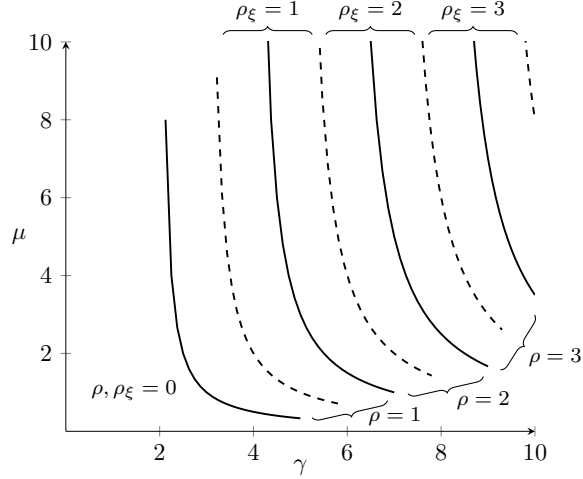
\begin{figure}[t]
\centering
\begin{tikzpicture}[scale = 0.91]
 \begin{axis}[xmin=0,xmax=10,ymin=-0,ymax=10]
 \addplot [thick,domain=2:5] (\x, {1/(\x-2)});
 \addplot [thick,domain=4:7] (\x, {3/(\x-4)});
 \addplot [thick,domain=6:9] (\x, {5/(\x-6)});
 \addplot [thick,domain=8:10] (\x, {7/(\x-8)});
 \addplot [black, dashed, thick,domain=3.22:5.8] (\x, {2/(\x-3)});
 \addplot [black, dashed, thick, domain=5:7.8] (\x, {4/(\x-5)});
 \addplot [black, dashed, thick,domain=7:9.3] (\x, {6/(\x-7)});
 \addplot [black, dashed, thick,domain=9:10] (\x, {8/(\x-9)});
 \end{axis}
 \node at (1,0.6){\footnotesize $\rho, \rho_\xi = 0$};
  \draw [decorate,decoration={brace,amplitude=3pt}]   (4.7,0.45) -- (3.6,0.2) node [black,midway,xshift=0.6cm, yshift=-0.1cm] {\footnotesize $\rho = 1$};
  \draw [decorate,decoration={brace,amplitude=3pt}]   (6.1,0.8) -- (5,0.55) node [black,midway,xshift=0.5cm, yshift=-0.2cm]  {\footnotesize $\rho = 2$};
 \draw [decorate,decoration={brace,amplitude=3pt}]   (6.85,1.7) -- (6.35,0.9) node [black,midway,xshift=0.5cm, yshift=-0.2cm]  {\footnotesize $\rho = 3$};
 \draw [decorate,decoration={brace,amplitude=3pt}] (2.3,5.8) -- (3.6,5.8)  node [black,midway,yshift=0.3cm] {\footnotesize $\rho_\xi = 1$};
  \draw [decorate,decoration={brace,amplitude=3pt}] (3.8,5.8) -- (5.1,5.8)  node [black,midway,yshift=0.3cm] {\footnotesize $\rho_\xi = 2$};
  \draw [decorate,decoration={brace,amplitude=3pt}] (5.3,5.8) -- (6.6,5.8)  node [black,midway,yshift=0.3cm] {\footnotesize $\rho_\xi = 3$};
\end{tikzpicture}
\caption{Partition of the parameter space of the BAM according to the values of $\rho$ (bold lines) and $\rho_\xi$ (dashed lines). The boundary curves are of the form $\mu = (2i-1)/(\gamma - 2i)$ and $\mu = (2i)/(\gamma - 2i - 1)$, for $i \in \mathbb{N} \setminus \{0\}$.}
\label{fig1}
\end{figure}

Next we describe explicitly the localisation site. For each $z \in \mathbb{Z}^d$ and $n \in \mathbb{N}$, define the ball $B(z, n) := \{y \in \ZZ^d : |y-z| \le n\}$. For each $z \in \mathbb{Z}^d$, define the Hamiltonian
\begin{align*}
\mathcal{H}(z) :=  \Delta \invsigma + \xi  \id_{B(z, \rho_\xi)}  
\end{align*}
restricted to the domain $B(z, \rho)$ with Dirichlet boundary conditions, denoting by $\lambda(z)$ its principal eigenvalue. Note that each $\lambda(z)$ is real since the Hamiltonian $\mathcal{H}(z)$ is similar to the Hermitian operator
$$ \sigma^{-\frac{1}{2}} \, \mathcal{H}(z) \, \sigma^{\frac{1}{2}} = \sigma^{-\frac{1}{2}}  \Delta \sigma^{-\frac{1}{2}} + \xi \id_{B(z, \rho_\xi)}  \, . $$

We refer to $\lambda(z)$ as the \textit{local principal eigenvalue at $z$}, and remark that it is a certain functional of the sets $\xi^{(\rho_\xi)}(z) := \{\xi(y)\}_{y \in B(z, \rho_\xi)}$ and $\sigma^{(\rho)}(z) := \{ \sigma(y)\}_{y \in B(z, \rho)}$. Note that the random variables $\{\lambda(z)\}_{z \in \mathbb{Z}^d}$ are identically distributed, and have a dependency range bounded by $2\rho$, i.e.\ the random variables $\lambda(y)$ and $\lambda(z)$ are independent if and only if $|y-z| > 2\rho$. Remark also that in the special case $\rho = 0$, $\lambda(z)$ reduces to the `net growth rate' $\eta(z) = \xi(z) - \invsigma(z)$.

For any sufficiently large $t$, define a \textit{penalisation functional} $\Psi_{t}: \mathbb{Z}^d \to \mathbb{R}$ by
$$ \Psi_{t}(z) := \lambda(z) - \frac{|z|}{\gamma t} \log \log t\,. $$
Note that $\Psi_t$ has a similar form to the penalisation functional introduced in \cite{Fiodorov13} to prove complete localisation in the PAM with Weibull potential field, representing the trade-off between energetic forces (given by the local principal eigenvalue $\lambda(z)$) and entropic forces (given by a probabilistic penalty which is linear in $|z|$ and decaying in $t$); see Remark \ref{remark:intuition}.

Define a large `macrobox' $V_t := [-R_t, R_t]^d \cap \mathbb{Z}^d$, with $R_t := t (\log t)^{\frac{1}{\gamma}}$. Fix a constant $0 < \theta <  1/2$ and define the macrobox level $L_{t} := ((1-\theta) \log |V_t|)^{\frac{1}{\gamma}}$. Let the subset $\Pi^{(L_{t})} := \left\{ z \in \mathbb{Z}^d : \xi(z) > L_{t} \right\} \cap V_t$ consist of sites in $V_t$ at which $\xi$-exceedances of the level $L_{t}$ occur. Finally, define the random site 
\[ Z_{t} := \argmax_{z \in \Pi^{(L_t)}} \Psi_{t}(z) \, .\]
The site $Z_t$ is well-defined eventually almost surely since, as we show in Lemma \ref{lem:asforxi}, the set $\Pi^{(L_t)}$ is non-empty and finite eventually almost surely. Moreover, for $t$ sufficiently large, $Z_t$ almost surely does not depend on the particular choice of $\theta$. We present again (see Theorem~\ref{thm:ex1}) our main theorem:

\begin{theorem}[Complete localisation]
\label{thm:main1}
As $t \to \infty$,
$$ \frac{u(t, Z_t)}{U(t)} \to 1 \qquad \text{in probability} \, .$$
\end{theorem}

\begin{remark}
\label{remark:intuition}
In order to determine $Z_t$ explicitly, a finite approximation is available for $\lambda(z)$ (see Proposition~\ref{prop:pathexp} for a precise formulation):
\begin{align}
\label{eq:pathexp}
\lambda(z) \approx  \eta(z) + \invsigma(z) \! \sum_{2 \leq k \leq 2j} \, \sum_{\substack{p \in \Gamma_{k}(z, z) \\ p_i \neq z, \, 0 < i < k \\ \mathrm{Set}(p) \subseteq B(z, \rho)}} \ \prod_{0 < i < k} (2d)^{-1} \frac{\invsigma(p_i)}{{\lambda}(z) - \eta_z(p_i)} \, ,
\end{align}
where $j := [\gamma - 1]$ and $\eta_z := \xi \id_{B(z, \rho_\xi)} - \invsigma$; see Section \ref{subsec:notation} for the definition of the path set $\Gamma_k(z, z)$. This path expansion can be iteratively evaluated to approximate $\Psi_t(z)$ as an explicit function of $\xi^{(\rho_\xi)}(z)$, $\sigma^{(\rho)}(z)$, $|z|$ and $t$, which, as we show, is sufficiently precise to determine the localisation site $Z_t$ with overwhelming probability.
\end{remark}

Before stating our second and third main results we shall introduce some more notation. First we define exponents that explicitly describe the correlation of the fields $\xi$ and $\sigma$ around the localisation site $Z_t$. To this end, define the function $q_\xi: \mathbb{N} \to [0, 1]$ and the non-negative constant~$q_\sigma$ by
\[ q_\xi(x) := \begin{cases} \left( 1 -  \frac{2x}{\gamma - 1} - \frac{1}{\mu + 1} \right)^+  & \text{if } \gamma > 1 \, , \\
 (1 - x)^+ & \text{else,} \end{cases} \qquad \text{and} \qquad  q_\sigma :=  \left (\frac{\gamma - 1}{\mu + 1} \right)^+   \, .\] 
We shall also need the concept of `interface cases', which correspond to the values of $(\gamma, \mu)$ where~$\rho$, and respectively $\rho_\xi$, are transitioning from one integer to the next. To this end define the sets 
\[  \mathcal{B} := \left\{ (\gamma, \mu) : \frac{\gamma - 1}{2} \frac{\mu}{\mu + 1} + \frac{1}{2}  = \rho  \right\} \quad \text{and} \quad \mathcal{B}_\xi :=  \left\{ (\gamma, \mu) :  \frac{\gamma - 1}{2} \frac{\mu}{\mu + 1} = \rho_\xi \right\}   \, . \]
Note that these sets correspond, respectively, to the bold and dashed curves in Figure~\ref{fig1}. Finally, define the random time $T_t := \inf\{ s > 0 : Z_{t+s} \neq Z_t \}$ and the scales
\begin{equation}
\label{eq:rtat}
r_t := \frac{t (d \log t)^{\frac{1}{\gamma} - 1}}{\log \log t}  \quad \text{and} \quad a_t := (d \log t)^\frac{1}{\gamma}  \, .
\end{equation}
The scales $r_t$ and $a_t$ describe, respectively, the scale of the distance from the origin of the localisation site and the scale of the height of the potential field at the localisation site.

\begin{theorem}[Description of the localisation site] 
\label{thm:main2}
As $t \to \infty$ the following hold:
\begin{enumerate}[(a)]
\item (Localisation distance)
\label{thm:main2a}
$$ \frac{Z_t}{r_t} \Rightarrow X \qquad \text{in law}\,, $$
where $X$ is a random vector whose coordinates are independent and distributed as Laplace (two-sided exponential) random variables with absolute-moment one.

\item (Local correlation of the potential field)
\label{thm:main2b}
If $(\gamma, \mu) \notin \mathcal{B}_\xi$, then for each $z \in B(0, \rho_\xi)$ there exists a $c > 0$ such that
\begin{align}
\label{eq:xiprofile}
 \frac{\xi(Z_t + z)}{a_t^{q_\xi(|z|)}} \to c \qquad \text{in probability} \, .
 \end{align}
If $(\gamma, \mu) \in \mathcal{B}_\xi$, then \eqref{eq:xiprofile} holds for each $z \in B(0, \rho_\xi -1)$, and moreover, for each $z $ such that $|z| = \rho_\xi$ there exists a $c > 0$ such that,
\[  f_{\xi(Z_t + z)}(x) \to \frac{ e^{c x} f_{\xi}(x)}{ \EE[e^{c \xi(0)} ] } \, , \]
uniformly over $x \in (0, L_t)$, where $f_{\xi(z)}$ is the density of the potential field at site $z$ (see Assumption \ref{assump:xi}).
\item (Correlation of the trapping landscape at $Z_t$)
\label{thm:main2c}
If $\mu > 0$ and $\gamma > 1$, then there exists a $c > 0$ such that
$$ \frac{ \sigma(Z_t)}{a_t^{q_\sigma}} \to  c \qquad \text{in probability} \, .$$
If $\mu = 0$ and $\gamma > 1$ then, for each $\nu > 0$,
$$ \PP \left(  \frac{\log \sigma(Z_t)}{ \log a_t} > q_\sigma - \nu \right) \to  1 \, .$$
If $\gamma = 1$ then,
\[ f_{\sigma(Z_t)}(x)  \to \frac{ e^{-1/x} f_\sigma(x) }{\EE[ e^{-1/\sigma(0)} ] } \, ,\]
uniformly over $x$, where $f_{\sigma(Z_t)}$ is the density of the trapping landscape at site $Z_t$.
\item (Local correlation of the trapping landscape)
\label{thm:main2d}
If $(\gamma, \mu) \notin \mathcal{B}$, then for each $z \in B(0, \rho) \setminus  \{0\}$ 
\begin{align}
\label{eq:sigmaprofile}
\sigma(Z_t + z) \to \delta_\sigma \qquad \text{in probability} \, .
\end{align}
If $(\gamma, \mu) \in \mathcal{B}$, then \eqref{eq:sigmaprofile} holds for each $z \in B(0, \rho) \setminus  \{0\}$ and moreover, for each $z$ such that $|z| = \rho$, there exists a $c > 0$ such that
\[ f_{\sigma(Z_t + z)}(x) \to \frac{ e^{c /x} f_\sigma(x)}{\EE[e^{c /\sigma(0)}  ]} \, ,  \]
uniformly over $x$, where $f_{\sigma(z)}$ is the density of the trapping landscape at site $z$ (see Assumption~\ref{assump:sigma}).
\item (Ageing)
\label{thm:main2e}
$$\frac{T_t}{t} \Rightarrow \Theta \qquad \text{in law}\,,$$
where $\Theta$ is a non-degenerate almost surely positive random variable.
\end{enumerate}
\end{theorem}

\begin{theorem}[Optimality results]
\label{thm:main3}
As $t \to \infty$ the following hold:
\begin{enumerate}[(a)]
\item (Optimality of the radius of influence)
\label{thm:main3a}
The radius of influence $\rho$ is optimal, in other words, there does not exist a functional $\psi_t$, depending on $\xi$ and $\sigma$ only through their values in balls of radius $\rho - 1$ around each site $z$, such that
\begin{align}
\label{eq:opt1}
 \PP \left( Z_t  = \argmax_{z \in \mathbb{Z}^d} \psi_t(z) \right) \to 1 \, .
 \end{align}
\item (Optimality of the radius of influence with respect to  the potential field)
\label{thm:main3b}
The radius of influence of the potential field $\rho_\xi$ is optimal, in other words, there does not exist a functional $\psi_t$, depending on $\xi$ only through its values in balls of radius $\rho_\xi - 1$ around each site $z$, such that
\begin{align}
\label{eq:opt2}
 \PP \left( Z_t  = \argmax_{z \in \mathbb{Z}^d} \psi_t(z) \right) \to 1 \, .
 \end{align}
\item (Criterion for reduction to the potential $\xi$)
\label{thm:main3c}
The localisation site is independent of the trapping landscape $\sigma$ if and only if $\gamma < 1$, in other words, if and only if $\gamma < 1$, there exists a random site $z_t \in \ZZ^d$, independent of $\sigma$, such that,
\begin{align}
\label{eq:opt3}
 \PP \left( Z_t  = z_t \right) \to 1 \, .
 \end{align}
\item (Criterion for reduction to the `net growth rate' $\eta$) 
\label{thm:main3d}
The localisation site $Z_t$ depends on $\xi$ and $\sigma$ only through the value of $\eta$ if and only if $\rho = 0$, in other words, if and only if $\rho = 0$, there exists a random site $z_t \in \ZZ^d$, dependent on $\xi$ and $\sigma$ only through $\eta$, such that,
\begin{align}
\label{eq:opt4}
 \PP \left( Z_t  = z_t \right) \to 1 \, .
 \end{align}
\end{enumerate}
\end{theorem}

\begin{remark}
\label{rem:radius}
We note several interesting features of the radius of influence $\rho$. As remarked above, $\rho$ is an increasing function of both $\gamma$ and $\mu$. Moreover, surprisingly it is not necessarily the case that $\rho \to \rho_\text{PAM} := [(\gamma-1)/2]^+$ in the $\mu \to \infty$ limit; indeed, if $\gamma \in (2i, 2i+1)$ for $i \in \mathbb{N} \setminus \{0\}$, then in fact $\rho \to \rho_\text{PAM} + 1$, meaning that influence of the BTM on the BAM is not continuous in the degenerate limit (i.e.\ as $\sigma(z) \to 1$ simultaneously for each $z$). On the other hand, $\rho_\xi \to \rho_\text{PAM}$ in the $\mu \to \infty$ limit, i.e.\ there is no discontinuity in the effect of the BTM on the BAM on the level of the radius of influence of the potential field $\xi$. The relationship between $\rho$, $\rho_\xi$ and $\rho_\text{PAM}$ is depicted in Figure~\ref{fig2}.
\end{remark}

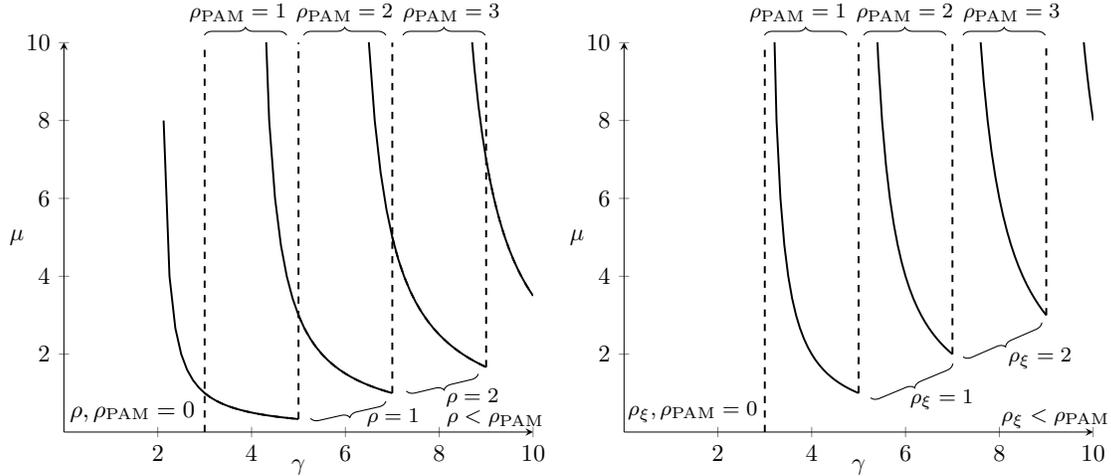
\begin{figure}[t!]\centering
\begin{tikzpicture}[scale = 0.91]
 \begin{axis}[xmin=0,xmax=10,ymin=-0,ymax=10]
 \addplot [thick,domain=2:5] (\x, {1/(\x-2)});
 \addplot [thick,domain=4:7] (\x, {3/(\x-4)});
 \addplot [thick,domain=6:9] (\x, {5/(\x-6)});
 \addplot [thick,domain=8:10] (\x, {7/(\x-8)});
 \addplot[black, dashed, thick] coordinates{(3, 0) (3, 10)};     
 \addplot [dashed,thick,domain=3:5] (\x, {1/(\x-2)});
 \addplot[black, dashed, thick] coordinates{(5, 1/3) (5, 10)};  
 \addplot [dashed,thick,domain=5:7] (\x, {3/(\x-4)});
 \addplot[black, dashed, thick] coordinates{(7, 3/3) (7, 10)}; 
 \addplot [dashed,thick,domain=7:9] (\x, {5/(\x-6)});
 \addplot[black, dashed, thick] coordinates{(9, 5/3) (9, 10)}; 
 \addplot [dashed,thick,domain=9:10] (\x, {7/(\x-8)});
 \end{axis}
 \node at (1,0.3){\small $\rho, \rho_\text{PAM} = 0$};
  \draw [decorate,decoration={brace,amplitude=3pt}]   (4.7,0.45) -- (3.6,0.2) node [black,midway,xshift=0.6cm, yshift=-0.1cm] {\footnotesize $\rho = 1$};
  \draw [decorate,decoration={brace,amplitude=3pt}]   (6.05,0.85) -- (5,0.6) node [black,midway,xshift=0.4cm, yshift=-0.2cm]  {\footnotesize $\rho = 2$};
 \node at (6.3,0.2) {\footnotesize $\rho < \rho_\text{PAM}$};
 \draw [decorate,decoration={brace,amplitude=3pt}] (2.05,5.8) -- (3.25,5.8)  node [black,midway,xshift=-0.1cm,yshift=0.3cm] {\footnotesize $\rho_\text{PAM} = 1$};
 \draw [decorate,decoration={brace,amplitude=3pt}] (3.5,5.8) -- (4.7,5.8)  node [black,midway,yshift=0.3cm]  {\footnotesize $\rho_\text{PAM} = 2$};
 \draw [decorate,decoration={brace,amplitude=3pt}] (4.95,5.8) -- (6.15,5.8)  node [black,midway,xshift=0.1cm,yshift=0.3cm] {\footnotesize $\rho_\text{PAM} = 3$};
\end{tikzpicture}
\begin{tikzpicture}[scale = 0.91]
 \begin{axis}[xmin=0,xmax=10,ymin=-0,ymax=10]
 \addplot [thick,domain=3:5] (\x, {2/(\x-3)});
 \addplot[black, dashed, thick] coordinates{(3, 0) (3, 10)};
 \addplot [thick, domain=5:7] (\x, {4/(\x-5)});
 \addplot[black, dashed, thick] coordinates{(5, 1) (5, 10)};
 \addplot [thick,domain=7:9] (\x, {6/(\x-7)});
 \addplot[black, dashed, thick] coordinates{(7, 2) (7, 10)};
 \addplot [thick,domain=9:10] (\x, {8/(\x-9)});
 \addplot[black, dashed, thick] coordinates{(9, 3) (9, 10)};
 \end{axis}
 \node at (1, 0.3){\footnotesize $\rho_\xi, \rho_\text{PAM} = 0$};
 \draw [decorate,decoration={brace,amplitude=3pt}]   (4.8,1.0) -- (3.6,0.5) node [black,midway,xshift=0.4cm, yshift=-0.25cm]  {\footnotesize $\rho_\xi = 1$};
\draw [decorate,decoration={brace,amplitude=3pt}]   (6.1,1.6) -- (4.95,1.1) node [black,midway,xshift=0.5cm, yshift=-0.25cm] {\footnotesize $\rho_\xi = 2$};
 \draw [decorate,decoration={brace,amplitude=3pt}] (2.05,5.8) -- (3.25,5.8)  node [black,midway,xshift=-0.1cm,yshift=0.3cm] {\footnotesize $\rho_\text{PAM} = 1$};
 \draw [decorate,decoration={brace,amplitude=3pt}] (3.5,5.8) -- (4.7,5.8)  node [black,midway,yshift=0.3cm]  {\footnotesize $\rho_\text{PAM} = 2$};
 \draw [decorate,decoration={brace,amplitude=3pt}] (4.95,5.8) -- (6.15,5.8)  node [black,midway,xshift=0.1cm,yshift=0.3cm] {\footnotesize $\rho_\text{PAM} = 3$};
 \node at (6.3,0.2) {\footnotesize $\rho_\xi < \rho_\text{PAM}$};
 \end{tikzpicture}
\caption{Partition of the parameter space of the BAM according to the relationship between $\rho$ (bold lines) and $\rho_\text{PAM}$ (dashed lines), and $\rho_\xi$ (bold lines) and $\rho_\text{PAM}$ (dashed lines) respectively, where $\rho_\text{PAM}$ denotes the radius of influence in the equivalent PAM with identical potential field. The boundary curves are of the form $\mu = (2i-1)/(\gamma - 2i)$ and $\mu = (2i)/(\gamma - 2i - 1)$ respectively, for $i \in \mathbb{N} \setminus \{0\}$.}
\label{fig2}
\end{figure}

\begin{remark}
\label{rem:localprofile}
The shape of the local profile of the potential field and trapping landscape in parts \eqref{thm:main2b}--\eqref{thm:main2d} of Theorem~\ref{thm:main2} is derived by considering the path expansion in equation \eqref{eq:pathexp} and determining the values of $\xi$ and $\sigma$ that appropriately balance: (i) the increase in $\lambda$ gained from favourable realisations of $\xi$ and $\sigma$; and (ii) the probabilistic penalty that results from such favourable realisations of $\xi$ and $\sigma$ if they are too unlikely. This balance is expressed through a convex function whose integral is asymptotically concentrated in the regions specified in Theorem~\ref{thm:main2}. This computation is carried out in the proof of Proposition~\ref{prop:corr}, identifying the constants in Theorem~\ref{thm:main2} explicitly.

We also give some heuristic reasons why we must distinguish the cases $(\gamma ,\mu) \in \mathcal{B}, \mathcal{B}_\xi$ in the correlation results. If $(\gamma ,\mu) \notin \mathcal{B}_\xi$, then the value of $\xi(Z_t + z)$ is growing (with high probability) as $t \to \infty$ for each $z \in B(0, \rho_\xi)$. However, if $(\gamma ,\mu) \in \mathcal{B}_\xi$, this only occurs for $z \in B(0, \rho_\xi - 1)$; at the interface of the radius, where $|z| = \rho_\xi$, the value of $\xi(Z_t + z)$ instead converges to a certain random variable with law distinct from the law of $\xi(0)$. Similarly, for  $(\gamma ,\mu) \notin \mathcal{B}$, $\sigma(Z_t + z)$ converges to $\delta_\sigma$ for each $z \in B(0, \rho) \setminus \{0\}$. However, if $(\gamma ,\mu) \in \mathcal{B}$, then this is only true for $z \in B(0, \rho - 1) \setminus \{0\}$. If $|z| = \rho$, the value of $\sigma(Z_t + z)$ instead converges to a certain random variable with law distinct from the law of $\sigma(0)$. These properties are reflective of the fact that the correlation in the fields $\xi$ and $\sigma$ induced by the localisation site $Z_t$ decays away from the site.

We also explain why the cases $\gamma \le 1$ and $\mu = 0$ must be further distinguished in our profile for $\sigma(Z_t)$. If $\gamma > 1$ then the value of $\sigma(Z_t)$ is growing, and indeed growing with a deterministic leading order. However, if $\gamma = 1$, this is no longer true and instead $\sigma(Z_t)$ converges to a certain random variable with law distinct from the law of $\sigma(0)$.\footnote{Of course, in the case $\gamma < 1$, with overwhelming probability $\sigma$ is independent of the localisation site $Z_t$ (cf.\ part \eqref{thm:main3c} of Theorem~\ref{thm:main2}) and so $\sigma(Z_t)$ has the same law as $\sigma(0)$.} The case $\mu = 0$ must be distinguished for a different reason; in this case, the extremes of $\sigma$ are so large that there are many sites $z$ for which $\sigma^{-1}(z)$ is smaller than the gap in the top statistics of $\Psi_t$. Past this threshold, differences in the magnitude of $\sigma$ no longer materially influence the determination of $Z_t$, and so we lose a degree of certainty about the order of growth of $\sigma(Z_t)$.

Note finally that if $(\gamma ,\mu)$ is not in $\mathcal{B}$ and $\mathcal{B}_\xi$ respectively, then the probabilities in equations \eqref{eq:opt1} and \eqref{eq:opt2} actually converge to $0$ for any such $\psi_t$; otherwise, the respective probability will converge to a constant $c \in (0, 1)$. Similarly, if $(\gamma, \mu)$ lies to the right of the dashed or bold line in Figure~\ref{fig3}, the probabilities in \eqref{eq:opt3} and \eqref{eq:opt4} respectively converge to $0$ for any such $z_t$; if $(\gamma, \mu)$ lies on either line, the repective probability instead converges to a constant $c \in (0,1)$. We do not prove these additional results here.
\end{remark}

\subsection{Notation}
\label{subsec:notation}
Here we list notation that will be commonly used for the remainder of the paper. 

\textbf{Asymptotic notation}: For functions $f$ and $g$ we use $f \sim g$ to denote that
\[ \lim_{x \to \infty} f(x) / g(x) = 1 \, ,\]
and $f = o(g)$ or $f \ll g$ to denote that
\[ \lim_{x \to \infty} f(x) / g(x) = 0 \, .\]
We use $f = O(g)$ to denote that, as $x \to \infty$, eventually for some constant $c > 0$,
\[ |f(x)| < c |g(x)| \, .\]
 
\textbf{Notation for paths}:
For an integer $k$ and sites $y, z \in \ZZ^d$, let $\Gamma_k(y, z)$ be the set of nearest neighbour paths in $\ZZ^d$ of length $k$ running from $y$ to $z$, with each $p \in \Gamma_k(y, z)$ indexed as
$$ y =: p_0 \to p_1 \to p_2 \to \ldots \to p_k := z \, . $$
Similarly, denote
$$\Gamma_k(y) := \bigcup_{z \in \ZZ^d} \Gamma_k (y, z) \ , \quad \Gamma(y, z) := \bigcup_{k \in \mathbb{N}} \Gamma_k(y, z)\,, $$
$$ \Gamma(y) := \bigcup_{k \in \mathbb{N}} \Gamma_k (y) \ , \quad \Gamma := \bigcup_{y \in \mathbb{Z}^d} \Gamma(y) \, .$$
For a site $z\in \ZZ^d$, let $n(z)$ denote the number of shortest paths from the origin to $z$, i.e.\ 
\[ n(z) := |\Gamma_{|z|}(0, z)| \, .\]
For a path $p \in \Gamma_k(y, z)$ denote $\mathrm{Set}(p):=\{p_0,p_1,\ldots,p_k\}$ and $|p| := k$. For a nearest neighbour random walk $X$ let $p(X_t) \in \Gamma(X_0)$ denote the geometric path associated with the trajectory of $\{X_s\}_{s \le t}$ and let $p_k(X) \in \Gamma_k(X_0)$ denote the geometric path associated with the random walk $\{X_s\}_{s \ge 0}$ up to and including its $k^{\rm{th}}$ jump.

\textbf{Notation for sets}:
For a domain $D \in \mathbb{Z}^d$, denote by 
\[ \partial D = \{y \in D^c :  \text{there exists } x \in D \text{ such that } |x-y| = 1 \} \, . \]
For a set $S \in \ZZ^d$ define $B(S, n) := \bigcup_{z \in S} B(z, n)$ and $\text{sep}\left(S \right) := \min_{x \neq y \in S} \{ |x - y| \} $.

\textbf{Notation for solutions of the BAM}:
For each $y, z \in \mathbb{Z}^d$ define $u_y(t, z)$ to be the solution of the Cauchy problem
\begin{align*} 
\frac{\partial u_y(t, z)}{\partial t} &= (\Delta \sigma^{-1} + \xi) u_y(t, z) \, , \qquad (t, z) \in [0, \infty) \times \mathbb{Z}^d\,;\\
\nonumber u_y(0, z) &= \id_{\{y\}}(z) \, , \quad \qquad \qquad \qquad z \in \mathbb{Z}^d\,;
\end{align*}
and, for $z \in \mathbb{Z}^d$ and $p \in \Gamma$, define 
\[ u^p(t, z) := \EE_{p_0} \left[ \exp\left\{ \int_{0}^{t} \xi(X_s) \, ds \right\} \id_{ \{X_t = z \} }  \id_{\{ p(X_t) = p \}}  \right]  \, , \quad  U^p(t) := \sum_{z \in \ZZ^d} u^p(t, z)\, . \]

\textbf{Notation for local principal eigenvalues}:
For each $z \in \mathbb{Z}^d$ and $n \in \mathbb{N}$ define the $n$-local principal eigenvalue $\lambda^{(n)}(z)$ to be the principal eigenvalue of the Hamiltonian
\[ \mathcal{H}^{(n)}(z) := \Delta \invsigma + \xi \]
restricted to the domain  $B(z, n)$ with Dirichlet boundary conditions.

\textbf{Other notation}: For $a, b \in \mathbb{R}$ define $a \wedge b := \min\{a, b\}.$ 
\section{Outline of proof}
\label{sec:outline}
The main idea of the proof of Theorem~\ref{thm:main1} is that the solution $u(t, z)$ can be decomposed into disjoint components by reference to the trajectories of the underlying BTM in the Feynman-Kac representation in 
\eqref{eq:fk}. Using such a path decomposition, we prove complete localisation by establishing that: (i) a single component carries the entire non-negligible part of the solution; and (ii) the non-negligible component is localised at $Z_t$.

To assist in the proof, we introduce the scale
\begin{equation}
\label{eq:dt}
d_t := \frac{1}{\gamma}(d \log t)^{\frac{1}{\gamma} - 1} 
\end{equation}
which is the derivative (on the log scale) of the height scale $a_t$, and naturally examines the gaps in the maximisers of $\xi$ in growing boxes. We also introduce auxiliary scaling functions
 $f_t, h_t, e_t, b_t \to 0$ and $g_t, s_t \to \infty$ as $t \to \infty$ that are convenient placeholders for negligibly decaying (respectively growing) functions. For technical reasons, we shall need to choose these functions to satisfy certain relationships, as follows. First, let $s_t$ be such that
 $$ ( \log s_t)^2 \ll \log \log t \, .$$
Then, choose $f_t, h_t, e_t, b_t$ and $g_t$ satisfying
\begin{align}\label{scalingfns}
\max\{ \bar F_\sigma(s_t) ,  (\log s_t)^2 / \log \log t , 1/ \log \log s_t \} \,  g_t \ll b_t \ll f_t h_t  \ll g_t h_t \ll e_t \, .
\end{align}
It is easy to check that such a choice is always possible.

\subsubsection*{Path decomposition}
We explain here how to construct the path decomposition.  Recall the definition of $V_t$ from Section \ref{sec:full}. For a path $p \in \Gamma(0)$ such that $\mathrm{Set}(p) \subseteq V_t$, let 
\[ z^{(p)} := \argmax_{ z \in \mathrm{Set}(p) }  \lambda(z)  \]
which is well-defined almost surely. Abbreviate 
\[ B_t := B \left(0, |Z_t|(1 + h_t) \right) \cap V_t \, .\]
We partition the path set $\Gamma(0)$ into the following five disjoint components
$$ E^{i}_t := \begin{cases}
 \left\{ p \in \Gamma(0) : \mathrm{Set}(p) \subseteq B_t, \, z^{(p)} = Z_t \right\}   , & i = 1 \, ; \text{(non-negligible component)} \\
\left\{ p \in \Gamma(0) : \mathrm{Set}(p) \subseteq V_t, \, z^{(p)} \in \Pi^{(L_{t})} \setminus Z_t  \right\}  , & i = 2  \, ; \text{(path does not hit best site)} \\
\left\{ p \in \Gamma(0) : \mathrm{Set}(p) \subseteq V_t, \, \mathrm{Set}(p) \not\subseteq B_t , \, z^{(p)} = Z_t \right\}  , & i = 3  \, ; \text{(path travels far)} \\
\left\{ p \in \Gamma(0) : \mathrm{Set}(p) \subseteq V_t, \, z^{(p)} \notin \Pi^{(L_{t})}  \right\}  , & i = 4 \, ;  \text{(path avoids all high sites)}\\
\left\{ p \in \Gamma(0) :  \mathrm{Set}(p) \not\subseteq V_t  \right\}  , & i = 5 \, ;  \text{(path leaves macrobox)} \\
\end{cases} 
$$
and associate each component $E^i_t$ with a portion of the total mass $U(t)$ of the solution. As such, for each $1 \le i \le 5$, let
$$u^{i}(t, z) := \sum_{p \in E^i_t} u^p(t, z) \quad \text{and} \quad U^i(t) = \sum_{z \in \ZZ^d} u^i(t, z)   \, .$$
Our strategy is to establish that each of $U^2(t)$, $U^3(t)$, $U^4(t)$ and $U^5(t)$ are negligible with respect to the total mass $U(t)$ of the solution, in other words that,
$$ \frac{U^{i}(t)}{U(t)} \to 0 \qquad \text{in probability, \qquad for } i = 2,3,4,5 \, .$$
To complete the proof of localisation, we also prove that $U^1(t)$ is localised at $Z_t$, i.e.\ that,
$$ \frac{u^{1}(t, Z_t)}{U^1(t)} \to  1 \qquad \text{in probability}.$$
Note that this strategy requires a balance to be struck in how $B_t$ is defined; it must be large enough that $U^3(t)$ is negligible, but small enough to ensure localisation. The scale $h_t$ has been fine-tuned in \eqref{scalingfns} precisely to ensure this balance is struck correctly.

\subsubsection*{Negligible paths}
The negligibility of $U^4(t)$ and $U^5(t)$ are simple to establish; the main difficulty is establishing the negligibility of $U^{2}(t)$ and $U^{3}(t)$. Our proof is based on formalising two heuristics.

\noindent \textbf{First heuristic}: Recall the constant $j := [\gamma - 1] \ge \rho$ and the definition of the $j$-local principal eigenvalue $\lambda^{(j)}$ from Section \ref{subsec:notation}. We expect that, for a path $p \in \Gamma(0) \setminus E^5_t$,
\begin{align}
\label{eq:heuristic1}
 U^p(t) &\approx \exp \left\{ t {\lambda}^{(j)}(z^{(p)}) \right\} a_t^{-|p|} \ ,
\end{align}
which represents the balance between (i) the exponential growth of the solution at each site, and (ii) the probabilistic penalty for travelling each step along the path $p$. 

The accuracy of this heuristic relies on some subtle observations about the BAM (and indeed the PAM) which we shall briefly discuss further. First is the claim that the $j$-local principal eigenvalues closely approximate the exponential growth rate of the solution at a site (note that here we could take a slightly smaller constant in place of $j$, but $j$ will turn out to be convenient for another reason; see immediately below). This approximation, in turn, is based on the fact that there is a lack of resonance between the top eigenvalues of the operator $\Delta + \xi$ restricted to any finite domain.

Second is the claim that it is never beneficial for a path to visit other sites of high potential, other than $z^{(p)}$. This is proved by way of a `cluster expansion' (see Lemma~\ref{lem:clusterexp}) which bounds the contribution to $U^p(t)$ between the time $p$ visits a site $z$ of high potential until it leaves the ball $B(z, j)$. Crucially, $j$ is chosen precisely to be the smallest integer for which this `cluster expansion' bound is smaller than the probabilistic penalty associated with the path getting from outside the ball $B(z, j)$ to $z$ (see the proof of Proposition~\ref{prop:U2neg}).

Third is the claim that the probabilistic penalty for travelling along the path $p$ is approximately $1/a_t$ for each step of the path. Implicit in this claim is the highly non-trivial fact that the trapping landscape $\sigma$ is not an obstacle to the diffusivity of the particle, in other words, that a sufficiently `quick' path exists from $0$ to the site $z^{(p)}$. If $d \ge 2$, this is essentially due to percolation estimates; if $d = 1$, then this relies crucially on the additional tail decay assumption on the distribution of $\sigma(0)$, and our proofs and methods break down without it.

\noindent \textbf{Second heuristic}: We expect that, for $i = 1, 2, 3$,
\begin{align}
\label{eq:heuristic2}
U^i(t) \approx \max_{p \in E^i_t} U^p(t) \ ,
\end{align}
which represents the idea that $U^i(t)$ should be dominated by the contribution from just a single path in the path set $E^i_t$. This is essentially due to the fact that the number of paths of length~$k$ grows exponentially in $k$, whereas the probabilistic penalty associated with a path of length $k$ decays as $a_t^{-k}$, which dominates since $a_t \to \infty$.

Let us consider what these heuristics imply for $U(t)$. Recall the definition of $\Pi^{(L_t)}$ from Section~\ref{sec:full}. By analogy with $\Psi_t$ and $Z_t$, define
\[ \Psi^{(j)}_t(z) =  \lambda^{(j)}(z) - \frac{|z|}{\gamma t} \log \log t  \]
and $Z_t^{(j)} := \argmax_{z \in \Pi^{(L_t)}} \Psi^{(j)}_t$. Note that it will turn out that $Z_t^{(j)} = Z_t$ with overwhelming probability (see Corollary \ref{corry:zeqz0}), so we will interchange between these freely in the discussion that follows. Clearly, by the two heuristics, the dominant contribution to $U(t)$ will come from a path $p \in \Gamma(0)$ that goes directly from the origin to $z^{(p)}$, and so we expect that
\begin{align*} U(t) & \approx \max_{p \in \Gamma(0)} \left\{ \exp \left\{ t {\lambda}^{(j)}( z^{(p)}) \right\}   a_t^{-|z^{(p)}|} \right\} \approx \exp \left\{ t  \max_{z \in \mathbb{Z}^d} {\Psi}_t^{(j)}(z) \right\}  = \exp \left\{ t  {\Psi}_t^{(j)}( Z_t^{(j)}) \right\} \, .
 \end{align*}
Indeed, we formalise this approximation as a lower bound
\begin{align}
\label{eq:lb}
\log U(t) \ge t {\Psi}_t^{(j)}( Z_t^{(j)}) + O(t d_t b_t) \, .
\end{align}

Similarly for $U^2(t)$, the heuristics imply that the dominant contribution will come from the path $p \in E^2_t$ that goes directly from the origin to the site
$$ Z_t^{(j, 2)} = \argmax_{z \in \Pi^{(L_t)} \setminus \{Z_t^{(j)}\} } {\Psi}_t^{(j)}(z) \, ,$$
and so
$$ U^2(t) \approx \exp \left\{ t {\lambda}^{(j)}( Z_t^{(j, 2)}) \right\}  a_t^{-|Z_t^{(j, 2)}|} \approx \exp \left\{ t {\Psi}_t^{(j)}( Z_t^{(j, 2)}) \right\} \, .$$
We formalise this approximation as an upper bound
$$ \log U^2(t) \le t {\Psi}_t^{(j)}( Z_t^{(j, 2)}) + O(t d_t b_t) \, ,$$
which, together with equation \eqref{eq:lb}, implies that
\begin{align*}
\log U^2(t) - \log U(t) \le - t \left( \Psi_t^{(j)}( Z_t^{(j)}) - \Psi_t^{(j)}(Z_t^{(j, 2)}) + O(d_t b_t) \right)  \, .
\end{align*}
Remark that the negligibility of $U^2(t)$ is then a consequence of the gap in the top order statistics of ${\Psi}_t^{(j)}$ being larger than the error (of order $O(d_t b_t)$) in these bounds.

Finally, the heuristics imply that the dominant contribution to $U^3(t)$ will come from a path $p$ that visits $Z_t$ but that also ventures outside $B_t$, and so
$$ U^3(t) \approx \exp \left\{ t {\lambda}^{(j)}( Z_t) \right\}  a_t^{-|Z_t|(1 + h_t)} \, .$$
We formalise this approximation as an upper bound
\begin{align}
\label{eq:ub}
\log U^3(t) \le t {\lambda}^{(j)}( Z_t) - \frac{1}{\gamma} |Z_t| (1 + h_t) \log \log t  + O(t d_t b_t) 
\end{align}
which, together with equation \eqref{eq:lb}, implies that
\begin{align*}
\log U^3(t) - \log U(t) \le - \frac{1}{\gamma} |Z_t| h_t \log \log t + O(t d_t b_t)  \, .
\end{align*}
Remark that the negligibility of $U^3(t)$ is then a consequence of $|Z_t| h_t \log \log t$ being larger than the error (also of order $O(t d_t b_t)$) in these bounds.

In Section~\ref{sec:extremal} we study extremal theory for $\lambda^{(j)}$ and $\Psi_t^{(j)}$, demonstrating, in particular, that
$$ \Psi_t^{(j)}(Z_t^{(j)}) - \Psi_t^{(j)}(Z_t^{(j, 2)}) > d_t e_t \quad \text{and} \quad |Z_t^{(j)}| h_t  \log \log t  > t d_t e_t $$
both hold eventually with overwhelming probability. We also show that $Z_t^{(j)} = Z_t$ with overwhelming probability. In the process, we establish the description of the localisation site $Z_t$ that is contained in Theorem~\ref{thm:main2}, as well as the optimality results in Theorem~\ref{thm:main3}. In Section~\ref{sec:neg}, we show how to formalise the heuristics in equations \eqref{eq:heuristic1} and \eqref{eq:heuristic2} into the bounds in equations \eqref{eq:lb} and \eqref{eq:ub}, and so complete the proof of the negligibility of $U^2(t)$ and $U^3(t)$. 

Throughout, we draw on the preliminary results established in Sections~\ref{sec:gentheory} and \ref{sec:prop}. Section~\ref{sec:gentheory} contains a compilation of general results on operators of the form $\Delta \invsigma + \xi$. Section~\ref{sec:prop} contains general results on the random fields $\xi$ and $\sigma$. Of particular concern here is the existence of `quick' paths through the trapping landscape $\sigma$.

\subsubsection*{Localisation}
In Section~\ref{sec:loc} we complete the proof of Theorem~\ref{thm:main1} by showing that $u^1(t, z)$ is localised at the site $Z_t$. The main idea is the same as in \cite{Gartner07} and \cite{Sidorova12}, namely that: (i) the solution $u^1(t, z)$ is asymptotically approximated by the principal eigenfunction of the operator $\Delta \invsigma + \xi$ restricted to the domain~$B_t$; and (ii) the principal eigenfunction decays exponentially away from the site~$Z_t$. Underlying this reasoning is the fact that the domain $B_t$ has been constructed to ensure that $\lambda^{(j)}(Z_t)$ is the largest $j$-local principal eigenvalue in the domain. This in turn allows us to give a Feynman-Kac representation of the principal eigenfunction $v_t$ (see Proposition~\ref{prop:pathexpeig}), which we analyse probabilistically to establish exponential decay.
\section{General theory for the BAM}
\label{sec:gentheory}
In this section we develop general theory for operators of the form $\Delta \invsigma + \xi$ which is valid for arbitrary $\xi$ and positive $\sigma$. This section will be entirely self-contained and is completely deterministic, and may be of independent interest. We have chosen to develop the theory in full generality so as to take advantage of the results in future work.

Throughout this section, let $D \subset \mathbb{Z}^d$ be a bounded domain and let $\xi$ and $\sigma$ be arbitrary functions $\xi: \mathbb{Z}^d \to  \mathbb{R}$ and $\sigma: \mathbb{Z}^d \to \mathbb{R}^+$,  with $\eta := \xi - \invsigma$. Denote by $\mathcal{H}$ the Hamiltonian
$$ \mathcal{H} := \Delta \invsigma + \xi $$
restricted to the domain $D$ with Dirichlet boundary conditions, and let $\{\lambda_i\}_{i \le |D|}$ and $\{\varphi_i\}_{i \le |D|}$ be respectively the (finite) set of eigenvalues and eigenfunctions of $\mathcal{H}$, with eigenvalues in descending order and eigenfunctions $\ell_2$ normalised. Finally, recall that $X_s$ denotes the BTM and define the stopping times
$$ \tau_z := \inf\{ t \ge 0 : X_t = z \} \quad \text{and} \quad  \tau_{D^c} := \inf\{ t \ge 0 : X_t \notin D \} \, .$$

We start by presenting representations and deriving simple bounds for $\lambda_1$ and $\varphi_1$.

\begin{lemma}[Principal eigenvalue monotonicity]
\label{lem:mono}
For each $z \in D$ and $\delta > 0$, let $\bar \lambda_1$ be the principal eigenvalue of the operator $\mathcal{H} + \delta \indic{z}$. Then $\bar \lambda_1 > \lambda_1$.

Moreover, for each strictly smaller domain $\bar D \subset D$, let $\bar \lambda_1$ be the principal eigenvalue of~$\mathcal{H}$ restricted to the domain $\bar D$ with Dirichlet boundary conditions. Then $\bar\lambda_1 < \lambda_1$.
\end{lemma}
\begin{proof}
These are general properties of elliptic operators.
\end{proof}

\begin{lemma}[Bounds on the principal eigenvalue]
\label{lem:minmax}
$$  \max_{z \in D} \left\{\eta(z)\right\} \le \lambda_1 \le \max_{z \in D} \left\{  \eta(z) + \sum_{|y-z|=1} \frac{1}{2d} \invsigma(y) \right\} \, .$$
\end{lemma}
\begin{proof}
The lower bound follows from the min-max theorem for the principal eigenvalue; the upper bound follows from the Gershgorin circle theorem.
\end{proof}

\begin{proposition}[Feynman-Kac representation for the principal eigenfunction]
\label{prop:genfk}
For each $y, z \in D$ the principal eigenfunction $\varphi_1$ satisfies the Feynman-Kac representation
\begin{align}
\label{eq:genfk}
\frac{\varphi_1(y)}{\varphi_1(z)} = \frac{\sigma(y)}{\sigma(z)} \, \EE_y \left[ \exp \left\{ \int_0^{\tau_z} \left( \xi(X_s) - \lambda_1 \right) \, ds  \right\} \id_{\{ \tau_{D^c} > \tau_z \}}  \right] .
\end{align}
\end{proposition}
\begin{proof}
Consider $z$ fixed and define $v^z(y) := \varphi_1(y)/\varphi_1(z)$. Note that the function $v^z$ satisfies the Dirichlet problem
\begin{align*}
 (\Delta \invsigma + \xi - \lambda_1) \,  v^z (y) & = 0\, ,  & y \in D \setminus \{z\} \, ; \\
 v^z(y) & = \id_{\{z\}}(y)\, , & y \notin D \setminus \{z\} \, .
\end{align*} 
It is easy to check (for instance, by integrating over the first holding time) that the Feynman-Kac representation on the right-hand side of equation \eqref{eq:genfk} also satisfies this Dirichlet problem; hence we are done if there is a unique solution. So assume another non-trivial solution $w$ exists. Then the difference $q := v^z - w$ satisfies the Dirichlet problem
\begin{align*}
(\Delta \invsigma + \xi - \lambda_1) \, q(y) & = 0\, ,  & y \in D \setminus \{z\} \, ; \\ 
q(y) & = 0\, , & y \notin D \setminus \{z\}  \, ;
\end{align*} 
which is nonzero if and only if $\lambda_1$ is an eigenvalue of the operator $\Delta \invsigma + \xi$ restricted to the domain $D \setminus \{z\}$ with Dirichlet boundary conditions. By the domain monotonicity of the principal eigenvalue in Lemma~\ref{lem:mono}, this is impossible.
\end{proof}

\begin{lemma}[Path-wise evaluation]
\label{lem:pathwise}
For each $k \in \mathbb{N}$, $y, z \in D$, $p \in \Gamma_k(z, y)$ such that $p_i \neq y$ for $i < k$ and $\mathrm{Set}(p) \subseteq D$, and $\zeta>\max_{0\le i < k}\eta(p_i)$, we have
\[ \mathbb{E}_z \left[ \exp \left\{\int_0^{\tau_y}(\xi(X_s)-\zeta) \, ds \right\} \indic{ p_k(X) = p} \right] =\prod_{i=0}^{k-1} \\
\frac{1}{2d} \frac{\invsigma(p_i)}{\zeta- \eta(p_i)}  \, .\]
\end{lemma}
\begin{proof}
This follows by integrating over the holding times at the sites $\{p_i\}_{0 \le i \le k-1}$, which are independent. The restriction on $\zeta$ ensures that the resulting integrals are finite.
\end{proof}

\begin{proposition}[Path expansion for the principal eigenvector]
\label{prop:genpathexpeig}
For each $y, z \in D$ the principal eigenfunction $\varphi_1$ satisfies the path expansion
\begin{align*}
\frac{\varphi_1(y)}{\varphi_1(z)} = \frac{\sigma(y)}{\sigma(z)} \ \sum_{k \ge 1} \ \sum_{\substack{ p \in \Gamma_{k}(y, z) \\ p_i \neq z, \, 0 \le i < k \\ \mathrm{Set}(p) \subseteq D}} \ \prod_{0 \leq i < k} \frac{1}{2d} \frac{\invsigma(p_i)}{\lambda_1 - \eta(p_i)}   \, . 
\end{align*}
\end{proposition}
\begin{proof}
The expectation on the right-hand side of equation \eqref{eq:genfk} can be expanded path-wise using Lemma~\ref{lem:pathwise}, which is valid by the lower bound in Lemma~\ref{lem:minmax}.
\end{proof}

\begin{remark}
Note that the initial factor $\sigma(y)$ in the above path expansion cancels with the term $\invsigma(p_1)$ appeaing in each component of the sum. This turns out to be crucial in establishing the localisation of the eigenfunctions in Sections \ref{sec:neg} and \ref{sec:loc}, since \textit{a priori} $\sigma(y)$ could be arbitrarily large.
\end{remark}

\begin{proposition}[Path expansion for the principal eigenvalue]
\label{prop:genpathexp}
For each $z \in D$ the principal eigenvalue has the path expansion
\begin{align*}
\lambda_1 &= \eta(z) + \invsigma(z) \sum_{k \geq 2} \ \sum_{\substack{ p \in \Gamma_{k}(z, z) \\ p_i \neq z , \, 0 < i < k \\ \mathrm{Set}(p) \subseteq D  } } \ \prod_{0 < i < k} \frac{1}{2d} \frac{\invsigma(p_i)}{\lambda_1 - \eta(p_i)}  \, .
\end{align*}
\end{proposition}
\begin{proof}
Recalling that the eigenfunction relation evaluated at a site $z$ gives 
$$ \lambda_1 = \eta(z) +  \sum_{|y-z| = 1} \invsigma(y) \, \frac{\varphi_1(y)}{\varphi_1(z)} \, ,$$ 
the result follows from Proposition~\ref{prop:genpathexpeig}.
\end{proof}

We now study the solution $u_z(t, y)$ to the Cauchy problem
\begin{align}
\label{eq:cauchy}
\frac{\partial u_z(t, y)}{\partial t} &= \mathcal{H} \, u(t, y)\ ,  & (t, y) \in [0, \infty) \times D \ ;\\
\nonumber u_z(0, y) &= \id_{\{z\}}(y)\ ,  & y \in \ZZ^d  \, .
\end{align}
In particular, we give the spectral representation of $u_z(t, y)$ and deduce upper and lower bounds.

\begin{proposition}[Feynman-Kac representation of the solution] 
\label{prop:fksln}
For each $y, z \in D$,
\begin{align*}
u_z(t, y) = \EE_{z} \left[ \exp \left\{ \int_0^t \xi(X_s) ds  \right\} \id_{\{X_t = y\}} \id_{ \{ \tau_{D^c} > t  \} } \right] \, .
\end{align*}
\end{proposition}
\begin{proof}
It can be directly verified that the Feynman-Kac representation satisfies \eqref{eq:cauchy}.
\end{proof}

\begin{lemma}[Time-reversal]
\label{lem:tr}
For each $y, z \in D$,
$$ u_z(t, y) \, \sigma(z) = u_y(t, z)\,  \sigma(y)  \, .$$
\end{lemma}
\begin{proof}
Consider the Hermitian operator
$$ \tilde{\mathcal{H}} := \sigma^{-\frac{1}{2}} \mathcal{H} \sigma^{\frac{1}{2}} = \sigma^{-\frac{1}{2}} \Delta  \sigma^{-\frac{1}{2}} + \xi $$
which can be viewed as the `symmetrised' form of $\mathcal{H}$. Since, 
\[ e^{\tilde{\mathcal{H}} t}  =  e^{\sigma^{-\frac{1}{2}} \mathcal{H} \sigma^{\frac{1}{2}}  t}  = \sigma^{-\frac{1}{2}} e^{\mathcal{H} t} \sigma^{\frac{1}{2}}\,, \]
we have, by the fact that $\tilde{\mathcal{H}}$ is Hermitian,
\begin{align*}
 u_z(t, y) & = e^{\mathcal{H} t} \indic{z} (y)  = \left( \frac{\sigma(y)}{\sigma(z)} \right)^{\frac{1}{2}} e^{\tilde{\mathcal{H}} t} \indic{z} (y)  = \left( \frac{\sigma(y)}{\sigma(z)} \right)^{\frac{1}{2}} e^{\tilde{\mathcal{H}} t} \indic{y} (z)  \\ 
 & =   \frac{\sigma(y)}{\sigma(z)}  e^{\mathcal{H} t} \indic{y} (z)   = \frac{\sigma(y)}{\sigma(z)} u_y(t, z) \, . \qedhere
 \end{align*}
\end{proof}

\begin{proposition}[Spectral representation for the solution]
\label{prop:sr}
For each $y, z \in D$, the solution $u_z(t, y)$ satisfies the spectral representation 
$$ u_z(t, y) =   \invsigma(z)  \sum_{i} \frac{ e^{\lambda_i t}  \varphi_i(z) \varphi_i(y) }{|| \sigma^{-\frac{1}{2}} \varphi_i ||^2_{\ell_2} }  \, .$$
\end{proposition}
\begin{proof}
Recall the Hermitian operator  $ \tilde{\mathcal{H}}$ from the proof of Lemma~\ref{lem:tr}. Note that each ($\ell_2$ normalised) eigenfunction $\tilde{\varphi}_i$ of $\tilde{\mathcal{H}}$ satisfies the relation
$$ \tilde{\varphi}_i = \frac{ \sigma^{-\frac{1}{2}} \varphi_i }{ || \sigma^{-\frac{1}{2}} \varphi_i ||_{\ell_2}  }  $$
with $\lambda_i$ the corresponding eigenvalue for $\tilde{\varphi}_i$. The proof then follows by applying the spectral theorem to $\tilde{\mathcal{H}}$.
\end{proof}

\begin{corollary}[Bounds on the solution]
\label{cor:sb}
For each $z \in D$ we have the bounds
$$\frac{ e^{\lambda_1 t} \invsigma(z)  \varphi^2_1(z) }{ || \sigma^{-\frac{1}{2}} \varphi_1 ||^2_{\ell_2}   } \le u_z(t, z) \le   e^{\lambda_1 t}  \, .$$
\end{corollary}
\begin{proof}
The lower bound follows directly from Proposition~\ref{prop:sr}. For the upper bound, first use Proposition~\ref{prop:sr} to write
\[  u_z(t, z) \le e^{\lambda_1 t}  \invsigma(z) \sum_i \frac{\varphi^2_i(z)}{\|\sigma^{-\frac1{2}}\varphi_i\|_{\ell_2}^2} \, . \]
Then, since $u_z(0,z)=1$, Proposition~\ref{prop:sr} also implies that
\[ \invsigma(z) \sum_i \frac{\varphi^2_i(z)}{\|\sigma^{-\frac1{2}}\varphi_i\|_{\ell_2}^2} = 1 \]
and the result follows. 
\end{proof}

\begin{proposition}[Bound on the total mass of the solution]
\label{prop:tm}
For each $y, z \in D$,
\[ \sum_{y \in D} u_z(t, y) \le e^{\lambda_1 t}  \sum_{y \in D} \frac{\varphi_1(y)}{\varphi_1(z)} \, . \]
\end{proposition}
\begin{proof}
We write $\mathcal{F}_{\tau_z}$ for the $\sigma$-algebra generated by the stopping time $\tau_z$. First decompose the Feynman-Kac representation for $u_y(t, z)$ in Proposition~\ref{prop:fksln} by conditioning on $\mathcal{F}_{\tau_z}$ and using the strong Markov property:
\begin{align*}
u_y(t, z) & = \EE_{\tau_z} \left[ e^{ \lambda_1 \tau_z}  \, \EE_y \left[ \exp \left\{ \int_0^{\tau_z} \left( \xi(X_s) - \lambda_1 \right)  \, ds \right\} \id_{\{ \tau_z <\tau_{D^c}\}} \bigg| \, \mathcal{F}_{\tau_z}  \right] \right. \\
& \qquad \left. \times \,  \EE_z \left[ \exp \left\{ \int_0^{t - \tau_z} \xi(X'_s) \, ds \right\} \id_{\{ X'_{t - \tau_z} = z, \tau'_{D^c} > t - \tau_z  \}}\bigg| \, \mathcal{F}_{\tau_z} \right]   \id_{\{ \tau_z \le t  \}} \right] \\
& = \EE_{\tau_z} \left[ e^{ \lambda_1 \tau_z }  \, \EE_y \left[ \exp \left\{ \int_0^{\tau_z} \left( \xi(X_s) - \lambda_1 \right)  \, ds \right\} \id_{\{ \tau_z <\tau_{D^c}\}} \bigg| \, \mathcal{F}_{\tau_z}  \right]  u_z(t-\tau_z, z) \id_{\{ \tau_z \le t  \}} \right]  \, ,
\end{align*}
where $\EE_{\tau_z}$ denotes expectation taken over $\tau_z$, $X_t'$ is an independent copy of $X_t$, and $\tau'_{D^c} := \inf \{t \ge 0: X'_t \notin D\}$. Using the upper bound in Corollary \ref{cor:sb} combined with the Feynman-Kac representation for the principal eigenfunction in Proposition~\ref{prop:genfk}, we have that
\begin{align*}
u_y(t, z)  \le e^{\lambda_1t}\EE_{y} \left[   \exp \left\{ \int_0^{\tau_z} \left( \xi(X_s) - \lambda_1 \right)  \, ds \right\} \id_{\{ \tau_z <\tau_{D^c}\}}\right]  = e^{\lambda_1t}  \frac{\varphi_1(y)}{\varphi_1(z)} \frac{\sigma(z)}{\sigma(y)}   \, .
\end{align*}
Finally, applying the time-reversal Lemma~\ref{lem:tr}, we have
\begin{align*}
u_z(t, y) = u_y(t, z) \frac{\sigma(y)}{\sigma(z)} \le e^{\lambda_1 t}  \frac{\varphi_1(y)}{\varphi_1(z)}   \, ,
\end{align*}
which, after summing over $y \in D$, yields the result.
\end{proof}

Next we prove a `cluster expansion' that is useful for bounding expectations of the `Feynman-Kac type'. It is similar in spirit to \cite[Lemma~4.2]{Gartner07} and \cite[Lemma~2.18]{Gartner98}, however we will need an additional form of the bound to accommodate the impact of the trapping landscape (see the proof of Lemma~\ref{L:negupper2}).

\begin{lemma}[Cluster expansion]
\label{lem:clusterexp}
For each $z \in D$ and for any $\zeta > \lambda_1$,
\begin{align*}
\EE_{z} \left[ \exp \left\{ \int_{0}^{\tau_{D^c}} (\xi(X_s) - \zeta) \, ds \right\} \right]  & \le 1 +   \frac{ \max_{w \in D} \{\invsigma(w)\} \, |D|}{\zeta-\lambda_1} 
\end{align*}
and
\begin{align*}
\EE_{z} \left[ \exp \left\{ \int_{0}^{\tau_{D^c}} (\xi(X_s) - \zeta) \, ds \right\} \right]  & \le  \frac{\invsigma(z)}{\zeta - \lambda_1}  \left( 1 +   \frac{ \max_{w \in D} \{\invsigma(w)\} \, |D|}{\zeta-\lambda_1}  \right) \, . 
\end{align*}
\end{lemma}
\begin{proof}
We proceed by modifying the proofs of \cite[Lemma~4.2]{Gartner07} and \cite[Lemma~2.18]{Gartner98}. First abbreviate 
\[  u(y):=  \EE_{y} \left[ \exp \left\{ \int_{0}^{\tau_{D^c}} (\xi(X_s) - \zeta) \, ds \right\} \right] \]
and remark that $u$ solves the boundary value problem
\begin{align}
\label{eq:cluster1}
 (\sigma^{-1} \Delta +\xi-\zeta)u(y) &=0\, , & y \in D\, ;\\
 \nonumber u(y) &=1\, , & y \notin D \, .
\end{align}
Note that, in contrast to in the proof of Proposition \ref{prop:genfk}, in the above boundary value problem the relevant operator is the adjoint of $\mathcal{H}$, since here we have not weighted the expectation by~$\sigma$. We make the substitution $w:=u-\id$, where $\id$ denotes the vector of ones, which turns~\eqref{eq:cluster1} into
 \begin{align*}
 (\sigma^{-1} \Delta+\xi-\zeta)w(y) &= - (\sigma^{-1} \Delta+\xi-\zeta) \id (y) = \zeta - \xi(y) \, , & y \in D \,;\\
w(y) &=0\, , & y \notin D \, .
 \end{align*}
Since $\zeta > \lambda_1$, the solution exists and is given by 
$$ w(y) = \left( \mathcal{R}_\zeta (\xi - \zeta ) \right) (y) $$
where $\mathcal{R}_\zeta$ is the resolvent of $\invsigma \Delta + \xi$ at $\zeta$. By Lemma~\ref{lem:minmax} and since $\zeta > \lambda_1$ we have that $\xi(y)   - \zeta \le \invsigma$ for all $y \in D$, and so by the positivity of the resolvent (guaranteed since $\mathcal{H}$ is elliptic and $\zeta > \lambda_1$) we obtain 
\[ w(z)  \le  \left( \mathcal{R}_\zeta  \invsigma  \right) (y)  =    \left( \sigma^{-\frac{1}{2}} \tilde{\mathcal{R}}_\zeta \sigma^{-\frac{1}{2}} \right) (y)  \le  \max_{z \in D} \{\invsigma(z)\} |D| \, \| \tilde{\mathcal{R}}_\zeta  \| \, , \]
where $\tilde{\mathcal{R}}_\zeta$ is the resolvent of the Hermitian operator $\tilde{\mathcal{H}} = \sigma^{-\frac{1}{2}} \Delta \sigma^{-\frac{1}{2}} + \xi$ at $\zeta$ and $\| \cdot \|$ denotes the operator norm. By considering the spectral representation of $\tilde{\mathcal{R}}_\zeta$ we have $\| \tilde{\mathcal{R}}_\zeta \| \le \left( \zeta-\lambda_1 \right)^{-1} $ which gives the first bound. For the second bound, consider that \eqref{eq:cluster1} implies the identify
\begin{equation}
\label{eq:cluster}
u(y) = \frac{ \invsigma(y) }{\zeta - \xi(y) + \invsigma(y) } \sum_{|x - y|=1} \frac{1}{2d} u(x) \, .
\end{equation} 
Applying the first bound to each $u(x)$ in the sum in \eqref{eq:cluster}, the result follows by bounding $\xi(y) - \invsigma(y)$ in the denominator of \eqref{eq:cluster} from above by $\lambda_1$, valid by the lower bound in Lemma~\ref{lem:minmax}.
\end{proof}

Finally, we give a general way to bound the contribution to the solution $u_z(t, y)$ from paths that hit a certain site $x \in D$ and then stay within a subdomain $E \subseteq D$ that contains $x$. In particular, we show that this contribution is proportional to the principal eigenfunction of $\mathcal{H}$ restricted to $E$. This is similar in spirit to \cite[Theorem 4.1]{Gartner07}, and it crucial to establishing complete localisation of the solution.

So fix a domain $E \subseteq D$, a site $x \in E$, and define the operator $\mathcal{H}^E$ to be the restriction of $\mathcal{H}$ to the domain $E$ with Dirichlet boundary conditions, with $\lambda^E_1$ and $\varphi^E_1$ respectively its principal eigenvalue and eigenfunction. Define the stopping time
$$ \tau_{x, E^c} := \inf \{t \ge \tau_x : X_t \notin E \} \, .$$
Then the contribution to the solution $u_z(t, y)$ from paths that hit $x$ and then stay within $E$ can be written
$$ u_z^{x, E}(t, y) := \EE_z \left[ \exp \left\{ \int_0^t  \xi(X_s)  \, ds  \right\}  \id_{\{ X_t = y, \tau_x \le t,  \tau_{x, E^c} > t, \tau_{D^c} > t \}} \right] \, .  $$

\begin{proposition}[Link between solution and principal eigenfunction; see {\cite[Theorem 4.1]{Gartner07}}]
\label{prop:thm4.1}
For each $x\in E$, $y \in E \setminus \{x\}$ and $z \in D$,
$$ \frac{u_z^{x, E}(t, y)}{\sum_{y \in D} u_z(t, y)} \le \frac{\sigma(x) \|\sigma^{-\frac1{2}} \varphi_1^E\|_{\ell_2}^2} { ( \varphi_1^E(x))^3 } \, \varphi_1^E(y)\,. $$
\end{proposition}
\begin{proof}
We proceed by modifying the proof of \cite[Theorem 4.1]{Gartner07}. The first step is to make use of the time-reversal in Lemma~\ref{lem:tr}, suitably adapted to $u_z^{x, E}(t, y)$. In particular, defining
\[ u_y^{\overleftarrow{x, E}}(t, z) := \EE_y \left[ \exp \left\{ \int_0^t  \xi(X_s)  \, ds  \right\}  \id_{\{ X_t = z, \tau_x  \le t, \tau_x < \tau_{E^c}, \tau_{D^c} > t \}} \right]   \]
we can write
\begin{align}
\label{eq:thm4.1}
\frac{u_z^{x, E}(t, y)}{\sum_{y \in D}u_z(t, y)}   \le \frac{u_z^{x, E}(t, y)}{u_z(t, x)}  =  \frac{\sigma(y)}{\sigma(x)} \frac{u_y^{\overleftarrow{x, E}}(t, z)}{u_x(t, z) } \, .
\end{align}
Next we decompose the Feynman-Kac formula for $u_y^{\overleftarrow{x, E}}(t, z)$ as in the proof of Proposition~\ref{prop:tm}, by conditioning on the $\sigma$-algebra generated by the stopping time $\tau_x$, and using the strong Markov property. More precisely, we write
\begin{align}
\label{eq:uxE}
u_y^{\overleftarrow{x, E}}(t, z) & = \EE_{\tau_x} \left[ e^{ \tau_x \lambda^E_1  }  \, \EE_y \left[ \exp \left\{ \int_0^{\tau_x} \left( \xi(X_s) - \lambda^E_1 \right)  \, ds \right\} \id_{\{ \tau_x <\tau_{E^c}\}} \bigg| \, \mathcal{F}_{\tau_x}  \right] \right. \\
& \nonumber \qquad \left. \times \,  \EE_x \left[ \exp \left\{ \int_0^{t - \tau_x} \xi(X'_s) \, ds \right\} \id_{\{ X'_{t - \tau_x} = z, \tau'_{D^c} > t - \tau_x  \}}\bigg| \, \mathcal{F}_{\tau_x}   \right]   \id_{\{ \tau_x \le t  \}} \right] \, ,
\end{align}
where $\EE_{\tau_x}$, $X_t'$ and $\tau'_{D^c}$ are defined as in the proof of Proposition~\ref{prop:tm}. Next, note that an application of Corollary \ref{cor:sb} gives the bound
\begin{align}
\label{eq:bound}
 1 \le u^{x, E}_x(w, x) \  \frac{\sigma(x) \|\sigma^{-\frac1{2}} \varphi^E_1\|^2_{\ell_2}  }{ ( \varphi^E_1(x) )^2} \  e^{-w \lambda^E_1 }  \, ,
 \end{align}
and recall the representation
$$  u^{x, E}_x(w, x) =  \EE_x \left[ \exp \left\{ \int_0^w \xi(X'_s) \, ds \right\}  \id_{ \left\{ X'_w =  x, \tau'_{E^c} > w \right\}}   \right]  \, .$$
Combining the bound in \eqref{eq:bound} with equation \eqref{eq:uxE} (setting $w = \tau_x$), gives
\begin{align*}
u_y^{\overleftarrow{x, E}}(t, z) & \le \frac{\sigma(x)  \|\sigma^{-\frac1{2}} \varphi^E_1\|^2_{\ell_2}  }{ ( \varphi^E_1(x) )^2}    \ \EE_{\tau_x} \left[ \EE_y \left[ \exp \left\{ \int_0^{\tau_x} \left( \xi(X_s) - \lambda^E_1 \right)  \, ds \right\} \id_{\{    \tau_{E^c} > \tau_x \}} \bigg| \, \mathcal{F}_{\tau_x}\right] \right. \\
& \qquad  \times \EE_x \left[ \exp \left\{ \int_0^{\tau_x} \xi(X'_s) \, ds \right\}  \id_{ \left\{ X'_{\tau_x } =  x, \tau'_{E^c} > \tau_x \right\}}  \bigg| \, \mathcal{F}_{\tau_x}   \right]   \\ 
& \qquad  \left. \times \, \EE_x \left[ \exp \left\{ \int_0^{t- \tau_x} \xi(X'_s) \, ds \right\}  \id_{ \left\{ X'_{t - \tau_x} = z , \tau'_{D^c} > t - \tau_x \right\}}  \bigg| \, \mathcal{F}_{\tau_x}   \right]  \id_{\{  \tau_x \le t \}}   \right] \,  \\
& \le \frac{\sigma(x)  \|\sigma^{-\frac1{2}} \varphi^E_1\|^2_{\ell_2}  }{ ( \varphi^E_1(x) )^2}     \ \EE_{\tau_x} \left[ \EE_y \left[ \exp \left\{ \int_0^{\tau_x} \left( \xi(X_s) - \lambda_1^E \right)  \, ds \right\} \id_{\{    \tau_{E^c} > \tau_x \}} \bigg| \, \mathcal{F}_{\tau_x} \right] \right. \\
& \qquad  \left. \times \EE_x \left[ \exp \left\{ \int_0^t \xi(X'_s) \, ds \right\}  \id_{ \left\{ X'_t = z , \tau'_{D^c} > t \right\}}   \bigg| \, \mathcal{F}_{\tau_x} \right]  \id_{\{  \tau_x \le t \}}  \right]  \\ 
& \le \frac{\sigma(x)^2  \|\sigma^{-\frac1{2}} \varphi^E_1\|^2_{\ell_2}  }{ \sigma(y)( \varphi^E_1(x) )^3}  \,  \varphi^E_1(y)  \ u_x(t, z)\, ,
\end{align*}
where the inequality in the second step results from deleting the condition that $X'_{\tau_x} = x$, and where the last inequality results from deleting the condition that $\tau_x \le t$, and where we have used the Feynman-Kac representation for $\varphi^E_1$ given by Proposition~\ref{prop:genfk}. Combining this with equation~\eqref{eq:thm4.1} gives the result.
\end{proof}
\section{Properties of the random environments}
\label{sec:prop}
In this section we establish properties of the i.i.d.\ fields $\xi$ and $\sigma$. In the first part we give asymptotics for the upper order statistics of $\xi$ and $\sigma$. The second part is devoted to proving the existence of `quick paths', which are an essential part of our proof that the trapping landscape does not prevent complete localisation in the BAM.

\subsection{Almost sure asymptotics for $\xi$ and $\sigma$} 

For each $a \le 1$, define the macrobox level $L_{t, a} := ((1-a) \log |V_t|)^{\frac{1}{\gamma}}$ and let the subset $\Pi^{(L_{t,a})} := \left\{ z \in \mathbb{Z}^d : \xi(z) > L_{t,a} \right\} \cap V_t$ consist of sites in $V_t$ at which $\xi$-exceedances of the level $L_{t, a}$ occur. Recall that $L_t := L_{t, \theta}$. 

\begin{lemma}[Almost sure asymptotics for $\xi$]  \label{lem:asforxi}
Denote by $\xi_{t, i}$ the $i^{\rm{th}}$ highest value of $\xi$ in $V_t$. Then for $a \in [0, 1)$ and $a' \in (0, 1]$, as $t \to \infty$,
$$\xi_{t, [|V_t|^{a}]} \sim L_{t, a} \qquad \text{and} \qquad |\Pi^{(L_{t, a'})} | \sim |V_t|^{a'} $$
hold almost surely.
\end{lemma}
\begin{proof}
These follow from well-known results on sequences of i.i.d.\ random variables; they are proved in a similar way as \cite[Lemma~4.7]{vanDerHofstad08}.
\end{proof}

Recall that for a set $S \in \ZZ^d$ we denote by $\text{sep}\left(S \right) := \min_{x \neq y \in S} \{ |x - y| \} $.
\begin{lemma}[Almost sure separation of high points of $\xi$]
\label{lem:assep}
For any $a > 0$ and $n \in \mathbb{N}$ let 
$$ \Pi_n^{(L_t, a)} := \{ z \in B(V_t, n) : \xi(z) > L_{t, a} \}  $$
be the set of $L_{t, a}$ exceedences of $\xi$ in the $n$-extended macrobox $B(V_t, n)$. Then,  for any $a'<a$, as $t \to \infty$
$$\text{sep} \left( \Pi_n^{(L_{t, a})} \cup \{0\}  \right) > |V_t|^{\frac{1-2a'}{d}} $$
eventually almost surely.
\end{lemma}
\begin{proof}
This result is proved as in \cite[Lemma~1]{Astrauskas07}.
\end{proof}
\begin{remark}
Note that we need the almost sure separation of high points in the $n$-extended macrobox $B(V_t, n)$ rather than just in $V_t$ because each ${\lambda}^{(n)}(z)$, for $z \in V_t$, depends on the random environments $\xi$ and $\sigma$ in the ball $B(z, n) \subseteq B(V_t,n)$. This result implies that, eventually almost surely, each $z \in \Pi^{(L_{t, a})}$ has the property that $\xi(y) < L_{t, a}$ for all $y \in B(z, n) \setminus \{z\}$.
\end{remark}

\begin{corollary}[Paths cannot always remain close to high points of $\xi$]
\label{cor:pathoutsideballs}
There exists a $c \in (0, 1)$ such that, for each $n \in \mathbb{N}$, all paths $p \in \Gamma(0, z)$ such that $\mathrm{Set}(p) \subseteq V_t$ satisfy, as $t \to \infty$,
  \[   \left| \left\{ i : p_i \notin B(\Pi^{(L_t)},n) \right\} \right| \ge  |z| - \frac{|z|}{t^c} ,\]
  eventually almost surely.
\end{corollary}
\begin{proof}
Abbreviate $N := \text{sep}  (\Pi^{(L_t)} \cup \{0\} ) $ and 
 $$Q := \left| \left\{ i : p_i \notin B(\Pi^{(L_t)},n) \right\} \right| \, .$$
Suppose a path $p$ passes through $m$ distinct $B(x, n)$ with $x \in \Pi^{(L_t)}$. Then, since there is a minimum distance of  $(N-2n)$ between each such ball, the path $p$ satisfies
 \[  Q  \ge m (N- 2n ) \, .\]
On the other hand, it is clear that that $Q \ge |z| - (2n+1)m$. Therefore
\[ Q \ge \min_{m \in \mathbb{N}}  \max \left\{ m(N-2n-1), |z| - (2n+1)m \right\} \ge \frac{(N-2n-1)|z|}{N} =  |z| - \frac{(2n+1)|z| }{N} \]
and the result follows from Lemma~\ref{lem:assep}.
\end{proof}

\begin{lemma}[Almost sure asymptotics for $\sigma$]
\label{lem:asforsigma}
Denote by $\sigma_{n}^{1}$ the largest value among $n$ i.i.d.\ copies of $\sigma(0)$. Then, under Assumption \ref{assump:sigma}, for any $c > 1$, as $n \to \infty$,
$$ g_\sigma(\sigma_{n}^1) \le  c \log n $$
eventually almost surely.
\end{lemma}
\begin{proof}
By \cite[Theorem 3.5.1]{Embrechts97} we have equivalence of the statements
\[ \left\{ \PP \left( g_\sigma(\sigma_{n}^{1}) \le   c \log n \text{ ev.}\right)=1 \right\}  \quad \text{and} \quad  \Big\{\sum_{n=1}^\infty  \PP\big( g_\sigma( \sigma(0) ) >  c \log n \big) < \infty\Big\} \, .
\]
The proof is complete by noticing that, since $g_\sigma$ is continuous by Assumption \ref{assump:sigma}, the random variable $\bar F_\sigma( \sigma(0))$ is uniformly distributed over $(0, 1)$. Hence, for any $c > 1$,
\[
 \sum_{n=1}^\infty \PP\big(g_\sigma(\sigma(0)) > c \log n \big) = \sum_{n=1}^\infty n^{-c}<\infty \, .\qedhere\]
\end{proof}

\subsection{Existence of quick paths}
In this section we prove the existence of paths $p \in \Gamma(0, z)$ for certain $z \in V_t$ that have the property that (i) all $\sigma(p_i)$ are relatively small, and (ii) $p$ is not much longer than a direct path to $z$; what we mean by `relatively small' and `not much longer' will depend on the dimension. We shall informally refer to such paths as \emph{quick paths}. The reason we are interested in quick paths is that they are intimately related to the probability that a particle undertaking the BAM reaches a certain site $z$ by time $t$.

In dimension higher than one, we will additionally require that such paths do not travel too close to a certain well-separated set $S_t$. The reason for this additional requirement is that we will eventually seek to apply our results to the site $Z_t$, which depends in a complicated way on $\sigma(z)$ for $z \in B(\Pi^{(L_t)}, \rho)$. We will wish to avoid this dependence, hence our insistence on the fact that the paths do not travel too close to $S_t$.

\subsubsection{Dimension one}
In dimension one, there is only one shortest path from $0$ to $z$ and this must pass through all intermediate sites. Hence, we seek to show that not too many traps on this path are too large. Clearly, the ability to do this depends on the tail decay of $\sigma$, which is the origin of the extra tail decay condition for $d = 1$ in Assumption \ref{assump:sigma}.

To proceed, we must undertake a rather delicate analysis of the trapping landscape $\sigma$ in the region between $0$ and $z$. We simplify this using coarse graining, essentially placing each site $y$ into a certain `bin' depending on the value of $\sigma(y)$. We then seek to bound the number of sites in each bin, weighted by the depth of the traps corresponding to each bin. To assist in the coarse graining, we state and prove a technical lemma on the regularity of the upper-tail of $\sigma(0)$.

\begin{lemma}[Regularity of the upper-tail of $\sigma(0)$]
\label{lem:sigmareg}
Under Assumption \ref{assump:sigma}, let $x_t$ be such that
$$ g_\sigma(\exp \{ \exp \{ x_t  \} \} ) = t  \, ,$$
which is well-defined by the continuity of $g_\sigma$. Then, for constants $c_1$ and $c_2$ such that $c_2 > c_1 \ge 1$, as $t \to \infty$,
$$ g_\sigma(\exp \{ \exp \{ c_2 x_t \} \} ) > c_1 t $$
eventually.
\end{lemma}
\begin{proof}
Let $c$ be the constant in part $(d)$ of Assumption \ref{assump:sigma}. In the case where $c < \infty$, for any $\varepsilon > 0$, as $t \to \infty$,
$$ t = g_\sigma(\exp \{ \exp \{ x_t \} \} )  <  x_t  \left( c + \varepsilon \right) $$
eventually. Choosing the $0 < \varepsilon < c (c_2 - c_1)/(c_1 + c_2)$, we have that, as $t \to \infty$,
$$  g_\sigma(\exp \{ \exp \{ c_2 x_t \} \} )  > c_2 x_t  \left( c - \varepsilon \right) > t \frac{c_2 ( c - \varepsilon ) }{c + \varepsilon } > c_1 t$$
eventually. On the other hand, in the case where $c = \infty$, then by Assumption \ref{assump:sigma},
$$ t = g_\sigma(\exp \{ \exp \{ x_t \} \} )  =  x_t \kappa_{x_t} $$
for some $\kappa_t \uparrow \infty$. Similarly
$$  g_\sigma(\exp \{ \exp \{ c_2 x_t \} \} )  = c_2 x_t \kappa_{c_2 x_t}  > c_1 x_t \kappa_{x_t} = c_1 t $$
eventually, which completes the proof.
\end{proof}

We now define the coarse graining scales that we will use. Let $n_t$ and $\sigma_t$ be arbitrary functions tending to $\infty$ as $t \to \infty$. 

\begin{lemma}[Existence of well-spaced coarse graining scales]
\label{lem:scales}
Let $\varepsilon < 1$ be a constant that satisfies part $(c)$ of Assumption \ref{assump:sigma}. Then there exist constants $0 < \delta_1 < \delta_2 <  \varepsilon < 1 < c_1$, an integer $I_t = O(\log \log n_t)$ and a set of scaling functions $\{\sigma^i_t\}_{0 \le i \le I_t}$ such that, as $t \to \infty$, the following are all satisfied eventually: 
\begin{enumerate}[(a)]
\item $\sigma^0_t=0 \, , \, \frac{\log\log\sigma^1_t}{\log\log \sigma_t} \in [ 1 + \delta_1, 1 + \delta_2] \, , \, \frac{\log\log\sigma^i_t}{\log\log \sigma^{i-1}_t} \in [ 1 + \delta_1, 1 + \delta_2] \ \text{for } 2 \le i \le I_t $ ; 
\item $  g_\sigma( \sigma^{I_t-1}_t) \le  c_1^{-1} \log n_t $ ; and
\item $  g_\sigma( \sigma^{I_t}_t) \ge  c_1 \log n_t   \, .$
\end{enumerate}
\end{lemma}
\begin{proof}
Choose $c_1$, $\delta_1$ and $\delta_2$ such that $1 < c_1^2 < 1 + \delta_2$ and $1 + \delta_1 < (1 + \delta_2)/c_1^2 $. Suppose that we define a sequence $\{\bar{\sigma}^i_t\}_{i \ge 0}$ such that 
$$ \bar \sigma^0_t=0 \ , \ \frac{\log\log \bar \sigma_t^1}{\log\log  \sigma_t} = 1 + \delta_1   \quad  \text{and} \quad  \frac{\log\log \bar \sigma_t^i}{\log\log \bar \sigma^{i-1}_t} = 1 + \delta_1  \quad \text{for each } i \ge 2\, ,  $$
and let $I_t$ be the maximum integer such that
$$ g_\sigma  (\bar \sigma_t^{I_t - 1}) \le  c_1^{-1} \log n_t \,  \, .$$
This satisfies
$$ I_t = O(\log \log n_t)\, ,   $$
since if $I_t>1$, then eventually
$$ (1 + \delta_1)^{I_t - 2} \log \log \bar \sigma^1_t =  \log \log \bar{\sigma}^{I_t - 1 }_t < g_\sigma(\bar{\sigma}_t^{I_t-1})  \le c_1^{-1} \log n_t \, .$$
Now set $\sigma^i_t= \bar\sigma^i_t$ for all $ 0 \le i \le I_t - 1$, and choose $\sigma_t^{I_t}$ by
$$ \begin{cases}  \log\log  \sigma^{I_t}_t = (1 + \delta_2 )\log\log \sigma^{I_t-1}_t\,, & I_t > 1\,; \\
 \log\log  \sigma^{I_t}_t = (1 + \delta_2 )\log\log \sigma_t\,,  & I_t = 1 \,.
 \end{cases} $$ 
It remains to check that $g_\sigma( \sigma^{I_t}_t) \ge  c_1 \log n_t  $. By definition,  
$$  \log \log \sigma^{I_t}_t =  \frac{1 + \delta_2}{1 + \delta_1} \log \log \bar \sigma^{I_t}_t \, .$$
Then by Lemma~\ref{lem:sigmareg}, and the fact that $1 + \delta_1 < (1 + \delta_2)/c_1^2 $, as $t \to \infty$,
$$ g_\sigma( \sigma^{I_t}_t ) > c_1^2 g_\sigma( \bar \sigma^{I_t}_t) $$
eventually. Finally, by the definition of $I_t$, 
$$ g_\sigma( \bar \sigma^{I_t}_t)  > c_1^{-1} \log n_t $$
which completes the proof.
\end{proof}

Finally, we prove the existence of a quick path. Let $c_1$, $\delta_1$, $\delta_2$, $I_t$ and $\{\sigma^i_t\}_{0 \le i \le I_t}$ satisfy the conditions in Lemma~\ref{lem:scales}. Moreover, for a path $p \in \Gamma_k$ define
\[  N_i = \sum_{ 0 \le j < k } \id_{\{\sigma(p_j)\in(\sigma^{i-1}_t,\sigma^i_t]\}} \quad  \text{for each } 1\le i\le I_t  \, . \] 
The following proposition essentially bounds the number of sites in each coarse graining scale, weighted by the log of the scale. This will turn out to be the correct definition of a `quick path' in Section \ref{sec:neg}.

\begin{proposition}[Existence of quick paths; $d = 1$]
\label{prop:quickpath1}
As $t \to \infty$, each path $p \in \Gamma_{|z|}(0, z)$ with $|z| < n_t$, satisfies
\begin{align*}
\PP \left( \sum_{i=1}^{I_t} N_{i} \log \sigma^i_t < n_t  \max\left\{ (\log \sigma_t)^2, \log \log n_t / \log \log \sigma_t  \right\}  \right) \to 1 \, 
\end{align*}
and
 $$\max_{0 \le i < |z|} \sigma(p_i) < \sigma^{I_t}_t\, ,$$
 eventually almost surely.
\end{proposition}
\begin{proof}
We first prove that the event 
$$ \mathcal{N}_t := \bigcup_{i=1}^{I_t  }\left\{ N_i \le 2 n_t \, \bar F_\sigma(\sigma^{i-1}_t )\right\} $$
satisfies $\PP ( \mathcal{N}_t ) \to 1$ as $t \to \infty$. Note that each $N_i$ is stochastically dominated by
$$\bar{N_i} \stackrel{d}{=} \text{Binom}(n_t, \bar F_\sigma(\sigma^{i-1}_t ) ) \, ,$$
with $\EE \bar N_i = n_t \bar F_\sigma(\sigma^{i-1}_t ) $ and $\text{Var} \bar N_i\le n_t \bar F_\sigma(\sigma^{i-1}_t ) $. By the union bound and Chebyshev's inequality,
\begin{align}\label{E:unionNi}
 \PP \Big(\bigcup_i \{\bar N_i > 2\EE \bar N_i\}\Big)\le\sum_i\PP(\bar N_i>2\EE \bar N_i) \le \sum_i\frac{\text{Var} \bar N_i}{(\EE \bar N_i)^2} \le \sum_i  \big(n_t \bar F_\sigma(\sigma^{i-1}_t )  \big)^{-1}\,.
\end{align}
Since the $\sigma^i_t$ are increasing in $i$, for any $1 \le i \le I_t$,
\begin{align*}
 \bar F_\sigma (\sigma^{i-1}_t ) \ge \bar F_\sigma(\sigma^{I_t-1}_t ) \ge n_t^{-c_1^{-1}} \,, 
\end{align*}
by condition $(b)$ of Lemma~\ref{lem:scales}. Combining with \eqref{E:unionNi}, by the union bound, as $t \to \infty$, eventually
\[  \PP (\mathcal{N}_t ) > 1 -  I_t \, n_t ^{c_1^{-1}-1} \to  1 \, ,\]
since $c_1 > 1$ and $I_t = O(\log \log n_t)$. 

So assume the event $\mathcal{N}_t$ holds and split the sum 
\[ \sum_{i=1}^{I_t} N_{i} \log \sigma^i_t  =  N_1 \log\sigma^1_t+\sum_{i=2}^{I_t} N_i \log\sigma^i_t 
 \, .\]
For the first term, on the event $\mathcal{N}_t$ and by condition $(a)$ in Lemma~\ref{lem:scales} we have
\begin{align*}
N_1 \log\sigma^1_t & \le 2 n_t \bar F_\sigma (\sigma_t^0) \log\sigma_t^1 =  2 n_t \log \sigma_t^1 \le 2 n_t (\log \sigma_t)^{1+\delta_2} < n_t (\log \sigma_t)^2 / 2
\end{align*}
eventually. Hence it suffices to show that each of the other terms, for $2\le i\le I_t$, satisfy
$$ I_t N_i \log\sigma^i_t < \frac{1}{2} n_t \log \log n_t / \log \log \sigma_t $$
eventually. Recall that by condition $(a)$ in Lemma~\ref{lem:scales}, $\log\sigma^i_t \le (\log \sigma^{i-1}_t)^{1+\delta_2}$ for $2\le i\le I_t$. Then, on the event $\mathcal{N}_t$ and by part $(c)$ of Assumption \ref{assump:sigma}, eventually,
\begin{align*}
  N_i \log \sigma^i_t & \le 2 n_t \bar F_\sigma ( \sigma^{i-1}_t ) \log \sigma^i_t \le 2 n_t (\log\sigma^{i-1}_t)^{- \varepsilon + \delta_2} \\
&\le n_t(\log\sigma^{i-1}_t)^{-c_2}\,,
\end{align*}
for some $c_2 > 0$, since $\delta_2 < \varepsilon$. So by monotonicity in $i$ and condition $(a)$ in Lemma~\ref{lem:scales},
\begin{align*}
 I_t N_i \log \sigma^i_t \le I_t n_t (\log\sigma^{1}_t)^{-c_2} < n_t \log \log n_t  (\log \sigma_t)^{-c_3 } 
\end{align*}
eventually, for any $0 <c_3 <c_2(1+\delta_1)$ which proves the claim.

Finally, the fact that, eventually almost surely,
\[ \max_{0 \le i < |z| } \sigma(p_i) < \sigma^{I_t}_t  \]
follows from combining condition $(c)$ in Lemma~\ref{lem:scales} with Lemma~\ref{lem:asforsigma}.
\end{proof}

\subsubsection{Dimension higher than one}
In dimensions higher than one we use percolation-type estimates to prove the existence of a path $p \in \Gamma(0, z)$ with $z \in S_t$ for some well-separated set $S_t$ that (i) avoids all the deep traps, (ii) has $|p|$ not much more than $|z|$, and (iii) does not travel too close to sites in $S_t$. Because we use percolation-type arguments, it will turn out that we need no extra assumption on the tail decay of $\sigma(0)$.

So let us start with the relevant percolation-type estimates; for background on percolation theory see \cite{Grimmett99}. Consider site percolation on $\mathbb{Z}^d$ with \mbox{$\PP(v\text{ open})=q$} independently for every $v\in\mathbb{Z}^d$. We say that a subset of $\mathbb{Z}^d$ is $\ast$-connected if it is connected with respect to the adjacency relation 
\[  v\overset{\ast}{\sim}w\Leftrightarrow\max_{1\le i\le d}|v_i-w_i|=1 \, ,\]
where $v_i$ and $w_i$ denote the coordinate projections of $v$ and $w$ respectively. If $v\overset{\ast}{\sim}w$ we say that $w$ is a $\ast$-neighbour of $v$. A $\ast$-connected subset of $\mathbb{Z}^d$ is referred to as a $\ast$-cluster. The relevance of $\ast$-clusters is that they represent the blocking clusters for open paths in $\ZZ^d$. For $v\in\mathbb{Z}^d$ a closed site, denote by $\mathcal{C}(v)$ the largest $\ast$-cluster of closed sites containing $v$.

For two sites $u,v$ in $\mathbb{Z}^d$ denote by $d_\infty(u,v)$ their chemical distance (also known as the graph distance) with respect to site percolation, defined to be the length of the shortest open path from $u$ to $v$ (and defined to be infinite if no such path exists).

\begin{lemma}[Expected size and maximum of closed $\ast$-clusters]
\label{lem:clustersize}
Let $q\in(1-(3d)^{-1},1)$ and suppose $u_1,\ldots,u_M$ are $M \in \mathbb{N}$ distinct closed sites in $\mathbb{Z}^d$. Then
\begin{enumerate}[(i)]
\item $\EE[|\mathcal{C}(u_1)|]\le (1-3^d(1-q))^{-1}$, and so in particular $\EE[|\mathcal{C}(u_1)|]\to1$ as $q\to1$; and
\item For every $x\in\mathbb{N}$, $$\PP(\max\{|\mathcal{C}(u_1)|,\ldots,|\mathcal{C}(u_M)|\}<x)\ge 1-M(3^d(1-q))^{[\log_{3^d} x]} \, .$$
\end{enumerate}
\end{lemma}
\begin{proof}
Consider performing a breadth-first search on $\mathcal{C}(u_1)$ starting from the site $u_1$, by first discovering the closed $\ast$-neighbours $v_1,\ldots,v_k$ of $u_1$, and then in turn discovering the closed $\ast$-neighbours of each of the $v_j$, $1\le j\le k$, iterating this procedure to explore $\mathcal{C}(u_1)$. Suppose that the site $w$ has just been explored in this procedure. Then the number of closed $\ast$-neighbours of $w$ that have not already been discovered is stochastically dominated by a Binom($3^d-1,1-q$) random variable. It follows that $|\mathcal{C}(u_1)|$ is stochastically dominated by the total progeny of a branching process with offspring distribution Binom($3^d,1-q$). Since the expected total progeny of this branching process is $(1-3^d(1-q))^{-1}$, this proves the first statement.

For the second statement, note that by the union bound we have 
\[ \PP(\max\{|\mathcal{C}(u_1)|,\ldots,|\mathcal{C}(u_M)|\}\ge x)\le\sum_{i=1}^M\PP(|\mathcal{C}(u_i)|\ge x)=M \, \PP(|\mathcal{C}(u_1)|\ge x) \, . \] 
Again by exploring $\mathcal{C}(u_1)$ we have 
\[
\PP(|\mathcal{C}(u_1)|\ge x)\le\PP(Z\ge x) \, ,
\]
where $Z$ is the total progeny of a branching process with offspring distribution Binom($3^d,1-q$). To complete the proof, note that by Markov's inequality we have 
\begin{equation*}
\PP(Z\ge x)\le\PP(Z(\lfloor \log_{3^d} x\rfloor)>0)\le (3^d(1-q))^{[\log_{3^d} x]} \, ,
\end{equation*}
where $Z(n)$ denotes number of individuals in generation $n$ of the branching process. 
\end{proof}

\begin{lemma}[Chemical distance]
\label{lem:chemdist}
Fix two sites $u,v$ in $\mathbb{Z}^d$ and a function $c:=c(q)$ with $c\to\infty$ as $q\to1$. Then, as $q \to 1$,
 \[  \PP\left(\frac{d_\infty(u,v)}{|u-v|}<1+c(1-q)\right)\to1 \, .\]
\end{lemma}
\begin{proof}
Denote by $\mathcal{C}_\infty$ the unique infinite open cluster, which exists almost surely for all $q$ sufficiently close to 1 (see \cite{Grimmett99}). Let $\hat p\in\Gamma_{|u-v|}(u,v)$ be any shortest path, denote by $K$ the subset of $\mathrm{Set}(\hat p)$ consisting only of closed sites, and define  \begin{align}\label{E:Sbound}
 S:=\Big|\bigcup_{x\in K}\mathcal{C}(x) \Big|  \le \sum_{x\in K}|\mathcal{C}(x)|  \, .                                                                                                                   \end{align}
By part $(i)$ of Lemma~\ref{lem:clustersize} and the FKG inequality (see \cite{Grimmett99}, Section 2.2), we have the bound
\[
\EE[S|\{u,v\in\mathcal{C}_\infty\}]\le \frac{\EE[ \, |K|\big|\{u,v\in\mathcal{C}_\infty\}]}{1-3^d(1-q)}\le\frac{|u-v|(1-q)}{1-3^d(1-q)}.
\] 
We now claim that, on the event $\{u,v\in\mathcal{C}_\infty\}$, it is possible to find a path $p\in\Gamma_k(u,v)$ for some $k\le|u-v|+(3^d-1)S$ such that every site in $\mathrm{Set}(p)$ is open. To obtain the required path $p$ take the direct path $\hat p$ and divert it around $\mathcal{C}(u)$ for each closed $u \in \mathrm{Set}(\hat p)$, so that every site in $\mathrm{Set}(p)$ is either in $\mathrm{Set}(\hat p)$ or in the outer boundary of some $\mathcal{C}(u)$, where by outer boundary we mean the set of sites $\{v\notin \mathcal{C}(u):\exists u\in \mathcal{C}(u),\, u\overset{\ast}{\sim}v\}$.  This procedure is possible since $u,v\in\mathcal{C}_\infty$. Then $\mathrm{Set}(p)$ will consist of just open sites since the outer boundary of each $\mathcal{C}(u)$ is a path of open sites. The bound on $|p|$ follows from the fact that the size of the outer boundary of a $\ast$-cluster $A$ is at most $(3^d-1)|A|$.
 
We complete the proof of the Lemma with Markov's inequality:
 \begin{align*}  
 \PP\left(\frac{d_\infty(u,v)}{|u-v|}\ge1+c(1-q)\right) &\le \PP\left(|S|>\frac{c(1-q)|u-v|}{(3^d-1)} \bigg| \{u,v\in\mathcal{C}_\infty\} \right) +\PP\left(\{u,v\in\mathcal{C}_\infty\}^c\right) \\
 & \le\frac{3^d}{c(1-3^d(1-q))} +\PP\left(\{u,v\in\mathcal{C}_\infty\}^c\right)\, . 
 \end{align*}
Since $\PP(u,v\in\mathcal{C}_\infty)\to1$ as as $q\to1$, this completes the proof.\qedhere
\end{proof}

We are now ready to show the existence of a quick path in dimensions higher than one. Let $S_t \subseteq \ZZ^d$ be such that 
\[ \text{sep}(S_t) >t^{\varepsilon} \quad \text{and} \quad \min_{u\in S_t}{|u|}>t^{\varepsilon}  \]
eventually for some $\varepsilon > 0$. Recall the definition of $j := [\gamma - 1]$. Let $\sigma_t$ be an arbitrary function tending to infinity as $t \to \infty$. Define the set
$$\ZZ^d(\sigma_t, S_t) := \{ z \in \ZZ^d: \sigma(z) \le \sigma_t , z \notin B(S_t, j)   \} \, . $$
For a site $z \in \mathbb{Z}^d$, let $|z|_{\rm{chem}}$ be the ‘chemical distance’ of the ball $B(z , j)$ in this set, that is, the length of the shortest path from the origin to $\partial B(z, j)$ that lies exclusively in this subgraph (setting it as $\infty$ if such a path does not exist).

\begin{proposition}[Existence of quick paths; $d > 1$]
\label{prop:quickpath2}
Let $z_t \in S_t \cap V_t$ and let $c_t$ be a function such that $c_t\to\infty$ as $t\to \infty$ and $\bar F_\sigma(\sigma_t)c_t \ll 1$. Then, there exists a constant $c  > 0$ such that, as $t \to \infty$,
\[ \PP \left( \frac{|z_t|_{\rm{chem}}}{|z_t|} \le  1 + \bar F_\sigma(\sigma_t) c_t + t^{-c}  \right) \to 1 \, .\]
\end{proposition}
\begin{proof}
Let $q := 1-\bar F_\sigma(\sigma_t)$. By Lemma~\ref{lem:chemdist}, with probability tending to 1 as $t\to\infty$ there exists a path $\hat p\in \Gamma_{\ell_t}(0,z_t)$ for some 
\[ \ell_t \le |z_t|(1+ \bar F_\sigma(\sigma_t) c_t  )\]
such that $\sigma(\hat p_i) \le \sigma_t$ for all $0\le i<\ell_t$. Let $i=\min\{0\le j<\ell_t:\,\hat p_j \in\partial B(z_t,j)\}$ and define $v_t := \hat p_i$ to be the first site in $\partial B(z_t,j)$ visited by path $\hat p$. We show how to modify $\hat p$ so that we obtain a new path $p\in\Gamma(0,v_t)$ for some $v_t\in\partial B(Z_t,j)$ with $\mathrm{Set}(p)\subseteq \ZZ^d(\sigma_t, S_t)$. 

The modification is done by diverting $\hat p$ around the balls of radius $j$ centred on sites in~$S_t$. In doing so, we may encounter new closed sites $v$, and these too must be avoided if we wish to find a path $p$ with $\mathrm{Set}(p)\subseteq\ZZ^d(\sigma_t, S_t)$. Formally, the set of these new closed sites is precisely
 \[
 \left\{x\in\partial B\left(S_t\cap B(\mathrm{Set}(\hat p), j), j\right):\,\sigma(x)>\sigma_t\right\}.
 \]
 Denote by $M_t$ the size of this set and its elements as $w_1,\ldots,w_{M_t}$, and choose $0<c_1<\varepsilon$ where $\varepsilon$ is the constant appearing in the definition of~$S_t$. Then by the separation of sites in~$S_t$, we have  
\[
|S_t\cap B(\mathrm{Set}(\hat p),j)| \le 2 \ell_t t^{-\varepsilon},
\]
and so 
\begin{align}
\label{eq:M}
M_t \le 3^d|B(0,j)|\ell_t t^{-\varepsilon}<|z_t| t^{-c_1}
\end{align}
for all $t$ sufficiently large. Choose now $0<c_2<c_1$, $\alpha<-1-(1-c_1)/c_2$, and $t$ sufficiently large so that 
$$\bar F_\sigma(\sigma_t) < 3^{d\alpha} \, .$$ 
Applying part $(ii)$ of Lemma~\ref{lem:clustersize}, we deduce that $$\max\{|\mathcal{C}(w_1)|,\ldots,|\mathcal{C}(w_{M_t})|\}\le t^{c_2}$$ with probability tending to $1$ as $t\to\infty$. We claim this implies that, by the separation of sites in $S_t$ and the fact that $c_2<\varepsilon$, with overwhelming probability there exists a path $p\in\Gamma(0,v_t)$ which avoids all $j$-balls centred on sites in $S_t$ and all closed sites. Indeed to obtain this path we take path $\hat p$ and then divert around $j$-balls centred on sites in $S_t$ and then further divert around any new closed $\ast$-clusters we encounter. Since we know that no such cluster is too large, they cannot cut the origin off from $v_t$ in $\ZZ^d(\sigma_t, S_t)$, and furthermore we will not encounter any more sites in $S_t$ on the new path.

We can now bound $|p|$. The number of additional sites we must visit to obtain $p$ from $\hat p$ is at most $3^d M_t(|B(0,j)|+t^{c_2})$ with probability tending to 1 as $t\to\infty$; this comes from counting the diversions around each $j$-ball and the diversions around each closed cluster we then encounter. Using \eqref{eq:M}, we can thus choose $0<c<c_1-c_2$ to yield the result.
\end{proof}

\section{Extremal theory for local eigenvalues}
\label{sec:extremal}
In this section, we use point process techniques to study the random variables $Z_t^{(j)}$ and $\Psi^{(j)}_t(Z_t^{(j)})$, and generalisations thereof; the techniques used are similar to those found in \cite{Astrauskas08, Fiodorov13, Sidorova12}, although we strengthen the results available in those papers. In the process, we complete the proof of Theorems~\ref{thm:main2} and \ref{thm:main3}. Throughout this section, let $\varepsilon$ be such that $0 < \varepsilon < \theta$.

\subsection{Upper-tail properties of the local principal eigenvalues}
The first step is to give upper-tail asymptotics for the distribution of the local principal eigenvalues ${\lambda}^{(n)}(z)$ for $z\in\Pi^{(L_t)}$ and $n \in \mathbb{N}$. These will allow us to study the random variables $Z_{t}^{(j)}$ and ${\Psi}^{(j)}_{t}(Z_{t}^{(j)})$ via point process techniques. For technical reasons, we shall actually consider an i.i.d.\ set of \emph{punctured} versions of $\lambda^{(n)}(0)$ which can be coupled to coincide with $\lambda^{(n)}(z)$ eventually almost surely for each $z\in\Pi^{(L_t)}$. 

To this end, let $\{\tilde{\xi}_z\}_{z \in V_t}$ be an independent collection of i.i.d.\ potential fields $\tilde{\xi}_z:\mathbb{Z}^d\to\mathbb{R}$ distributed as
\begin{align}\label{tildefields}
\tilde\xi_z(y)\stackrel{d}{=}\begin{cases}\xi(0) \, ,&\mbox{if }y=z \, ;\\
\xi(0)\indic{\xi(0)<L_t} \, ,&\mbox{otherwise} \, ,
\end{cases}
\end{align}
and define
  \[\tilde\Pi^{(L_t)}:=\{z\in V_t:\,\tilde\xi_z(z)>L_t\} \quad \text{and} \quad  \tilde\Pi^{(L_{t, \varepsilon})}:=\{z\in V_t:\,\tilde\xi_z(z)> L_{t, \varepsilon}\} \]
  by analogy with $\Pi^{(L_t)}$ and $\Pi^{(L_{t, \varepsilon})}$.
Similarly, let $\{\tilde{\sigma}_z\}_{z \in V_t}$ be an independent collection of i.i.d.\ trapping landscapes $\tilde{\sigma}_z:\mathbb{Z}^d\to\mathbb{R}$ distributed as $\tilde{\sigma}_z(y) \stackrel{d}{=} \sigma(0)$, and abbreviate $\tilde{\eta}_z(z) = \tilde{\xi}_z(z) - \tilde{\sigma}^{-1}_z(z)$. Then, for each $z \in V_t$ and $n \in \mathbb{N}$, let $\tilde{\lambda}^{(n)}(z)$ be the principal eigenvalue of the \emph{punctured} Hamiltonian
 \[ \tilde{\mathcal{H}}^{(n)}(z) := \Delta \tilde{\sigma}^{-1}_z  + \tilde \xi_z  \]
 restricted to the domain $B(z,n)$ with Dirichlet boundary conditions, and observe that $\{\tilde{\lambda}^{(n)}(z)\}_{z \in V_t}$ are i.i.d.\ by construction. Note that we have suppressed the explicit $t$-dependence in $\tilde \xi_z$ (and hence $\tilde \lambda^{(n)}(z)$ etc.). 
 
 We now proceed to study the upper-tail asymptotics for the distribution of each $\tilde\lambda^{(n)}(z)$.

\begin{proposition}[Path expansion for $\tilde{\lambda}^{(n)}$]
\label{prop:pathexp}
For each $n \in \mathbb{N}$ and $z \in \tilde\Pi^{(L_{t, \varepsilon})}$ uniformly, as $t \to \infty$,
\begin{align*}
\tilde{\lambda}^{(n)}(z) &=  \tilde\eta_z(z) + \tilde{\sigma}^{-1}_z(z) \sum_{2 \leq k \leq 2j} \sum_{\substack{ p \in \Gamma_{k}(z, z) \\ p_i \neq z , \, 0 < i < k  \\ \mathrm{Set}(p) \subseteq B(z, n) }} \prod_{0 < i < k} (2d)^{-1} \frac{ \tilde{\sigma}_z^{-1}(p_i)}{\tilde{\lambda}^{(n)}(z) - \tilde{\eta}_z(p_i)} + o(d_t e_t) \,  , \\
& =  \tilde\eta_z(z) + O(a_t^{-1}) \, .  
\end{align*}
\end{proposition}

\begin{proof}
Applying Proposition~\ref{prop:genpathexp} we have that
\[ \tilde{\lambda}^{(n)}(z) = \tilde\eta_z(z) + \tilde{\sigma}^{-1}_z(z)   \sum_{k \geq 2} \sum_{\substack{ p \in \Gamma_{k}(z, z) \\ p_i \neq z , \, 0 < i < k  \\ \mathrm{Set}(p) \subseteq B(z, n) }} \prod_{0 < i < k} (2d)^{-1}  \frac{  \tilde{\sigma}_z^{-1}(p_i) }{\tilde{\lambda}^{(n)}(z) - \tilde\eta_z(p_i)} \, .  \]
Now observe that, by Lemma \ref{lem:minmax} and the truncation in $\tilde\xi_z$, for each $p_i \in B(z, n) \setminus \{z\}$,
$$ \tilde{\lambda}^{(n)}(z) - \tilde\eta_z(p_i) > L_{t, \varepsilon} - L_t - \delta_\sigma^{-1} \sim (\theta - \varepsilon) a_t \, . $$
Moreover, each $\tilde{\sigma}_z^{-1}(p_i)$ is bounded above by $\delta^{-1}_\sigma$. Finally, as $t \to \infty$,
$$ a_t ^{-(2j + 2)} = o(d_t e_t) \, ,$$
by the definition of $j$. This means that, up to the error $o(d_t e_t)$, we can truncate the sum at paths with $2j$ steps. It also means that the total contribution from the sum over paths $p \in \Gamma_k(z, z)$ is $O(a_t^{-1})$.
\end{proof}

\begin{proposition}[Extremal theory for $\tilde{\lambda}^{(n)}$;  see {\cite[Section 6]{Astrauskas08}, \cite[Proposition~4.2]{Fiodorov13}}]
\label{prop:asympt}
For each $n \in \mathbb{N}$, there exists a scaling function $A_t = a_t + O(1)$ such that, as $t \to \infty$ and for each fixed $x \in \mathbb{R}$,
\[ t^d \, \PP \left(\tilde{\lambda}^{(n)}(0) > A_t + x d_t  \right) \to e^{- x} \, .  \]
Moreover, there exists a $c > 0$ such that, as $t \to \infty$ and uniformly for $x > 0$,
\[ t^d \, \PP \left(\tilde{\lambda}^{(n)}(0) > A_t + x d_t  \right) < e^{-cx^{\min\{1, \gamma\} }} \, .  \]
\end{proposition}
\begin{proof}
To ease notation, we abbreviate the fields $\tilde\xi_0$ and $\tilde\sigma_0$ by $\xi$ and $\sigma$ respectively (although these should not be confused with the original fields $\xi$ and $\sigma$, since the law of $\tilde \xi_0(z)$ has been truncated at $z \neq 0$). First remark that, by Lemma \ref{lem:minmax}, as $t \to \infty$, 
$$ \tilde{\lambda}^{(n)}(0)> A_t + x d_t   
\quad \text{implies that} \quad \xi(0) > L_{t, \varepsilon} \,,$$
eventually, which means that we can apply the path expansion in Proposition~\ref{prop:pathexp} to $\tilde{\lambda}^{(n)}(0)$. Let $A_t$ be an arbitrary scale such that $A_t = a_t + O(1)$, and define the function
\[ Q(A_t; \xi, \sigma) :=   \invsigma(0) + \invsigma(0) \sum_{2 \leq k \leq 2j} \sum_{\substack{ p \in \Gamma_{k}(0, 0) \\ p_i \neq 0 , \, 0 < i < k  \\ \mathrm{Set}(p) \subseteq B(z, j) }} \prod_{0 < i < k} (2d)^{-1} \frac{\invsigma(p_i)}{A_t - \eta(p_i)}  \, , \]
if $\xi(y) < L_t$ for each $y \in B(0, j) \setminus \{0\}$ and $Q(A_t; \xi, \sigma) := 0$ otherwise. Then, since $\tilde{\lambda}^{(n)}(0)$ is strictly increasing in $\xi(0)$ we have that, as $t \to \infty$,
\begin{align}
 \nonumber & \PP \left( \tilde{\lambda}^{(n)}(0)  >  A_t + x d_t  \right)  \sim  \PP \bigg( \xi(0) >  A_t + x d_t + Q(A_t + x d_t; \xi, \sigma)  \bigg)  \\
\label{eq:intrep1}  &  \qquad \sim  \PP \bigg( \xi(0) >  A_t + x d_t + Q(A_t; \xi, \sigma)  \bigg)  \\
 \label{eq:intrep4} &  \qquad \sim t^{-d} e^{-x} \int_{\xi, \sigma} \exp \bigg\{ a_t^\gamma - \bigg( A_t + Q(A_t; \xi, \sigma) \bigg)^\gamma \bigg\} \, d \mu_\xi \, d \mu_\sigma 
\end{align}
where the first asymptotic accounts for the error in the path expansion Proposition~\ref{prop:pathexp}, the second and third asymptotics result from Taylor expansions, and are uniform in $\xi$ and $\sigma$, and where $\mu_\xi$ and~$\mu_\sigma$ stand for the joint probability densities of $\{\xi(y)\}_{y \in B(0, n) \setminus \{0\}}$ and $\{\sigma(y)\}_{y \in B(0, n)}$ respectively.  Consider then the integral in \eqref{eq:intrep4}, which we abbreviate as
\[  I(A_t) :=  \int_{\xi, \sigma} f(A_t; \xi, \sigma) \, d \mu_\xi \, d \mu_\sigma  \ , \quad f(A_t; \xi, \sigma):= \exp \bigg\{ a_t^\gamma - \bigg( A_t + Q(A_t; \xi, \sigma) \bigg)^\gamma \bigg\}  \, . \]
Note that $Q(A_t; \xi, \sigma)$ is uniformly bounded as $t \to \infty$ (by $2\delta_\sigma^{-1}$ for instance) . Hence, for $C$ sufficiently large, as $t \to \infty$ eventually
\[  f( a_t + C; \xi, \sigma)  < 1 < f(a_t - C; \xi, \sigma)   \]
uniformly in $\{\xi(y)\}_{y \in B(0, n) \setminus \{0\}}$ and $\{\sigma(y)\}_{y \in B(0, n)}$. This implies that $I(A_t - C) < 1 < I(a_t + C)$. Moreover, since $f(A_t; \xi, \sigma)$ is continuous in $A_t$ uniformly for each $\{\xi(y)\}_{y \in B(0, n) \setminus \{0\}}$ and $\{\sigma(y)\}_{y \in B(0, n)}$, the function $I(A_t)$ is continuous in $A_t$. Hence, by the intermediate value theorem function, there exists an $A_t = a_t + O(1)$ such that, as $t \to \infty$ eventually $I(A_t) = 1$, which gives the first result. For the second, instead of \eqref{eq:intrep1} we bound $Q(A_t + xd_t; \xi, \sigma)$ above, uniformly in $x > 0$, by $Q(A_t; \xi, \sigma)$, which produces the bound
\[  t^{-d} \int_{\xi, \sigma} \exp \bigg\{ a_t^\gamma - \bigg( A_t + Q(A_t; \xi, \sigma) \bigg)^\gamma \bigg( 1 + \frac{x}{\gamma} (\log t)^{-1}  \bigg)^\gamma \bigg\} \, d \mu_\xi \, d \mu_\sigma\, . \]
In the case $\gamma \ge 1$, we bound this expression above uniformly in $x > 0$ by
\[  t^{-d} \int_{\xi, \sigma} \exp \bigg\{ a_t^\gamma - \bigg( A_t + Q(A_t; \xi, \sigma) \bigg)^\gamma \bigg( 1 + \frac{x}{\gamma} (\log t)^{-1}  \bigg) \bigg\} \, d \mu_\xi \, d \mu_\sigma \sim e^{- \frac{x}{\gamma}(1+o(1))} \, , \]
using the definition of $A_t$ and the fact that $A_t + Q(A_t; \xi, \sigma) \sim a_t$ in the last step. The case $\gamma < 1$ is simpler, since then we have simply
\[  \PP \bigg( \xi(0) >  A_t + x d_t + \invsigma(0)  \bigg) =   \PP \bigg( \xi(0) >  a_t + x d_t + O(1)  \bigg) \] 
and the bound follows from the regularity of Weibull tail of $\xi(0)$ in Assumption \ref{assump:xi}.
 \end{proof}
  
We now define the set-up we shall need to examine the correlation of the potential field and trapping landscape near sites of high $\tilde{\lambda}^{(n)}$; since the nature of this correlation differs depending on $(\gamma, \mu)$, so does our set-up. Fix a constant $\nu \in (0, 1)$. Recalling the definition of the `interface cases' $\mathcal{B}$ and $\mathcal{B}_\xi$, define the `interface sites'
\[ \mathcal{F} :=  \begin{cases}
  y \in \mathbb{Z}^d : |y| = \rho\,, & \text{if } (\gamma, \mu) \in \mathcal{B}\, , \\
\varnothing\,,   & \text{else}\, , \\
\end{cases} \qquad \text{and} \qquad
 \mathcal{F}_\xi :=  \begin{cases}
  y \in \mathbb{Z}^d : |y| = \rho_\xi\,, & \text{if } (\gamma, \mu) \in \mathcal{B}_\xi\, , \\
\varnothing\,,   & \text{else}\,. \\
\end{cases} \,  \]
Recalling the definition of $n(y)$, for each $y \in \mathbb{Z}^d$ define the positive constants
\[   c_\sigma :=  \begin{cases} 
 \left(\frac{\gamma}{\mu} \right)^{\frac{1}{\mu+1}}  \, , & \text{if } q_\sigma > 0 \, ,\\
 0 \, ,  & \text{else} \, ,
\end{cases} \ , \quad c_\xi(y) :=  \begin{cases} 
 \left( n(y)^2 (2d)^{-1} \delta^{-1}_\sigma c^{-1}_\sigma \right)^{\frac{1}{\gamma - 1}} \, ,  & \text{if } q_\xi(|y|) > 0 \, ,\\
 0 \, , & \text{else} \, ,
\end{cases}  \, ,\] 
 \[  \bar c_\sigma(y) := n(y)^2 (2d)^{-1} \gamma c_\sigma^{-1}    \qquad \text{and} \qquad   \bar c_\xi(y) := \bar c_\sigma(y) \, \delta_\sigma^{-1} \, . \] 
For each $n \in \mathbb{N}$, if $\mu > 0$ and $\gamma > 1$, define the rectangles
\[ E_\xi :=  \prod_{y \in (B(0, n \wedge \rho_\xi) \setminus \{0\}) \setminus \mathcal{F}_\xi }  (-f_t, f_t) \ \times \prod_{y \in (B(0, n) \setminus B(0, n \wedge \rho_\xi)) \cup \mathcal{F}_\xi} (f_t, g_t)  \, ,  \]
\[  E_\sigma :=   ( - f_t,  f_t) \ \times \ \prod_{y \in (B(0, n) \setminus \{0\}) \setminus \mathcal{F} } (0, f_t) \ \times \prod_{y \in (B(0, n) \setminus B(0, n \wedge \rho )) \cup \mathcal{F} } (0,  g_t)  \, , \]
\[ S_\xi := \prod_{y \in (B(0, n \wedge \rho_\xi) \setminus \{0\}) \setminus \mathcal{F}_\xi }  a_t^{q_\xi(|y|)} (c_\xi(y) - f_t, c_\xi(y) + f_t)  \ \times \prod_{y \in (B(0, n) \setminus B(0, n \wedge \rho_\xi)) \cup \mathcal{F}_\xi} (f_t, g_t)  \, ,   \]
and
\[ S_\sigma := a_t^{q_\sigma}(c_\sigma  - f_t, c_\sigma  + f_t) \ \times \ \prod_{y \in (B(0, n) \setminus \{0\}) \setminus \mathcal{F}  }   (\delta_\sigma,\delta_\sigma+ f_t) \ \times \prod_{y \in (B(0, n) \setminus B(0, n \wedge \rho )) \cup \mathcal{F} } (0,  g_t) \, . \]
If $\mu = 0 $ and $\gamma > 1$, define instead 
\[ E_\sigma :=  (a_t^{-\nu}, \infty) \ \times \prod_{y \in B(0, n) \setminus \{0\} } (0, g_t) \quad \text{and} \quad S_\sigma :=  a_t^{\gamma - 1}(a_t^{-\nu}, \infty) \ \times \ \prod_{y \in B(0, n) \setminus \{0\} }  (0,  g_t) \, , \]
whereas if $\gamma \le 1$, maintain the definition of $E_\sigma$ but define instead
\[ S_\sigma :=  \prod_{y \in B(0, n) }  (0,  g_t) \, . \]
For each $n \in \mathbb{N}$ and $z \in \mathbb{Z}^d$, define the event
\begin{align*}
& \mathcal{S}^{(n)}_t(z) := \left\{ \{\tilde\xi_z(y)  \}_{y \in B(z, n) \setminus \{z\}} \in S_\xi  \ , \ \{\tilde\sigma_z(y) \}_{y \in B(z, n)} \in S_\sigma \right\} \, ,
\end{align*}
and, for each $x \in \mathbb{R}$ and the scaling function  $A_t$ from Proposition~\ref{prop:asympt}, further define the event
\[  \mathcal{A}_t := \left\{ \tilde{\lambda}^{(n)}(0) > A_t  + x d_t  \right\} \, .\]

\begin{proposition}[Correlation of potential field and trapping landscape]
\label{prop:corr}
For each $n \in \mathbb{N}$, as $t \to \infty$,
$$  \PP  \left(   \mathcal{S}^{(n)}_t(0) \big| \mathcal{A}_t \right) \to 1 \, . $$
Moreover, as $t \to \infty$,
\begin{align}
\label{eq:thresxi}
f_{\tilde{\xi}_0(y)|\mathcal{A}_t}(x)  \to \frac{ e^{\bar c_\xi(y) x} f_\xi(x)} { \EE[ e^{\bar c_\xi(y) \xi(0)}  ]} \ , \quad \text{for each } y \in \mathcal{F}_\xi , 
\end{align}
uniformly over $x \in (0, L_t)$, and
\begin{align}
\label{eq:thressigma2}
f_{\tilde{\sigma}_0(y)|\mathcal{A}_t}(x)   \to \frac{e^{\bar c_\sigma(y) /x}f_{\sigma}(x)}{\EE[e^{\bar c_\sigma(y) / \sigma(0)} ]} \ , \quad \text{for each } y \in \mathcal{F}   ,
\end{align}
uniformly over $x$. Finally, if $\gamma = 1$, then for each $x \in \mathbb{R}^+$, as $t \to \infty$,
\begin{align}
\label{eq:thressigma1}
 f_{\tilde\sigma_0(0)|\mathcal{A}_t}(x)   \to \frac{ e^{ -1/x} f_\sigma(x) }{\EE[ e^{-1/\sigma(0)} ]} \, ,
 \end{align}
 uniformly over $x$.
\end{proposition}
\begin{proof}
Once again, we abbreviate $\tilde\xi_0$ and $\tilde\sigma_0$ as $\xi$ and $\sigma$ respectively. Define a field $s : B(0, n) \setminus \{0\} \cup B(0, n) \to \mathbb{R}$ with projections $s_\xi$ and $s_\sigma$ onto $B(0, n) \setminus \{0\}$ and $B(0, n)$ respectively. For a scale $C_t \sim a_t $ define the function
\begin{align*} & Q_t(C_t; s) :=  a_t^{-q_\sigma} (c_\sigma +  s_\sigma(0))^{-1} -  a_t^{-q_\sigma} (c_\sigma +  s_\sigma(0))^{-1} \\
 & \qquad \times \sum_{2 \leq k \leq 2j}   \sum_{\substack{ p \in \Gamma_{k}(0, 0) \\ p_i \neq 0  , \, 0 < i < k  \\ \mathrm{Set}(p) \subseteq B(0, n) }}  \prod_{0 < i < k} \frac{(2d)^{-1} \, ( \delta_\sigma +  s_\sigma(p_i))^{-1}}{C_t - a_t^{q_\xi(|p_i|)} (c_\xi(p_i) + s_\xi(p_i)) + (\delta_\sigma + s_\sigma(p_i))^{-1}} \,,
 \end{align*}
if, for each $y \in B(0, n) \setminus \{0\}$,
\[ a_t^{q_\xi(|y|)} (c_\xi(y) + s_\xi(y)) \in (0, L_t) \ , \quad  s_\sigma(y) > 0 \quad \text{and} \quad a_t^{q_\sigma} (c_\sigma + s_\sigma(0) ) > 0 \]
are satisfied, and $Q_t(C_t; s) := 0$ otherwise. Define further the function
\begin{align*}
 R_t(C_t; s) & :=  a_t^\gamma - (C_t + Q_t(C_t; s))^\gamma + \sum_{y \in B(0, n) } \left( \log f_\xi \left(a_t^{q_\xi(|y|)} (c_\xi(y) + s_\xi(y)) \right) + \log a_t^{q_\xi(|y|)} \right) \\
 & \qquad  +  \log f_\sigma \left(a_t^{q_\sigma}(c_\sigma + s_\sigma(0) \right) + \log a_t^{q_\sigma}  + \sum_{y \in B(0, n) \setminus\{0\} }  \log f_\sigma \left(\delta_\sigma +  s_\sigma(y) \right)  \, .
\end{align*}
To motivate these definitions, consider that, similarly to the above, we can write
\begin{align}
\label{eq:intrep2}
 \PP \left( \tilde{\lambda}^{(n)}(0)  >  A_t + x d_t  \right)  \sim  t^{-d} e^{-x} \int_{\mathbb{R}^{2|B(0, n)| - 1}}  \exp \left\{ R_t(A_t; s)  \right\} \, d s \, .
\end{align}
It remains to show that the integral in \eqref{eq:intrep2} is asymptotically concentrated on the set $E_\xi \times E_\sigma$ and that equations \eqref{eq:thresxi}--\eqref{eq:thressigma1} are satisfied. This fact can be checked by a somewhat lengthy computation which we only sketch here. We shall treat separately three cases: (i) $\gamma > 1$ and $\mu > 0$; (ii) $\gamma > 1$ and $\mu = 0$; and (iii) $\gamma \le 1$. We begin with case (i), which is the most delicate. 

We must analyse the variables $s_\sigma(0)$, $\{s_\sigma(y)\}_{y \in B(0, n) \setminus \{0\}}$, and $\{s_\xi(y)\}_{y \in B(0, n) \setminus \{0\}}$ separately; we start with $s_\sigma(0)$. In what follows abbreviate $R_t(A_t; s)$ by $R_t(s)$. Fix an arbitrary choice of the components of $s$ and consider how $R_t(s)$ varies with $s_\sigma(0)$. Notice that the function $R_t(s)$ can be decomposed into two parts, one of which decreases as $s_\sigma(0)$ increases (through $Q_t$) and another which increases as $s_\sigma(0)$ increases (through $f_\sigma$). The first part is analysed by Taylor expanding $(A_t + Q_t(A_t;s))^\gamma$, from which it can be seen that the dependence on $s_\sigma(0)$ is, as $t \to \infty$,
\[ \gamma  \, a_t^{-q_\sigma}\, a_t^{\gamma-1} \, (c_\sigma +  s_\sigma(0))^{-1} \,  \left( 1 + o(1) \right) \] 
where the error term $o(1)$ is uniform in $s$. The second part is given by $- \log f_\sigma (a_t^{q_\sigma} (c_\sigma + s_\sigma(0)))$ which is, if $\mu > 0$, eventually
\[   a_t^{q_\sigma \mu} (c_\sigma + s_\sigma(0))^\mu   \, .\]
Hence, since we defined $q_\sigma$ precisely so that
\[-q_\sigma + \gamma - 1 = q_\sigma \mu \, ,\]
the function $R_t$ has the asymptotic form, as $t \to \infty$,
\[  R_t(s) = f_1(t; s) + a_t^{\kappa_1} \left( g_1(s_\sigma(0)  + o(1) \right) \]
where $f_1(t; s)$ is some function not depending on $s_\sigma(0)$, $\kappa_1$ is some positive constant, the function $g_1(x)$ satisfies
\[  g_1(x) := - \gamma (c_\sigma + x)^{-1} - (c_\sigma + x)^\mu  \, ,\]
and the error term $o(1)$ is uniform in $s$. Then we have, uniformly in $s$, as $t \to \infty$,
\begin{align}
\label{eq:intrep3}
 \int_{\mathbb{R}} e^{R_t(s)} \, d s_\sigma(0) \sim e^{f_1(t;s)} \int_{\mathbb{R}} \exp \left\{ a_t^{\kappa_1} g_1(s_\sigma(0))  \right\} \, d s_\sigma(0) \, .
 \end{align}
Remark that $g_1(x)$ achieves a unique maximum at $0$ (by the construction of $c_\sigma$). Therefore, by the Laplace method, the above integral is eventually asymptotically concentrated around $0$ on the order $a_t^{\kappa_1}$, and hence the integral is concentrated on the domain $s_\sigma(0) \in (-f_t, f_t)$.

Consider now the variables $\{s_\sigma(y)\}_{y \in B(0, n) \setminus \{0\}}$. Fix an $s_\sigma(0) \in (-f_t, f_t)$ and an arbitrary choice of the remaining components of $s$. Again, similarly to the above, the function $R_t(s)$ can be decomposed into two parts, one whose dependence on $s_\sigma(y)$ is, as $t \to \infty$,
\[  n(y)^2 \, (2d)^{-1} \, \gamma  \, c_\sigma^{-1} \,  a_t^{\gamma - 2|y|}  \, a_t^{-q_\sigma} \, (\delta_\sigma + s_\sigma(y))^{-1}  \, \left(1 + o(1) \right) \,  \] 
uniformly in $s$, and another whose dependence is
\[  - \log  f_\sigma(\delta_\sigma + s_\sigma(y) ) \, . \]
Then we have, uniformly in $s$, as $t \to \infty$,
\[  \int_{\mathbb{R}} e^{R_t(s)} \, d s_\sigma(y) \sim e^{f_2(t;s)} \int_{\mathbb{R}} \exp \left\{  \gamma c_\sigma^{-1} a_t^{\kappa_2} (\delta_\sigma + s_\sigma(y))^{-1}   \right\}  f_\sigma(\delta_\sigma + s_\sigma(y) )  \, d s_\sigma(y) \,  , \]
where $f_2(t; s)$ is some function not depending on $s_\xi(y)$, $\kappa_2$ is some non-negative constant with $\kappa_2 > 0$ if and only if $y \in B(0, \rho) \setminus \mathcal{F}$,  and where the error term $o(1)$ is uniform in $s$. Hence, if $y \in B(0, \rho) \setminus \mathcal{F}$, then along with the lower-tail assumption in \ref{assump:sigma}, it is clear that the above integral is asymptotically concentrated on $s_\sigma(y) \in (0, f_t)$. On the other hand, if $y \in \mathcal{F}$, then the integrand is asymptotically
\[ e^{ \bar c_\sigma(y) / (s_\sigma(y) + \delta_\sigma)} f_\sigma(s_\sigma(y) + \delta_\sigma) \, , \]
uniformly over $s_\sigma(y)$, which establishes \eqref{eq:thressigma2}. Trivially, if $y \notin B(0, \rho)$, then the integral is concentrated on $s_\sigma(y) \in (f_t, g_t)$.

Finally, consider the variables $\{s_\xi(y)\}_{y \in B(0, n) \setminus \{0\}}$ and fix $s_\sigma(0) \in (-f_t, f_t)$, $s_\sigma(y) \in (0, f_t)$ for each $y \in B(0, \rho) \setminus \mathcal{F}$, and an arbitrary choice of the remaining components of $s$.  The function $R_t(s)$ can be decomposed into two parts, one whose dependence on $s_\xi(y)$ is of order, as $t \to \infty$,
\[ n(y)^2 \, (2d)^{-1} (\delta_\sigma + s_\sigma(y))^{-1} \gamma c_\sigma^{-1} \, a_t^{q_\xi(|y|)}  \,  a_t^{\gamma - 1 - 2|y|}  a_t^{-q_\sigma}\, (c_\xi(y) + s_\xi(y)) \left( 1 + o(1) \right) \, ,\] 
uniformly in $s$, another whose dependence is 
\[   a_t^{q_\xi(|y|) \gamma} (c_\xi(y) + s_\xi(y))^\gamma     \, .\]
Hence, since we defined $q_\xi(|y|)$ precisely so that
\[  q_\xi(|y|) + \gamma - 1 - 2|y| - q_\sigma = q_\xi(|y|) \gamma \, , \]
if $y \in B(0, \rho_\xi)$, the function $R_t$ has the asymptotic form, as $t \to \infty$,
\[ R_t(s) = f_3(t; s)  + a_t^{\kappa_3} \left( g_3(s_\xi(y))  + o(1) \right)  \]
where $f_3(t; s)$ is some function not depending on $s_\xi(y)$, $\kappa_3$ is some non-negative constant with $\kappa_3 > 0$ if any only if $y \in B(0, \rho_\xi) \setminus \mathcal{F}_\xi$, the function $g_3(x)$ satisfies
\[ g_3(x) :=  \gamma \, n(y)^2 \, (2d)^{-1} \, \delta^{-1}_\sigma \, c_\sigma^{-1}  (c_\xi(y) + x) - (c_\xi(y) + x)^\gamma  \, , \]
 and where the error term $o(1)$ is uniform in $s$. Then we have, uniformly in $s$, as $t \to \infty$,
\[  \int_{\mathbb{R}} e^{R_t(s)} \, d s_\xi(y) \sim e^{f_3(t;s)} \int_{\mathbb{R}} \exp \left\{ a_t^{\kappa_3} g_3(s_\xi(y))  \right\} \, d s_\xi(y) \, .  \]
If $y \in B(0, \rho_\xi) \setminus \mathcal{F}_\xi$, and since $g_3(x)$ achieves a unique maximum at $0$ (by the construction of $c_\xi(y)$), again by the Laplace method this integral is also asymptotically concentrated on $s_\xi(y) \in  (-f_t, f_t)$. On the other hand, if $y \in \mathcal{F}_\xi$, then the integrand  is asymptotically
\[ e^{ \bar c_\xi(y) s_\xi(y)} f_\xi(s_\xi(y)) \, , \]
uniformly over $s_\xi(y)$, which establishes \eqref{eq:thresxi}. Trivially, if $y \notin B(0, \rho_\xi)$, then the integral is concentrated on $s_\xi(y) \in (f_t, g_t)$. Since we have now shown that each component of \eqref{eq:intrep2} is asymptotically concentrated on the respective component of the set $E_\xi \times E_\sigma$, integrating first over $s_\xi(y)$ and $s_\sigma(y)$ for $y \in B(0, n) \setminus \{0\}$, and then over $s_\sigma(0)$, we have the result.

We now turn to case (ii). In this case the integral over $s_\sigma(0)$ in \eqref{eq:intrep3} becomes
\begin{align*}
e^{f_1(t; s)} \int_{\mathbb{R}}  e^{ -\gamma s^{-1}_\sigma(0)}  f_\sigma \left( A_t^{\gamma - 1} s_\sigma(0) \right) \, d s_\sigma(0)  \sim e^{f_1(t; s)} \int_{\mathbb{R}}  e^{ - \gamma  s^{-1}_\sigma(0) }    f_\sigma \left( a_t^{\gamma - 1} s_\sigma(0) \right) \, d s_\sigma(0) \, ,
\end{align*}
where we used the regularity in \ref{assump:sigma} in the last step. On the region $(0, a_t^{-\nu})$, this integral can be bounded above as
\[ \int_0^{a_t^{-\nu}}  e^{ - \gamma  s^{-1}_\sigma(0)}    f_\sigma \left( a_t^{\gamma - 1} s_\sigma(0) \right) \, ds_\sigma(0) \le  \int_0^{a_t^{-\nu}}  e^{ - \gamma  s^{-1}_\sigma(0) }    \, ds_\sigma(0) \le  e^{- \gamma a_t^{\nu}} \, . \] 
On the other hand, for any $0 < c < \nu$, the integral is bounded below by
\[  \int_{a_t^{-c}}^\infty  e^{ - \gamma  s^{-1}_\sigma(0) }    f_\sigma \left( a_t^{\gamma - 1} s_\sigma(0) \right) \, ds_\sigma(0)   \ge e^{- \gamma a_t^{c}} \bar F_\sigma \left( a_t^{\gamma - 1 - c} \right)  \gg e^{ - \gamma a_t^{\nu}} \]
with the final asymptotic following since $\mu = 0$. Hence the integral in \eqref{eq:intrep3} is asymptotically concentrated on $s_\sigma(0) \in (a_t^{-\nu}, \infty)$. Finally, notice that for fixed $s_\sigma(0) \in (a_t^{-\nu}, \infty)$ we have that, as $t \to \infty$,
\[ Q_t(A_t; s) = a_t^{1 - \gamma} s_\sigma^{-1}(0) + o(d_t)\,  \]
since $\nu < 1$, with the error uniform in $s_\sigma(0)$. Hence, for $s_\sigma(0) \in (a_t^{-\nu}, \infty)$, as $t \to \infty$,
\[\exp \{ R_t(A_t; s) \} \sim t^{-d} a_t^{\gamma - 1}  e^{-\gamma s_\sigma^{-1}(0)}  \prod_{s_\xi} f_\xi(s_\xi) \prod_{s_\sigma} f_\sigma(s_\sigma) \,    \]
and so the integral in \eqref{eq:intrep2} is asymptotically concentrated on $E_\xi \times E_\sigma$.

Case (iii) is easier to handle. Now the integral in \eqref{eq:intrep3} becomes
\[ e^{f_1(t; s)} \int_{\mathbb{R}}  \exp\{ - \gamma a_t^{\gamma - 1} s^{-1}_\sigma(0) + o(1) \} \, f_\sigma(s_\sigma(0)) \, d s_\sigma(0) \, , \]
with the error uniform in $s$. If $\gamma < 1$, then this integral is clearly concentrated on $s_\sigma(0) \in (0, g_t)$. If $\gamma = 1$, then the integrand of this integral is asymptotically
\[  e^{ s^{-1}_\sigma(0)} f_\sigma(s_\sigma(0)) \, , \]
uniformly over $s_\sigma(0)$, which establishes \eqref{eq:thressigma1}. The remainder of the proof is identical.
\end{proof}

\subsection{Constructing the point process}
The existence of asymptotics for the (punctured) local principal eigenvalues allows us to establish scaling limits for the penalisation functional $\Psi^{(j)}_t$. We start by constructing a point set from the pair $(z, \Psi^{(j)}_t(z))$ which will converge to a point process in the limit. Here we make use of the length scale $r_t$ defined in equation \eqref{eq:rtat}, and the first and second order scales $A_t$ and $d_t$ for the extremes of the local principal eigenvalues defined in Proposition~\ref{prop:asympt} and equation \eqref{eq:dt} respectively. Note that the appropriate rescaling functions for the point set are actually $A_{r_t}$ and $d_{r_t}$, although since $d_{r_t} \sim d_t$ we shall eventually end up substituting these.

For technical reasons, we shall actually need to consider a certain generalisation of the functional~$\Psi^{(j)}_t$. More precisely, for each $c \in \mathbb{R}$, define the functional $ \Psi^{(j)}_{t, c}: V_t \to \mathbb{R}$ by
$$ {\Psi}^{(j)}_{t, c}(z) := {\lambda}^{(j)}(z) - \frac{|z|}{\gamma t} \log \log t + c \frac{|z|}{t} \, .$$
For each $z \in \Pi^{(L_t)}$ define
$$ Y^{(j)}_{t, c, z} := \frac{ {\Psi}^{(j)}_{t, c}(z) - A_{r_t} }{d_{r_t}}  \qquad \text{and} \qquad \mathcal{M}^{(j)}_{t, c} := \sum_{z \in \Pi^{(L_t)}} \id_{(z r_t^{-1} , Y^{(j)}_{t, c, z})}  \, .$$ 
Finally, for each $\tau \in \mathbb{R}$ and $\alpha > -1$ let  
$$ \hat{H}_\tau^\alpha := \{ (x, y) \in \dot{\mathbb{R}}^{d+1} : y \geq \alpha|x| + \tau \} $$
where $\dot{\mathbb{R}}^{d+1}$ is the one-point compactification of $\mathbb{R}^{d+1}$.

\begin{proposition}[Point process convergence]
\label{prop:pp1}
For each $\tau, c \in \mathbb{R}$ and $\alpha > -1$, as $t \to \infty$,
$$ \mathcal{M}^{(j)}_{t, c}|_{\hat{H}_\tau^\alpha} \Rightarrow \mathcal{M} \quad \text{in law}\,,$$
where $\mathcal{M}$ is a Poisson point process on $\hat{H}_{\tau}^\alpha$ with intensity measure $ \nu(dx, dy) = dx \otimes e^{- y - |x|} dy$.
\end{proposition}
\begin{proof}
The idea of the proof is to replace the set $\{\lambda^{(j)}(z)\}_{z \in \Pi^{(L_t)}}$ with the set of i.i.d.\ punctured principal eigenvalues $\{\tilde{\lambda}^{(j)}\}_{z \in V_t}$ and then apply standard results in i.i.d.\ extreme value theory to show convergence to $\mathcal{M}$ in  $\hat{H}_\tau^\alpha$.

To this end, define $\tilde{\Psi}^{(j)}_{t, c}(z)$ and $\tilde{Y}^{(j)}_{t, c, z}$  equivalently to $\Psi^{(j)}_{t, c}(z)$ and  $Y^{(j)}_{t, c, z}$ after replacing $\lambda^{(j)}(z)$ everywhere with $\tilde{\lambda}^{(j)}(z)$ and further define $$\tilde{\mathcal{M}}^{(j)}_{t, c}=\sum_{v\in V_t}\id_{(z r_t^{-1},\tilde{Y}_{t,c,z}^{(j)})}.$$ Recall that $\{\tilde{\lambda}^{(j)}_t\}_{z \in V_t}$ are i.i.d.\ with tail asymptotics and uniform tail decay governed by Proposition~\ref{prop:asympt}. By applying an identical argument as in \cite[Lemma~3.1]{Sidorova12} and \cite[Lemma 4.3]{vanDerHofstad06}, we have that, as $t \to \infty$,  
\[ \tilde{\mathcal{M}}^{(j)}_{t, c} \big|_{\hat{H}_\tau^\alpha} \Rightarrow \mathcal{M} \quad \text{in law} \, .\]
Note that the uniform tail decay is necessary for the point process convergence to hold since $\hat{H}_\tau^\alpha$ is a non-compact set (see \cite[Lemma 4.3]{vanDerHofstad06}). We claim that if $z \in V_t$ is such that
\[ (z r_t^{-1} , \tilde{Y}^{(j)}_{t, c, z}) \in \hat{H}_\tau^\alpha  \, ,\]
then $z \in \tilde\Pi^{(L_t)}$ eventually almost surely. This is since $(z r_t^{-1} , \tilde{Y}^{(j)}_{t, c, z}) \in \hat{H}_\tau^\alpha$ is equivalent to
\[ \tilde{\lambda}^{(j)}(z) \ge A_{r_t} + \frac{\alpha |z| d_{r_t}}{r_t} + \frac{|z|}{\gamma t} \log \log t - \frac{c |z|}{t} + \tau d_{r_t} \]
which implies that, as $t \to \infty$,
\begin{align*}
\tilde{\lambda}^{(j)}(z) & \ge a_t (1 + o(1)) + (\alpha + 1 + o(1))\frac{|z|}{\gamma t} \log \log t + O(d_t)  \\
& \ge a_t( 1 + o(1)) + O(d_t) 
\end{align*}
since $A_{r_t} \sim a_{r_t} \sim a_t$, $d_{r_t} \sim d_t$ and $\alpha > -1$. The claim then follows by the upper bound in Lemma~\ref{lem:minmax}. As a consequence, we have that, as $t \to \infty$,
\begin{equation}
\label{eq:pplaw}
 \sum_{z \in \tilde\Pi^{(L_t)}} \id_{( z r_t^{-1}  , \tilde{Y}^{(j)}_{t, c, z} )}  \big|_{\hat{H}_\tau^\alpha} \Rightarrow \mathcal{M} \quad \text{in law} \, . 
\end{equation}

To complete the proof we show how to construct a coupling, valid eventually almost surely, with the property that 
  \begin{align}
\label{eq:lambdatilde} 
\left\{ \lambda^{(j)}(z) \right\}_{z \in \Pi^{(L_t)}} = \left\{ \tilde{\lambda}^{(j)}(z) \right\}_{z \in \tilde\Pi^{(L_t)}} .
\end{align}
To do this, note that by Lemma~\ref{lem:assep} there almost surely exists a $t_0$ such that, for all $t>t_0$, we have $r(\Pi^{(L_t)})>2j$ and, for each $z \in \Pi^{(L_{t})}$ and $ y\in B(z,j)\setminus\{z\}$, it holds that $\xi(y) < L_t$. For such $t$ we define the coupling as follows: for $z \in \Pi^{(L_t)}$ and $y\in B(z,j)$ set $\tilde\xi_z(y)=\xi(y)$ and $\tilde\sigma_z(y) = \sigma(y)$; for $z\notin\Pi^{(L_{t})}$ independently sample $\tilde\xi_z(z)$ as $\xi(0)$ conditioned on $\xi(0)< L_t$; and otherwise independently sample $\tilde\xi_z(y)$ and $\tilde\sigma_z(y)$ according to their respective law. By the separation guaranteed for $t > t_0$, $\{\tilde\xi_z\}_{z\in V_t}$ and $\{\tilde\sigma_z\}_{z\in V_t}$ are indeed i.i.d.\ fields with law as in \eqref{tildefields}, and moreover~\eqref{eq:lambdatilde} holds by construction. Combining with \eqref{eq:pplaw} completes the proof.
\end{proof}
\begin{remark}
Although we state Proposition~\ref{prop:pp1} for arbitrary $c \in \mathbb{R}$, we shall only apply it to $c = 0$ and one other value of $c$ that will be determined in Section \ref{sec:neg}.
\end{remark}

We now use the point process $\mathcal{M}$ to analyse the joint distribution of top two statistics of the functional $\Psi^{(j)}_{t, c}$. So let 
\[ Z_{t, c}^{(j)} := \argmax_{z \in \Pi^{(L_t)}} \Psi^{(j)}_{t, c}(z) \quad \text{and} \quad Z_{t, c}^{(j, 2)} := \argmax_{\substack{z \in \Pi^{(L_t)} \\ z \neq Z_{t, c}^{(j)}}} \Psi^{(j)}_{t, c} \, . \]
Note that eventually these are well-defined almost surely, since $\Pi^{(L_t)}$ is finite and non-zero by Lemma~\ref{lem:asforxi}. 

\begin{corollary}
 \label{cor:limit}
For each $c \in \mathbb{R}$, as $t \to \infty$,
$$ \left(  \frac{Z_{t, c}^{(j)}}{r_t}, \frac{Z_{t, c}^{(j, 2)}}{r_t},  \frac{ {\Psi}^{(j)}_{t, c}(Z_{t, c}^{(j)}) - A_{r_t}}{d_{r_t}} , \frac{{\Psi}^{(j)}_{t, c} (Z_{t, c}^{(j, 2)}) - A_{r_t}}{d_{r_t}}  \right)  $$
converges in law to a random vector with density
$$ p(x_1, x_2, y_1, y_2) =  \exp \{- (y_1 + y_2) - |x_1| - |x_2|) - 2^d e^{- y_2} \} \id_{\{y_1 > y_2\}} \, .$$
\end{corollary}
\begin{proof}
This follows from the point process density in Proposition~\ref{prop:pp1} using the same computation as in \cite[Proposition~3.2]{Sidorova12}.
\end{proof}

\subsection{Properties of the localisation site} 
In this subsection we use the results from the previous subsection to analyse the localisation sites $Z_{t, c}^{(j)}$ and $Z_t$, and in the process complete the proof of Theorems \ref{thm:main2} and \ref{thm:main3}. For each $c \in \mathbb{R}$, introduce the events
$$\mathcal{G}_{t,c} := \{{\Psi}^{(j)}_{t, c}(Z_{t, c}^{(j)}) - {\Psi}^{(j)}_{t, c}(Z_{t, c}^{(j, 2)}) > d_t e_t \} \, ,$$
$$ \mathcal{H}_t := \{  r_t f_t < |Z_{t}^{(j)}| < r_t g_t \} \quad \text{and} \quad \mathcal{I}_t := \{  a_t(1-f_t) < \Psi_t^{(j)}(Z_{t}^{(j)}) < a_t(1+f_t) \} \, , $$
and the event
\begin{align}
\label{eq:defevent}
\mathcal{E}_{t,c} :=  \mathcal{S}^{(j)}_t(Z_{t}^{(j)})  \cap \mathcal{G}_{t,0} \cap \mathcal{G}_{t,c} \cap \mathcal{H}_t \cap \mathcal{I}_t
\end{align}
which acts to collect the relevant information that we shall later need.

\begin{proposition} \label{prop:event}
 For each $c \in \mathbb{R}$, as $t \to \infty$,
\begin{align*}
\PP(\mathcal{E}_{t,c}) \to 1 \, .
\end{align*}
\end{proposition}
\begin{proof}
This follows from Proposition~\ref{prop:asympt}, Corollary \ref{cor:limit}, the coupling in the proof of Proposition~\ref{prop:pp1}, and since $A_{r_t} \sim a_t$ and $d_{r_t} \sim d_t$.
\end{proof}

In the next few propositions, we prove that the sites $Z_{t,c}^{(j)}$ and $Z_{t}^{(j)}$ are both equal to the localisation site $Z_t$ with overwhelming probability. 
\begin{proposition}
\label{prop:zeqzc}
For each $c \in \mathbb{R}$, on the event $\mathcal{E}_{t,c}$, as $t \to \infty$,
$$ Z_{t,c}^{(j)} = Z_{t}^{(j)} $$
holds eventually.
\end{proposition}
\begin{proof}
Assume that $Z^{(j)}_{t, c} \neq Z_{t}^{(j)}$ and recall that $1/ \log \log t < e_t / g_t$ eventually by \eqref{scalingfns}. On the event $\mathcal{E}_{t,c}$, the statements
$$ {\Psi}^{(j)}_{t} (Z^{(j)}_{t}) - {\Psi}^{(j)}_{t} (Z_{t, c}^{(j)}) > d_t e_t  \quad \text{and} \quad {\Psi}^{(j)}_{t, c} (Z^{(j)}_{t,c}) - {\Psi}^{(j)}_{t, c} (Z_{t}^{(j)}) > d_t e_t $$
and, eventually,
$$|{\Psi}^{(j)}_{t}(Z_{t}^{(j)}) - {\Psi}^{(j)}_{t, c}(Z_{t}^{(j)})| = |c|\frac{|Z_{t}^{(j)}|}{t} < \gamma \frac{d_t g_t}{\log \log t} < d_t e_t $$
all hold, giving a contradiction.
\end{proof}

\begin{lemma}
\label{lem:jtorho}
For each $c \in \mathbb{R}$, on the event $\mathcal{E}_{t,c}$, as $t \to \infty$,
$$ \lambda^{(j)}(Z_{t}^{(j)}) \ge \lambda(Z_{t}^{(j)}) \quad \text{and} \quad \lambda^{(j)}(Z_t) \geq \lambda(Z_t)  $$
and
$$ \lambda^{(j)}(Z_{t}^{(j)}) - \lambda(Z_{t}^{(j)})  < d_t e_t  $$
all hold eventually.
\end{lemma}
\begin{proof}
The first two statements follow from the domain monotonicity of the principal eigenvalue in Lemma~\ref{lem:mono}. For the third statement, remark that the event $\mathcal{E}_{t,c}$ implies that $Z_t^{(j)} \in \Pi^{(L_{t, \varepsilon})}$, that $\xi(y) < L_t$ for all $y \in B(Z_t^{(j)}, \rho)$, that $\xi(y) < g_t$ for all $y$ such that $j \ge |y - Z_t^{(j)}| > \rho_\xi$, and that $\sigma(Z_t^{(j)}) > a_t^{q_\sigma} f_t$. Hence, by considering the path expansion in Proposition~\ref{prop:pathexp} (valid by the coupling in the proof of Proposition \ref{prop:pp1}), we have that for some $C > 0$,
\begin{align} 
\label{eq:jtorho}
\lambda^{(j)}(Z_{t}^{(j)}) - \lambda(Z_{t}^{(j)})  <  \frac{C a_t^{ -  (\frac{\gamma - 1}{\mu + 1})^+  } g_t }{f_t (L_{t, \epsilon} - L_t)^{2 \rho + 1}}  < d_t e_t
\end{align}
eventually, with the last equality holding since 
\begin{equation}
\label{eq:rho}
-2\rho - 1 - \left( \frac{\gamma - 1}{\mu + 1} \right)^+ < 1 - \gamma \, .  \qedhere
\end{equation}
\end{proof}
\begin{remark}
Note that $\rho$ is precisely the smallest integer such that \eqref{eq:rho} holds.
\end{remark}

\begin{corollary}[Equivalence of $ Z_{t}^{(j)}$ and $Z_t$]
\label{corry:zeqz0}
For each $c \in \mathbb{R}$, on the event $\mathcal{E}_{t,c}$, as $t \to \infty$,
$$ Z_{t}^{(j)} = Z_t$$
eventually.
\end{corollary}
\begin{proof}
Assume that $Z^{(j)}_{t} \neq Z_{t}$. On the event $\mathcal{E}_{t,c}$, Lemma~\ref{lem:jtorho} implies that
\begin{align*}
&  \left({\Psi}^{(j)}_{t}(Z_{t}^{(j)}) - \Psi_t(Z_t^{(j)}) \right) - \left( {\Psi}^{(j)}_{t}(Z_t) - \Psi_{t}(Z_t) \right) \\
& \quad = \left( {\lambda}^{(j)}(Z_{t}^{(j)}) - {\lambda}(Z_{t}^{(j)}) \right) - \left( {\lambda}^{(j)}(Z_{t}) - {\lambda}(Z_{t}) \right) < d_t e_t  
\end{align*}
holds eventually. On the other hand, on the event $\mathcal{E}_{t,c}$, and by the definition of $Z_{t}$ and $Z^{(j)}_{t}$ as the argmax of ${\Psi}_t$ and $\Psi_t^{(j)}$ respectively,
$$ {\Psi}_t^{(j)}(Z^{(j)}_{t}) - {\Psi}_t^{(j)}(Z_t) > d_t e_t \quad \text{and} \quad {\Psi}_t(Z_t) - {\Psi}_t(Z^{(j)}_{t}) > 0  $$
also hold, giving a contradiction.
\end{proof}

Finally, we prove a criterion for the independence of $Z_t$ from the trapping landscape $\sigma$. Define $\psi_t(z):= \xi(z) - \frac{|z|}{\gamma t} \log \log t$, and let $z_t := \argmax_{z \in \Pi^{(L_t)}}  \psi_t(z)$. Note that $z_t$ is independent of $\sigma$.

\begin{proposition}[Criterion for the independence of $Z_t$ from the trapping landscape $\sigma$]
\label{prop:critforind}
If $\gamma < 1$, then as $t \to \infty$,
\[ \PP \left( Z_t = z_t  \right) \to 1 \, .\]
\end{proposition}
\begin{proof}
By Proposition~\ref{prop:event} we may assume $\mathcal{E}_{t, c}$ holds. Observe that, on $\mathcal{E}_{t,c}$ and by Proposition~\ref{prop:pathexp}, any $z \in \Pi^{(L_{t, \varepsilon)}} \setminus \{Z^{(j)}_t\}$ satisfies
\[ \psi_t(Z^{(j)}_t) > \Psi^{(j)}_t(Z^{(j)}_t) > \Psi^{(j)}_t(z) + d_t e_t > \psi_t(z) + O(1) + d_t e_t \ .
\]
Moreover, by Lemma~\ref{lem:minmax} and on $\mathcal{E}_{t,c}$, any $z \in \Pi^{(L_t)} \setminus \Pi^{(L_{t, \varepsilon)}}$ also satisfies
\[ \psi_t(Z^{(j)}_t) > \Psi^{(j)}_t(Z^{(j)}_t) > \psi_t(z) + O(1) + d_t e_t \ . \]
Since $d_t e_t \to \infty$ if $\gamma < 1$, this implies that $Z^{(j)}_t = \argmax_{z \in \Pi^{(L_t)}} \psi_t(z) =: z_t$. Corollary \ref{corry:zeqz0} completes the proof.
\end{proof}

\subsection{Proof of Theorem \ref{thm:main2}}
We prove Theorem \ref{thm:main2} on the event $\mathcal{E}_{t,c}$, since by Proposition~\ref{prop:event} this event holds with overwhelming probability eventually. 
Part~\eqref{thm:main2a} is implied directly by the definition of the event $\mathcal{E}_{t,c}$. Parts~\eqref{thm:main2b}--\eqref{thm:main2d} follow by combining the definition the event $\mathcal{E}_{t,c}$ with Proposition~\ref{prop:corr} and applying the coupling in the proof of Proposition \ref{prop:pp1}. Finally, part~\eqref{thm:main2e} is a consequence of the point process convergence, and is proved in an identical manner to the corresponding results in \cite{Fiodorov13, Sidorova12}.

\subsection{Proof of Theorem \ref{thm:main3}}
Consider parts \eqref{thm:main3a} and \eqref{thm:main3b}. By definition, $Z_t$ depends only on the values of $\xi$ and $\sigma$ in balls of radius $\rho_\xi$ and $\rho$ respectively around each site, and so the radii $\rho_\xi$ and $\rho$ are certainly sufficient. To show necessity, consider that the results in parts~\eqref{thm:main2b}--\eqref{thm:main2d} of Theorem~\ref{thm:main2} establish the correlation of the fields $\xi$ and $\sigma$ at a distance $\rho_\xi$ and $\rho$ respectively around $Z_t$. Hence these radii are necessary as well. 

Consider then part \eqref{thm:main3c}. The sufficient condition for the reduction to $\xi$ follows directly from Proposition~\ref{prop:critforind}. To show necessity, consider that the results in part~\eqref{thm:main2c} of Theorem~\ref{thm:main2} establish that, if $\gamma \ge 1$, the value of $\sigma(Z_t)$ is not an independent copy of $\sigma(0)$, and hence $Z_t$ must depend on $\sigma$.

It remains to prove part~\eqref{thm:main3d}. If $\rho = 0$ then $Z_t$ depends only on $\eta$ by definition. On the other hand, suppose $\rho \ge 1$ and, for the purposes of contradiction, that there exists a random site $z_t$, depending only on $\xi$ and $\sigma$ through $\eta$, such that, as $t \to \infty$,
\[ \PP(Z_t = z_t) \to 1 \, .\]
 Fix a site $y$ and a constant $c>\delta_\sigma$. We establish a contradiction by considering two bounds on the probability of the event 
\[
\{\sigma(y)<c,\,|Z_t-y|=1\}.
\]
We first consider the case $(\gamma,\mu)\notin\mathcal{B}_\sigma$.
Then by part \eqref{thm:main2d} of Theorem \ref{thm:main2}, conditionally on event $\{|Z_t-y|=1\}$, we have that $\sigma(y)\to\delta_\sigma$ in probability as $t\to\infty$. This implies that there exists some $c_1 > 0$ such that
\begin{align}
\label{eq:sigmabound}
\PP(\sigma(y)<c,\,|Z_t-y|=1)> \left( \PP(\sigma(y)<c) + c_1 \right) \,\PP(|Z_t-y|=1)
\end{align}
eventually. In the case $(\gamma,\mu)\in\mathcal{B}_\sigma$, conditionally on event $\{|Z_t-y|=1\}$ and again by part~\eqref{thm:main2d} of Theorem \ref{thm:main2},
\[ f_{\sigma(y)}(x) \to  c_2 e^{\bar c_\sigma /x} f_\sigma(x) \, \]
for some $c_2 > 0$, and so \eqref{eq:sigmabound} holds in this case as well.

We now work on the event $\{Z_t = z_t\}$ and show how to obtain a lower bound on the probability of the event $\{\sigma(y)<c,\,|z_t-y|=1\}$. Let $\bar\eta=\{\eta(v):\,v\neq y\}$. Remark first that, since $z_t \in \Pi^{(L_t)}$, by Proposition~\ref{prop:pathexp} we have that $\lambda_t(z_t)$ is increasing in $\eta(y)$ for $|y-z_t|=1$. Hence there exists a function $\beta_t: \bar\eta \to \mathbb{R} \cup \{\infty\}$ such that, conditionally on $\bar \eta$,
 \[ \{   |z_t - y| = 1  \}   \quad \text{and}  \quad \{ \eta(y) \ge \beta_t(\bar \eta)   \}\quad \]
 agree almost surely. To see this, set $\beta_t(\bar\eta)$ to be the minimum $\eta(y)$ such that with such a value of $\eta(y)$, we have $|z_t-y|=1$ (and setting it to be infinity if no such value exists). Then clearly, if $\eta(y)<\beta_t(\bar\eta)$ we cannot have $|z_t-y|=1$, and on the other hand we claim that if $\eta(y)\ge\beta_t(y)$ we have $|z_t-y|=1$. This follows by the almost-sure separation of Lemma~\ref{lem:assep}, which ensures that $\{y=z_t\}$ has probability 0.  Denote by $\mathcal{F}_{\bar \eta}$ the $\sigma$-algebra generated by $\bar\eta$. Then, eventually almost surely,
 \begin{align*}
\PP(\sigma(y)<c,\,|z_t-y|=1)&=\EE_{\bar\eta}\left[\EE[\indic{|z_t-y|=1}\indic{\sigma(y)<c}|\,\mathcal{F}_{\bar\eta}]\right]\\
&=\EE_{\bar\eta}\left[\EE[\indic{\eta(y)> \beta_t(\bar\eta)} \indic{\sigma(y)<c}|\,\mathcal{F}_{\bar\eta}]\right] \\
&\le \EE_{\bar\eta}\left[\EE[\indic{\eta(y)> \beta_t (\bar\eta)}|\,\mathcal{F}_{\bar\eta}] \, \EE[\indic{\sigma(y)<c}|\,\mathcal{F}_{\bar\eta}]\right]\\
&=\PP(\sigma(y)<c) \, \PP(|z_t-y|=1) \, ,
\end{align*}
where the second equality uses the fact that $z_t$ depends on $\sigma$ only through $\eta$, and the inequality holds since, conditionally on $\mathcal{F}_{\bar\eta}$, the events $\{\eta(y)> \beta_t(\bar\eta)\}$ and $\{\sigma(y)<c\}$ are negatively correlated. Since $z_t = Z_t$ with probability going to one, combining with \eqref{eq:sigmabound} gives the required contradiction.

\section{Negligible paths}
\label{sec:neg}
In this section we show that the contribution to the total mass $U(t)$ from the components $U^2(t)$, $U^3(t)$, $U^4(t)$ and $U^5(t)$ are all negligible. We proceed in two parts: first we prove a lower bound on the total mass $U(t)$, and then we bound from above the contribution to the total mass from each $U^i(t)$. Throughout this section, let $\varepsilon$ be such that $0 < \varepsilon < \theta$.

\subsection{Preliminaries}
We begin by proving a general result on eigenfunction decay around sites of high potential, which will be used in both the lower and upper bound. For each $z \in \Pi^{(L_{t, \varepsilon})}$, let $\varphi_1$ denote the principal eigenfunction of the Hamiltonian ${\mathcal{H}}^{(j)}(z)$. 

\begin{proposition}
\label{prop:eigdecay}
For each $z \in \Pi^{(L_{t, \varepsilon})}$ uniformly, as $t \to \infty$, almost surely
\[  \sum_{y \in B(z, j) \setminus \{z\} } \varphi_1(y) \to 0  \qquad \text{and} \qquad  \sum_{y \in B(z, j) \setminus \{z\} } \frac{ \sigma(y)^{-\frac12} \varphi_1(y) }{ || \sigma^{-\frac{1}{2}} \varphi_1 ||_{\ell_2} } \to 0 \, . \]
\end{proposition}
\begin{proof}
By Proposition~\ref{prop:genpathexpeig}, we have the path expansion
\begin{align}
\frac{\varphi_1(y)}{\varphi_1(z)} = \frac{\sigma(y)}{\sigma(z)} \ \sum_{k \ge 1} \ \sum_{\substack{ p \in \Gamma_{k}(y, z) \\  p_i \neq z , \, 0 \le i < k \\ \mathrm{Set}(p) \subseteq  B(z, j) } } \ \prod_{0 \leq i < k} (2d)^{-1} \frac{\invsigma(p_i)}{{\lambda}^{(j)}(z) - \eta(p_i)}  \ , \quad y\in B(z,j)\setminus\{z\} \, .
\end{align}
Since, by Lemmas \ref{lem:assep} and \ref{lem:minmax}, for each $y\in B(z,j)\setminus\{z\}$, almost surely
$${\lambda}^{(j)}(z) - \eta(y_i) > L_{t, \varepsilon} - L_t - \delta^{-1}_\sigma \, , $$
and moreover since $\invsigma(y) < \delta^{-1}_\sigma$ for all $y \in B(z, j)$, the result follows.
\end{proof}

\begin{corollary}[Bound on total mass of the solution]
\label{cor:tm}
For each $z \in \Pi^{(L_{t, \varepsilon})}$ uniformly and any $c > 1$, as $t \to \infty$, almost surely
\[ \EE_z \left[  e^{\int_0^t  \xi(X_s) \, ds} \indic{ \tau_{B(z, j)^c} > t} \right] < c \, e^{t \lambda^{(j)}(z)} \, \]
eventually.
\end{corollary}
\begin{proof}
This follows by combining Propositions \ref{prop:eigdecay} and \ref{prop:tm}.
\end{proof}

\subsection{Lower bound on the total mass $U(t)$}
Recall that by the discussion in Section \ref{sec:outline},  the total mass $U(t)$ can be approximated by considering both the benefit of being near a site of high potential and the probabilistic penalty from diffusing to that site. To formalise a lower bound for $U(t)$ we need a bound on both of these terms.

We begin by bounding from below the benefit to the solution from paths that start and end at a site of high potential.

\begin{lemma}\label{lem:uloweratz}
For each $z \in \Pi^{(L_{t, \varepsilon})}$ uniformly,
$$ \log u_z(t, z) \geq t {\lambda}^{(j)}(z) + o(1) \, $$
eventually almost surely.
\end{lemma}
\begin{proof}
Recall the Feynman-Kac formula for the solution $u_z(t, z)$ (see, e.g., Proposition~\ref{prop:fksln}), and note that the expectation is larger than the corresponding expectation taken only over paths that do not leave $B(z, j)$. Using Corollary \ref{cor:sb}, we then have that
$$ u_z(t, z) \ge \frac{ e^{{\lambda}^{(j)}(z) t} \invsigma(z)  \varphi^2_1(z) }{ || \sigma^{-\frac{1}{2}} \varphi_1 ||^2_{\ell_2}   }\, ,  $$
where $\varphi_1$ denotes the principal eigenfunction of the Hamiltonian ${\mathcal{H}}^{(j)}(z)$. Since the domain $B(z, j)$ is finite, the fact that the eigenfunction $\sigma^{-\frac12} \varphi_1 $ is localised at $z$ (by Proposition~\ref{prop:eigdecay}) ensures that the square eigenfunction $\invsigma \varphi_1^2$ is also localised at $z$, and the result follows.
\end{proof}

The next step is to bound from above the probabilistic penalty incurred by diffusing to a certain site. This will be a function both of the distance of the site from the origin, as well as the size of the traps on paths from the origin to the site. Here we use the existence of quick paths that we established in a general setting in Section \ref{sec:prop}.

Recall the scaling function $s_t$, which satisfies the properties in \eqref{scalingfns}. If $d = 1$, for $\sigma_t := s_t$ and $n_t := r_t g_t$, recall the definitions of $I_t$ and $\{\sigma^i_t\}$ from Proposition~\ref{prop:quickpath1}. Let $p \in \Gamma_{|Z_t|}(0, Z_t)$ be the (unique) shortest path from $0$ to $Z_t$ and define
$$ N^p_i := \sum_{0 \le l < |Z_t|} \id_{\{\sigma(p_l) \in (\sigma_t^{i-1}, \sigma_t^i] \}}  \ , \quad i = 1, \ldots, I_t \, . $$ 
If $d \ge 2$, for $z_t := Z_t$, $\sigma_t := s_t$ and $S_t := \Pi^{(L_t)}$, recall the definition of $|Z_t|_{\text{chem}}$ from Proposition~\ref{prop:quickpath2}. Denote by $\Theta^d_t$ the event
\[ \Theta^d_t := \begin{cases} 
  \left\{ \sum_{i=1}^{I_t} N^p_i \log \sigma_t^i < t d_t b_t  \, , \ \max_{0 \le l < |Z_t|} \sigma(p_l) < \sigma_t^{I_t} \right\} \, , & d = 1\,; \\
\left\{ |Z_t|_{\text{chem}} < |Z_t| + r_t b_t \right\} \, ,& d \ge 2\, .
\end{cases} \]

\begin{proposition}[Existence of quick paths]
\label{prop:existencepath}
For each $c \in \mathbb{R}$, as $t \to \infty$, 
\[\PP(\Theta^d_t, \mathcal{E}_{t,c}) \to 1 \, .  \]
\end{proposition}
\begin{proof}
Recall that on event $\mathcal{E}_{t,c}$ we have that $|Z_t| < r_t g_t$. Suppose $d = 1$. Then the result follows immediately from Proposition~\ref{prop:quickpath1} and the properties of the scaling function $s_t$ in \eqref{scalingfns}, since
$$  \log \log r_t g_t \sim \log \log t  \, .$$
Suppose then $d \ge 2$. Note that conditioning on $\xi$ determines $\Pi^{(L_t)}$ and also that, by Lemma~\ref{lem:assep}, eventually almost surely $\Pi^{(L_t)}$ satisfies the properties required by the set $S_t$. Since $Z_t \in \Pi^{(L_t)}$, conditioning on the values of $\sigma$ in $B(\Pi^{(L_t)},j)$ therefore determines $Z_t$. Given $Z_t$ and $\Pi^{(L_t)}$, the event $\Theta_t^d$ is fully determined by the values of $\sigma$ in $\mathbb{Z}^d\setminus B(\Pi^{(L_t)},j)$. Hence we can apply Proposition~\ref{prop:quickpath2} with $z_t=Z_t$, $\sigma_t=s_t$ and $S_t=\Pi^{(L_t)}$, to deduce that there exists a $c_1 < 1$ such that, for all functions $c_t \to \infty$ such that $\bar F_\sigma(s_t) c_t \ll 1$,
$$ |Z_t|_{\rm{chem}} <  |Z_t| ( 1 + \bar F_\sigma(s_t) c_t + t^{-c_1}  ) $$
with probability tending to $1$. By \eqref{scalingfns}, we can pick a $c_t$ such that
$$ r_t g_t \bar F_\sigma (s_t) c_t \ll r_t b_t \, ,$$
and so we have the result.
\qedhere

%
\end{proof}

We are now ready to prove the lower bound.
\begin{proposition}
\label{prop:lowerbound}
For each $c \in \mathbb{R}$, on the events $\mathcal{E}_{t,c}$ and $\Theta^d_t$, as $t \to \infty$,
$$ \log U(t) \ge t\lambda^{(j)}(Z_t)-\frac{|Z_t|}{\gamma}\log\log t + O(t d_t b_t) \,  $$
almost surely.
\end{proposition}

\begin{proof}
In the following proof set $z = Z_t$ and abbreviate $\tau = \tau_z$. We first consider the case of $d\ge2$. By the Feynman-Kac formula \eqref{eq:fk}, the total mass $U(t)$ can be written as
\[  U(t)=\EE_0 \left[ \exp \left\{\int_0^t\xi(X_s)ds \right\} \right] \, .\]
Using the non-negativity of $\xi$ and by the strong Markov property, we have, for each $ r\in(0,1)$,
\begin{align}\label{eq:Usplit}\notag
 U(t) & \ge \EE_0\Big[\exp\big\{\int_0^t\xi(X_s)ds\big\} \id_{\{\tau< r t\}}\Big] \\ \notag
& \ge \EE_0\Big[\exp\big\{\int_\tau^{t-( r t-\tau)}\xi(X_s)ds\big\} \id_{\{\tau< r t\}}\Big] =\EE_0\Big[\exp\big\{\int_\tau^{t-( r t-\tau)}\xi(X_s)ds\big\}\Big]\PP_0(\tau< r t) \\ 
& \ge u_z((1-r)t, z) \, \PP_0(\tau < rt) \, . 
\end{align}
We now seek to bound $\PP_0(\tau < rt)$. Since we are on event $\Theta^d_t$, there exists a path 
\[p \in \bigcup_{y \in \partial B(z, j) } \Gamma_{\ell_t} (0, y)\]
for some $\ell_t < |z| + r_t b_t$ such that $\sigma(x) < s_t$ for all $x \in \mathrm{Set}(p)$. Moreover, since we are on event $\mathcal{E}_{t, c}$, each $\sigma(x) \in B(z, j) \setminus \{z\}$ is such that $\sigma(x) < a_t^\nu$ for some $\nu \in (0, 1)$.  We shall denote by $\{\tilde X_t\}_{t\in\mathbb{R}^+}$ a random walk with generator $\Delta\tilde\sigma^{-1}$, where $\tilde\sigma(x)=s_t$ for all $x \in \mathrm{Set}(p)$, $\tilde\sigma(x)=a_t^\nu$ for all $x \in B(z, j) \setminus \{z\}$, and $\tilde\sigma(x)=\sigma(x)$ otherwise. By a simple coupling argument we have that
\begin{align}\label{eq:taucouple}\PP_0(\tau< r t)\ge\PP_0(\tilde\tau< r t) \, ,
\end{align}
 where $\tilde\tau$ is the first hitting time of $z$ by $\tilde X$. Using a similar calculation as in \cite{Konig09}[Proposition~4.2], for any $r_1 + r_2 \le r$,
\begin{align*}
  \PP_0(\tilde\tau< r t) & >(2d)^{-\ell_t - j} \, \PP \left(\text{Poi}( r_1 ts_t^{-1})= \ell_t \right) \PP \left(\text{Poi}( r_2 t a_t^{-\nu} ) = j \right)\\
& = (2d)^{- \ell_t - j} e^{- r_1 ts_t^{-1}}\frac{( r_1 ts_t^{-1})^{\ell_t}}{(\ell_t)!}  e^{- r_2 t a_t^{-\nu} }\frac{( r_2 t a_t^{-\nu})^{j}}{j!} \, .
\end{align*}
Applying Stirling's formula, we obtain
\begin{align}
\label{eq:stir}
 \log\PP_0(\tilde\tau< r t)&\ge - r_1 ts_t^{-1} - r_2 t a_t^{-\nu}
-\ell_t  \log \left( \frac{ 2d \, \ell_t }{e r_1 t s_t^{-1}}\right) + j \log r_2 + O(\log t)  \, .
\end{align} 
Now note that on the event $\mathcal{E}_{t, c}$ we have that $Z_t \in \Pi^{(L_{t, \varepsilon})}$. Hence we can combine equations \eqref{eq:Usplit}--\eqref{eq:stir} and Lemma~\ref{lem:uloweratz} to get that
\begin{align*}
 \log U(t) & \ge(1-r_1 - r_2) t\lambda^{(j)}(z)-r_1 ts_t^{-1} - r_2 t a_t^{-\nu} - \ell_t \log \left( \frac{ 2d \,   \ell_t}{e r_1 ts_t^{-1}} \right) + j \log r_2 + O(\log t) \, .
\end{align*}
Use the bound $\ell_t < |z| + r_t b_t$ and choose $r = r_1 + r_2$ to maximise this equation, that is, set
\[  r_1 :=\frac{|z| + r_t b_t}{t(\lambda^{(j)}(z)+s_t^{-1})} \quad \text{and} \quad r_2 :=  \frac{ j }{ t( \lambda^{(j)}(z) + a_t^{-\nu} ) } \, .\]
It is clear that on event $\mathcal{E}_{t,c}$ we have $r \in (0,1)$. With these values of $r_1$ and $r_2$ we obtain
\[  \log U(t)\ge t\lambda^{(j)}(z)-(|z|+ r_t b_t) \Big\{\log\Big(\frac{\lambda^{(j)}(z)+s_t^{-1}}{s_t^{-1}}\Big) +O(1)\Big\} + O(\log t) \, . \]
On event $\mathcal{E}_{t,c}$ we have that $\lambda^{(j)}(z) < a_t(1+f_t)$. Since also $|z|<r_tg_t$ on event $\mathcal{E}_{t,c}$ we find that
\begin{align*}
  \log U(t)&\ge t\lambda^{(j)}(z)-|z|\log(\lambda^{(j)}(z)) - r_t b_t \log(\lambda^{(j)}(z)) + O \left( r_t g_t \log(s_t) \right) \\
&\ge t\lambda^{(j)}(z)-\frac{|z|}{\gamma}\log\log t+ O\left(t d_t b_t\right) 
\end{align*}
by the choice of the scaling functions $s_t$ in equation (\ref{scalingfns}).

Next, we turn to the case $d = 1$. Denote by $\{\bar X_t\}_{t\in\mathbb{R}^+}$ a random walk with generator $\Delta\bar\sigma^{-1}$ where $\bar\sigma(x)=\sigma_t^i$ if $\sigma(x)\in(\sigma_t^{i-1},\sigma_t^i]$. Again, by a simple coupling argument
\[  \PP_0(\tau<rt)\ge\PP_0(\bar\tau<rt) \,, \]
where $\bar\tau$ is the first hitting time of $z$ by $\bar X$ and $r \in (0,1)$. Furthermore, we have
\begin{align*}
 \PP_0(\bar\tau< rt) >  2^{-|Z_t|}\prod_{i=1}^{I_t}\PP(\text{Poi}(r_i t(\sigma_t^i)^{-1}) =  N^p_i),
\end{align*}
for any $\{r_i\}_{1\le i\le I_t}$ satisfying $\sum_i r_i \le r$. By a similar calculation to the $d\ge 2$ case, we have
\begin{align*}
 \log U(t) \ge t (1 - r) {\lambda}^{(j)}(z)  + \sum_{i=1}^{I_t} \left( - r_i t (\sigma_t^i)^{-1} - N^p_i \log (2 N^p_i /( e r_i t (\sigma_t^i)^{-1}) \right) + O(\log t) \, .
\end{align*}
Choose $r$ and $\{r_i\}$ to maximise this equation, that is, set
\[ r_i = \frac{N^p_i }{t(\lambda^{(j)}(z)+(\sigma_t^i)^{-1})} \quad \text{and} \quad  r = \sum_i r_t   \]
noting that $r \in (0, 1)$ for the same reason as in the $d \ge 2$ case. Then,
\begin{align*}
 \log U(t) & \ge t {\lambda}^{(j)}(z) + \sum_{i=1}^{I_t} \left( - N^p_i \left( \log \left(  {\lambda}^{(j)}(z) \sigma_t^i  \right)  \right) + O(1) \right) + O(\log t)  \\
 & = t{\lambda}^{(j)}(z) - |z| \log \left(  {\lambda}^{(j)}(z)  \right) - \sum_{i=1}^{I_t} \left( N^p_i \log  \sigma_t^i   + O(|z|) \right) + O(\log t) \, .
\end{align*}
The result follows since we are on event $\Theta_t^d$.
\end{proof}

\subsection{Contribution from each $U^i(t)$ is negligible}

In this section we prove that the contribution to $U(t)$ from the each of the components $U^i(t)$, for $i=2,3,4,5$, is negligible. The most difficult step is bounding the contribution from the components $U^2(t)$ and $U^3(t)$. 

The difficulty with these components is that paths are permitted to visit sites of high potential that are not $Z_t$. Away from these sites, there is a probabilistic penalty associated with each step of the path; this is easy to bound. However, close to these sites, the maximum contribution from the path may come from a complicated sequence of return cycles to the site. This motivates our set-up, which groups paths into equivalence classes depending only on their trajectory away from sites of high potential.

For each $t,$ we define a partition of paths into equivalence classes as follows. Suppose $p, \bar p \in \Gamma$ are two finite paths in $\mathbb{Z}^d$.  Define inductively,  $r^0=0$, and
\begin{align*}
s^{\ell} :=\min\{i \ge r^{\ell-1}:\,p_i \in\Pi^{(L_{t})}\} \quad \text{and} \quad  r^\ell :=\min\{i>s^\ell:\,p_i \in \partial B(p_{s^\ell},j)\}
\end{align*}
for each $\ell \in \mathbb{N}$, setting each to be $\infty$ if no such minimum $i$ exists, and define similarly $(\bar s^\ell, \bar r^\ell)_{\ell\ge1}$ for path $\bar p$. Then we say that $p$ and $\bar p$ are in the same equivalence class if and only if, for all $\ell \ge 0$,
\[ s^{\ell+1}-r^\ell=\bar s^{\ell+1}-\bar r^\ell \quad \text{and} \quad p_{r^\ell+i}=\bar p_{\bar r^\ell+i} \ , \quad \text{for each} \ i\in\{0,1,\ldots,s^{\ell+1}-r^\ell\} \, . \]
Note that although $s^\ell$ and $r^\ell$ depend on $t$ (through the set $\Pi^{(L_t)}$), we suppress this dependence for clarity. If $p$ and $\bar p$ are in the same equivalence class at time $t$ we write $p \sim \bar p$. Denote by $P(p):=\{\bar p \in \Gamma:\, p \sim\bar p\}$. Informally, the equivalence class $P(p)$ consists of paths that have identical trajectory except for when they are in balls of radius $j$ around sites $z \in \Pi^{(L_t)}$ (or, more accurately, when they first hit a site $z \in \Pi^{(L_t)}$ until when they leave the ball $B(z, j)$).

It is natural to group these equivalence classes $P(p)$ according to (i) how many balls of radius $j$ around sites $z \in \Pi^{(L_t)}$ the path visits; and (ii) the total length of the path outside such balls. So for $m, n \in \mathbb{N}$, let $\mathcal{P}_{n, m}$ be the set of equivalence classes $P(p)$ of paths $p$ that satisfy
\[ \max\{\ell:\,r^\ell<\infty\}=m \quad \text{and} \quad \sum_{\ell=0}^{m-1}(s^{\ell+1}-r^\ell)+ s^{m+1} \indic{s^{m+1} < \infty} + |p| \indic{s^{m+1}  =\infty} - r^m =n \, . \]
Note that if a path $p$ satisfies these two properties for some $m$ and $n$ then any other path $\bar p \in P_p$ will also satisfy these properties for the same $m$ and $n$ and hence $\mathcal{P}_{n,m}$ is well-defined. The quantity $m$ counts the number of balls of radius $j$ around $z \in \Pi^{(L_t)}$ that the path \textit{exits} (which is easier to work with than the number of balls the path enters); the quantity $n$ counts the total length of the path between leaving each of these balls and hitting the next site $z \in \Pi^{(L_t)}$.

Recalling the definitions of $p(X_t)$, define the event
$$\{ p(X) \in P(p) \} := \bigcup_{s \ge 0} \{ p(X_s) \in P(p) \}  \, , $$
and remark that we have the relationship
\begin{align}
\label{eq:indicle}
 \{p(X_t) \in P(p) \} \subseteq  \{ p(X) \in P(p)  \} \, . 
 \end{align}
Denote by
\[ U^{P(p)}(t)=\mathbb{E}_0 \left[ \exp\left\{\int_0^t\xi(X_s)\,ds\right\} \indic{p(X_t)\in P(p)}\right].
\]
the contribution to the total solution $U(t)$ from the path equivalence class $P(p)$. 

The following lemma bounds the contribution of each $P(p) \in \mathcal{P}_{n, m}$ in terms of $m$ and $n$. The key fact motivating our set-up is that the contribution is decreasing in $n$.

\begin{lemma}[Bound on the contribution from each equivalence class]
\label{L:negupper1}
Let $m, n \in \mathbb{N}$ and $p \in \Gamma(0)$ such that $\mathrm{Set}(p) \subseteq V_t$ and $P(p) \in \mathcal{P}_{n,m}$. Define $z^{(p)}:=\argmax_{z\in\mathrm{Set}(p)}{\lambda}^{(j)}(z)$ and let $\zeta>\max\{{\lambda}^{(j)}(z^{(p)}), L_{t,\varepsilon}\}$. Then there exist constants $c_1,c_2>0$ such that, uniformly in $m, n, p$ and $\zeta$, as $t \to \infty$,
\begin{align*}
 U^{P(p)}(t)
&\le e^{\zeta t} \big(c_1(\zeta-L_{t}) \big)^{-n}
 \left(1+c_2 \left(\zeta-{\lambda}^{(j)}(z^{(p)})\right)^{-1}\right)^m
\end{align*}
eventually almost surely.
\end{lemma}

\begin{proof}
The strategy of the proof is to split $U^{P(p)}(t)$ into three components, corresponding to the contribution: (i) from when $X_s$ is outside $B(\Pi^{(L_t)}, j)$ until $X_s$ hits a site $ z \in \Pi^{(L_t)}$; (ii) from when $X_s$ hits $z \in B(\Pi^{(L_t)}, j)$ until when $X_s$ leaves the ball $B(z, j)$; and (iii) if $X_s$ hits $ z \in \Pi^{(L_t)}$ and does not subsequently leave $B(z, j)$, from this component separately. To bound the contribution from these components, we make use of Corollary \ref{cor:tm}, Lemma~\ref{lem:clusterexp} and Lemma~\ref{lem:pathwise} respectively. 

There are two cases to consider, depending on whether the event described in (iii) occurs, that is, if $s^{m+1} < \infty$. We begin with this case. To simplify notation in the following we abbreviate
\[  I_a^b:=\exp\left\{\int_a^b(\xi(X_s)-\zeta)\,ds\right\} .\]
Recall the definition of $(s^\ell,r^\ell)_{\ell \in \mathbb{N}}$ and define the stopping times
\begin{align*}
R^0  :=0  \ , \quad S^{\ell}:=\inf\{s\ge R^{\ell-1}:\,X_s=p_{s^{\ell}}\} \quad \text{and} \quad R^\ell:=\inf\{s\ge S^\ell:\,X_s=p_{r_t^\ell}\}
\end{align*}
for each $\ell \in \{1, \ldots, m\}$, and similarly define $S^{m+1}$ since $s^{m+1} < \infty$. We can then write
\begin{align}
U^{P(p)}(t) &= \EE_0 \left[e^{\int_0^t\xi(X_s)\,ds} \indic{p(X_t) \in P(p)}\right] =e^{\zeta t} \, \EE_0\left[I_0^t \, \indic{p(X_t) \in P(p)}\right]\notag \\
&=e^{\zeta t} \, \EE_0\left[\left(\prod_{\ell=0}^{m} I_{R^\ell}^{S^{\ell+1}}\right)
\left(\prod_{\ell=1}^{m} I_{S^\ell}^{R^{\ell}}\right) I_{S^{m+1}}^t \indic{p(X_t) \in P(p)}\right].\notag
\end{align}
Note that, conditionally on $\mathcal{F}_{S^{m+1}}$ (the $\sigma$-algebra generated by $S^{m+1}$), the quantity $I_{S^{m+1}}^t$ is independent of all other $I_a^b$ in this expectation. Thus we have
\begin{align}
U^{P(p)}(t) &= e^{\zeta t} \, \EE\Bigg\{ \, \EE_0\left[\left(\prod_{\ell=0}^{m}I_{R^\ell}^{S^{\ell+1}}\right)
\left(\prod_{\ell=1}^{m}I_{S^\ell}^{R^{\ell}}\right) \indic{p(X_t) \in P(p)} \Big|\,\mathcal{F}_{S^{m+1}}\right]\notag\\
&\phantom{=e^{\zeta t}\EE\Bigg\{} \times \EE_0\left[I_{S^{m+1}}^t \indic{p(X_t) \in P(p)} \big|\,\mathcal{F}_{S^{m+1}}\right]\Bigg\}. \label{E:upper1}
\end{align}
We use Corollary \ref{cor:tm} to bound the expectation on the second line of \eqref{E:upper1}; in the calculation that follows, abbreviate $s := s^{m+1}$ and $S := S^{m+1}$. We obtain, for some $C > 1$,
\begin{align*}
 \EE_0 \left[I_{S}^t \indic{p(X_t) \in P_t(p)} \big|\, S \right] &  \le \indic{S \le t} \, \EE_{p_s} \left[ I_{0}^{t-S} \indic{\tau_{B(p_s, j)} > t - S  } \Big| S \right]   \le C  e^{(t-S)(\lambda^{(j)}(p_s)-\zeta)}  \le C 
  \end{align*}
almost surely, since $\zeta>\lambda^{(j)}(p_s)$. Combining with \eqref{E:upper1} and using equation \eqref{eq:indicle} we obtain
\begin{align}
U^{P(p)}(t) & \le C \, e^{\zeta t} \, \EE_0\left[\left( \prod_{\ell=0}^{m} I_{R^\ell}^{S^{\ell+1}} \right)
\left(\prod_{\ell=1}^{m} I_{S^\ell}^{R^{\ell}}\right) \indic{ p(X) \in P(p)}\right]\notag\\
&= C e^{\zeta t} \,
\EE_0\left[\left(\prod_{\ell=0}^{m} I_{R^\ell}^{S^{\ell+1}}\right) \indic{ p(X) \in P(p)}\right]\EE_0\left[
\left(\prod_{\ell=1}^{m} I_{S^\ell}^{R^{\ell}}\right) \indic{ p(X) \in P(p)}\right].\label{E:upper2}
\end{align}
Let $\xi^{(\ell)}_{\max{}}=\max_{r^\ell\le k<s^{\ell+1}}\xi(p_k)$, for $\ell = \{0, 1, \ldots, m\}$. By Lemma~\ref{lem:pathwise}, which we can apply here since $\zeta> L_{t, \varepsilon} > L_t \ge \max_{0 \le l \le m} \xi^{(\ell)}_{\max{}}$,
\begin{align}
 \EE_0\left[\left(\prod_{\ell=0}^{m} I_{R^\ell}^{S^{\ell+1}}\right)
 \indic{p(X) \in P(p) }\right] & \le (2d)^{-n} \prod_{\ell=0}^{m} \prod_{k=r^\ell}^{s^{\ell+1}-1} \! \left(1+\sigma(p_k)(\zeta-\xi^{(\ell)}_{\max{}})\right)^{-1} \label{eq:r_to_s}\\
&  \le (2d)^{-n} \left(1+\delta_\sigma(\zeta-L_t)\right)^{-n},\notag
\end{align}
almost surely, using the definition of $n$ and the lower bound on $\sigma$. Making the new abbreviation $s := s^\ell$, we have
\begin{align*}
 \EE_0 \left[\left(\prod_{\ell=1}^{m} I_{S^\ell}^{R^\ell}\right) \indic{p(X) \in P(p) } \right]
&=  \prod_{\ell=1}^{m}\EE_{ p_s  } \left[I_0^{\tau_{B( p_s, j) } } \indic{p(X) \in P(p) } \right]  \le \prod_{\ell=1}^{m}\EE_{ p_s  } \left[I_0^{\tau_{B( p_s, j) } }\right]  \, .
\end{align*}
Since $\zeta > \lambda^{(j)}(z^{(p)})$, we can apply the first bound in the cluster expansion in Lemma~\ref{lem:clusterexp} to deduce that
\begin{align}\label{eq:s_to_r}
 \prod_{\ell=1}^{m}\EE_{ p_s } \left[I_0^{\tau_{B( p_s, j) } }\right] \le \left(1+\frac{\delta^{-1}_\sigma |B(0,j)|}{\zeta- \lambda^{(j)}(z^{(p)})}\right)^{m}.
\end{align}
Using these two estimates, we obtain from equation (\ref{E:upper2}) the desired bound.

We now deal with the case that $s^{m+1} = \infty$. Similarly to the above, we condition on $\mathcal{F}_{R^{m}}$ (the $\sigma$-algebra generated by $R^m$) to write $ U^{P(p)}(t)$ as
\begin{align*}
 e^{\zeta t} \,  \EE\Bigg\{ \EE_0 \left[ \left( \prod_{\ell=0}^{m}I_{R^\ell}^{S^{\ell+1}} \right)
\left( \prod_{\ell=1}^{m}I_{S^\ell}^{R^{\ell}} \right) \indic{p(X)\in P(p)} \Big|\,\mathcal{F}_{R^{m}}\right] \EE_0\left[I_{R^m}^t\indic{R^m\le t}\big|\,\mathcal{F}_{R^{m}}\right]\Bigg\}.
\end{align*}
Set $l:=|p|-r^m>0$ and $\tau_{\text{end}} := \inf\{ s>0 : X_s = X_t\}$. Observe that, since $\zeta > L_{t, \varepsilon} > L_t \ge \xi(X_t)$, almost surely
\begin{align*}
 \EE_0\left[I_{R^m}^t \indic{p(X_t)\in P(p)}\big|\,\mathcal{F}_{R^{m}}\right] \le  \EE_0 \left[I_{R^m}^{\tau_{\text{end}}}\indic{p(X_t) \in P(p)} \big|\,\mathcal{F}_{R^{m}}\right] 
\end{align*}
and applying Lemma~\ref{lem:pathwise} (valid by Lemma~\ref{lem:minmax}) we get that
\begin{align*}
 \EE_0 \left[I_{R^m}^t  \indic{p(X_t) \in P(p)}   \big|\,\mathcal{F}_{R^{m}}\right] \le (2d)^{-l} (1+\delta_\sigma(\zeta-L_t))^{-l} 
\end{align*}
almost surely. The rest of the proof proceeds similarly to the previous case.
\end{proof}

We can use Lemma~\ref{L:negupper1} to bound the contribution to the total mass $U(t)$ from $U^2(t)$ and $U^3(t)$.

\begin{proposition}[Upper bound on $U^2(t)$]
\label{prop:U2neg}
There exists a constant $c$ such that, as $t \to \infty$,
$$ \log U^2(t) \le t  \max_{z \in \Pi^{(L_t)} \setminus \{Z_t\} }\Psi_{t, c}^{(j)}(z) +  O(t d_t b_t) $$
almost surely.
\end{proposition}
\begin{proof}
Recall the path set $E_t^2$, and for each $m, n \in \mathbb{N}$ define
$$ \mathcal{P}^2_{n,m} := \bigcup_{p \in E_t^2} P_t(p) \cap \mathcal{P}_{n, m}   \, .$$ 
Note that $|\mathcal{P}^2_{n,m}|\le \kappa^{n+m}$, with $\kappa=\max\{2d,|\partial B(0,j)|\}$. We observe that
\begin{align*}
 U^2(t) &=  \sum_{n, m} U^{\mathcal{P}_{n,m}^{2}}(t) \le  \sum_{n, m} \kappa^{n+m} \max_{P \in \mathcal{P}^2_{n,m} } \left\{ U^{P}(t) \right\}  =  \sum_{n, m} \kappa^{-n -m}  \max_{P  \in \mathcal{P}^2_{n,m} } \left\{ \kappa^{2(n + m)}  U^{P}(t) \right\}  \\
  & \le \max_{n,m} \max_{P \in \mathcal{P}^2_{n,m}} \left\{ \kappa^{2(n + m)}  U^{P}(t)  \right\}  \sum_{n,m} \kappa^{-n -m} \, .
 \end{align*}
For each $P \in \mathcal{P}^2_{n,m}$, denote by $z^{(P)}$ the site $y \in \Pi^{(L_t)}$ on a given path $p \in P$ which maximises ${\lambda}^{(j)}(y)$, remarking that this a class property of $P$ eventually almost surely by Lemma~\ref{lem:assep}. Using Lemma~\ref{L:negupper1}, for each $P \in \mathcal{P}^2_{n,m}$ and for any $\zeta> \max\{{\lambda}^{(j)}(z^{(P)}), L_{t,\varepsilon}\}$, we have that there exist constants $c_1,c_2,c_3>0$ such that, eventually almost surely,
\begin{align*}
\kappa^{2(n+m)}  \, U^{P}(t)
&\le  e^{\zeta t} \left(c_1(\zeta-L_{t}) \right)^{-n}
 \left(c_2 + c_3(\zeta-{\lambda}^{(j)}(z^{(P)}))^{-1}\right)^m \, .
\end{align*}
Set $\zeta = \max\{{\lambda}^{(j)}(z^{(P)}), L_{t,\varepsilon}\} +  d_t b_t$. To lower bound $n$, observe that the number of steps between exiting a $j$-ball and hitting another site in $\Pi^{(L_t)}$ is at least $j+1$. We apply Corollary~\ref{cor:pathoutsideballs} to the balls $B(\Pi^{(L_t)},j+1)$ to deduce that, eventually almost surely \begin{align}\label{eq:boundonn}
n > m(j+1) + |z^{(P)}| - |z^{(P)}|^{c_4}\,,
\end{align} for some $c_4 < 1$.
Then, by monotonicity in $n$,
 \begin{align*}
 \kappa^{2(n+m)}  \, U^{P}(t)
&\le e^{t ( {\lambda}^{(j)}(z^{(P)})  + d_t b_t)} (c_1(L_{t, \varepsilon} -L_{t}))^{-|z^{(P)}| + |z^{(P)}|^{c_4}} \\ 
 & \qquad \times \left(c_1(L_{t, \varepsilon} -L_{t}))^{-j-1} (c_2+c_3 d_t b_t)^{-1}) \right)^m 
\end{align*}
eventually almost surely.  Note that $j$ was chosen precisely to be the smallest integer such that
\begin{align}\label{eq:j} (j + 1) \log a_t + \log (d_t) \to \infty \end{align}
which implies, since $b_t \gg 1/\log \log t$ by \eqref{scalingfns}, that
\[
(j + 1) \log a_t + \log (c_2 + c_3 d_t b_t) \to \infty \, .
\]
By Lemma~\ref{lem:assep}, for $z \in \Pi^{(L_{t})}$, as $t \to \infty$,
$$ |z|^{c_4} < t d_t b_t$$
eventually almost surely. Moreover,
\[ \log \left(L_{t, \varepsilon} - L_t \right) > \log a_t + c_5 \]
eventually for some positive $c_5$. So there exists a constant $c$ such that
 \begin{align*}
 2(n+m) \log \kappa + \log U^{P}(t)
&\le c |z^{(P)}|  + {\lambda}^{(j)}(z^{(P)})  t  - \frac{1}{\gamma} |z^{(P)}| \log \log t + t d_t b_t
\end{align*}
eventually almost surely, which yields the result.
 \end{proof}

\begin{proposition}[Upper bound on $U^3(t)$]
\label{prop:U3neg}
There exists a constant $c$ such that, as $t \to \infty$,
$$ \log U^3(t) \leq t {\Psi}_{t, c}^{(j)}(Z_t) - h_t \frac{1}{\gamma} |Z_t| \log \log t + O\left( t d_t b_t \right) $$ almost surely.
\end{proposition}
\begin{proof}
Recall the set of paths $E_t^3$ and define $\mathcal{P}^3_{n,m}$ by analogy with $\mathcal{P}^2_{n,m}$. The proof then follows as for Proposition~\ref{prop:U2neg} after strengthening the bound in \eqref{eq:boundonn} to give that for each $p \in E_t^3$ and for some $c_1 <1$, eventually almost surely
\begin{equation*}
n > m(j+1) + (1 + h_t) \frac{1}{\gamma} |Z_t| \log \log t - |Z_t|^{c_1} \, . \qedhere
\end{equation*}
\end{proof}

\begin{proposition}[Upper bound on $U^4(t)$]
\label{prop:U4neg}
For all $t \ge 0$,
$$ U^4(t) \le e^{t L_t}  \, .$$
\end{proposition}
\begin{proof}
This follows trivially from the definition of $U^4(t)$. \end{proof}

\begin{proposition}[Negligibility of $U^5(t)$]
\label{prop:U5neg}
As $t \to \infty$, almost surely,
$$ \frac{U^5(t)}{U(t)} \to 0 \, .$$
\end{proposition}
\begin{proof}
The equivalent statement for the PAM with Weibull potential is proved in \cite[Section 2.5]{Gartner98}, and is a consequence of a large probabilistic penalty for diffusing outside the macrobox $V_t$. The assumption that $\sigma(0) > \delta_\sigma$ ensures that the proof applies equally well in our case.
\end{proof}

\begin{corollary}
\label{cor:neg}
There exists a constant $c$ such that, as $t \to \infty$,
$$ \frac{U^2(t) + U^3(t) + U^4(t) + U^5(t)}{U(t)} \id_{\mathcal{E}_{t,c}} \id_{\Theta^d_t}\to 0$$
almost surely.
\end{corollary}
\begin{proof}
Let $c$ be the maximum of the constants appearing in Propositions~\ref{prop:U2neg} and \ref{prop:U3neg}. Combining Propositions~\ref{prop:lowerbound} and \ref{prop:U2neg}, and recalling that $Z_{t, c}^{(j)} = Z_t$ eventually by Proposition~\ref{prop:zeqzc} and Corollary \ref{corry:zeqz0}, we have that, on the events $\mathcal{E}_{t,c}$ and $\Theta^d_t$, eventually almost surely
$$ \log U^2(t) - \log U(t)  \le  t \left( {\Psi}_{t, c}^{(j)}(Z_{t, c}^{(j, 2)}) - {\Psi}_{t, c}^{(j)}(Z_{t, c}^{(j)})  \right) +c|Z_{t}|+ O(td_tb_t) \, .$$
Using the gap in the maximisers of $\Psi_{t, c}^{(j)}$ and since $|Z_t| < r_tg_t$, we have that, as $t \to \infty$,
$$ \log U^2(t) - \log U(t)  \le - t d_t e_t +O(r_tg_t)+ O(td_t b_t)\to  -\infty $$
by the properties of the scaling functions in \eqref{scalingfns}. Similarly, combining Propositions \ref{prop:lowerbound} and~\ref{prop:U3neg}, we have that, on the events $\mathcal{E}_{t,c}$ and $\Theta^d_t$, eventually almost surely
$$ \log U^3(t) - \log U(t)  \le  - h_t \frac{1}{\gamma} |Z_t| \log \log t + c |Z_t| + O(t d_t b_t) $$
and so, using that $|Z_t| >  r_t f_t$ on the event $\mathcal{E}_{t,c}$, as $t \to \infty$,
$$ \log U^3(t) - \log U(t)  \le - r_t f_t h_t \frac{1}{\gamma} \log \log t  + O(t d_t b_t) \to  -\infty $$
by the properties in \eqref{scalingfns}. Finally, combining Propositions \ref{prop:lowerbound}, \ref{prop:U4neg} and \ref{prop:U5neg}, we get the result.
\end{proof}

\section{Localisation}
\label{sec:loc}
In this section we complete the proof of Theorem~\ref{thm:main1}; that is, we show that the non-negligible component of the total solution, $u^1(t, z)$, is eventually localised at $Z_t$. Recall the idea of the proof that was outlined in Section \ref{sec:outline}, that: (i) the solution $u^1(t, z)$ is closely approximated by the principal eigenfunction of $\Delta \invsigma + \xi$ restricted to the domain 
\[ B_t := B \left( 0, |Z_t| (1 + h_t)  \right)  \cap V_t \] 
 and; (ii) the principal eigenfunction decays exponentially away from $Z_t$. Throughout this section, fix the constant $c > 0$ from Corollary~\ref{cor:neg}.

\subsection{Approximating the solution with the principal eigenfunction}
Let $\lambda_t$ and $v_t$ denote, respectively, the principal eigenvalue and eigenfunction of the Hamiltonian $\Delta \invsigma + \xi$ restricted to the domain $B_t$ with Dirichlet boundary conditions, renormalising $v_t$ so that $v_t(Z_t) = 1$. 

\begin{lemma}[Gap in $j$-local principal eigenvalues in $B_t$]
\label{lem:gap}
On the event $\mathcal{E}_{t, c}$, each $z \in B_t \setminus \{Z_t\}$ satisfies
$$ {\lambda}^{(j)}(Z_t) - {\lambda}^{(j)}(z) > d_t e_t + o(d_t e_t) \, .$$
\end{lemma}
\begin{proof}
On the event $\mathcal{E}_{t, c}$, we have that $\lambda^{(j)}(Z_t) > a_t(1 - f_t)$ and so the claim is true for $z \notin \Pi^{(L_t)}$ by Lemma \ref{lem:minmax}. On the other hand, if $z \in \Pi^{(L_t)}$ then
$$ d_t e_t < {\Psi}_t^{(j)}(Z_t) -  {\Psi}_t^{(j)}(z) =  {\lambda}^{(j)}(Z_t) -  {\lambda}^{(j)}(z) + \frac{|z| - |Z_t|}{\gamma t} \log \log t  \, . $$
To complete the proof, notice that, for each $z \in B_t$,
$$ \frac{|z| - |Z_t| }{\gamma t} \log \log t <  \frac{r_t g_t h_t}{\gamma t} \log \log t =  d_t g_t h_t \ll d_t e_t$$
since $g_t h_t \ll e_t$ by \eqref{scalingfns}.
\end{proof}

\begin{corollary}
\label{cor:gap}
Eventually on the event $\mathcal{E}_{t, c}$, each $z \in B_t \setminus \{Z_t\}$ satisfies
$$ \lambda_t > {\lambda}^{(j)}(z) +d_te_t+o(d_te_t)  \, .$$
\end{corollary}
\begin{proof}
First note that, on the event $\mathcal{E}_{t, c}$, the ball $B(Z_t, j) \subseteq B_t$. Hence, by the domain monotonicity in Lemma~\ref{lem:mono}, we have $\lambda_t \ge {\lambda}^{(j)}(Z_t)$, and so the result follows from Lemma~\ref{lem:gap}.
\end{proof}

\begin{proposition}[Feynman-Kac representation for the principal eigenfunction]
\label{prop:pathexpeig}
Eventually on the event $\mathcal{E}_{t, c}$, 
\begin{align*}
v_t(z) & = \frac{\sigma(z)}{\sigma(Z_t)}  \EE_z \left[ \exp \left\{ \int_0^{\tau_{Z_t}} \left( \xi(X_s) - \lambda_t \right) \, ds  \right\}  \id_{\{ \tau_{B_t^c} > \tau_{Z_t} \}}  \right] \, ,
\end{align*}
where 
$$ \tau_{Z_t} := \inf\{ t \ge 0 : X_t = Z_t \}  \quad \text{and} \quad  \tau_{B_t^c} := \inf\{ t \ge 0 : X_t \notin B_t \} \, .$$
\end{proposition}
\begin{proof}
This is an application of Proposition~\ref{prop:genfk}, valid precisely because of Corollary~\ref{cor:gap}.
\end{proof}

\subsection{Exponential decay of the principal eigenfunction}

Recall the partition of paths into equivalence classes in Section \ref{sec:neg}, the quantities $r^\ell$ and $s^\ell$ associated to each equivalence class, and, for $m, n \in \mathbb{N}$, the set of equivalence classes $\mathcal{P}_{n, m}$. Recall also the event $\{p(X) \in P(p)\}$.

Define the path set
$$ \bar{E}_t^1 := \left\{ p \in E_t^1: |p| = \min \left\{ i: p_i = Z_t  \right\} \right\} \,,$$ 
 and for each $m, n \in \mathcal{N}$ define
$$ \bar{\mathcal{P}}^1_{n, m} := \bigcup_{p \in \bar{E}_t^1} P_t(p) \cap \mathcal{P}_{n,m} \,.  $$
Further, for each $P \in  \bar{\mathcal{P}}^1_{n, m}$ and $y \in B_t$ define
\begin{align}
\label{eq:fkvp}
 v^{P}_t(y) := \frac{\sigma(y)}{\sigma(Z_t)} \mathbb{E}_{y}\left[ \exp\left\{\int_0^{\tau_{Z_t}} \left( \xi(X_s) - \lambda_t \right) \,ds\right\}\indic{p(X) \in P}\right] \, .
\end{align}
For each $P \in \bar{\mathcal{P}}^1_{n, m}$ denote by $z^{(P)}$ the site $y \in \Pi^{(L_t)}$ on a given path $p \in P$, excluding the site $Z_t$, which maximises ${\lambda}^{(j)}(y)$, setting $z^{(P)} = \varnothing$ (and ${\lambda}^{(j)}(\varnothing) = 0$) if no such $y$ exists. Remark that, whenever $z^{(P)}$ is defined, it is a class property of $P$ eventually almost surely, by Lemma~\ref{lem:assep}.

\begin{lemma}[Bound on the contribution from each equivalence class]
\label{L:negupper2}
Let $m, n \in \mathbb{N}$ and $P \in \bar{\mathcal{P}}^1_{n, m}$. Then there exist constants $c_1,c_2>0$ such that, for each $m, n, P$ and $y \in B_t \setminus \Pi^{(L_t)}$ uniformly, as $t \to \infty$,
\begin{align*}
 v^{P}_t(y) \, \sigma(Z_t) \le  \left(c_1(\lambda_t-L_{t}) \right)^{-n}
 \left(1+c_2(\lambda_t-{\lambda}^{(j)}(z^{(P)}))^{-1}\right)^{m-1}
\end{align*}
and, for every $y \in \Pi^{(L_t)}$ uniformly,
\begin{align*}
 v^{P}_t(y) \, \sigma(Z_t) \le  \left( \lambda_t-{\lambda}^{(j)}(z^{(P)}) \right)^{-1} \left(c_1(\lambda_t-L_{t}) \right)^{-n}
\left(1+c_2(\lambda_t-{\lambda}^{(j)}(z^{(P)}))^{-1} \right)^{m-1}
\end{align*}
both hold eventually almost surely.
\end{lemma}
\begin{proof}
Starting with the Feynman-Kac representation for $v^{P}_t(y)$ in equation \eqref{eq:fkvp}, the proof follows similarly as in Lemma~\ref{L:negupper1} for $\zeta = \lambda_t$, which is a valid setting for $\zeta$ because of Corollary~\ref{cor:gap}. Two modifications are necessary to adapt the proof.

The first modification comes from the observation that, for any $p\in P$, the final site $Z_t$ gives no contribution to the expectation, and hence we have $m-1$ instead of the $m$ in Lemma~\ref{L:negupper1}. 

The second modification is necessary to take into account the additional $\sigma(y)$ factor present in the Feynman-Kac representation in equation \eqref{eq:fkvp}, which \textit{a priori} could be arbitrarily large. How we take this into account depends on whether $p$ starts at a site of high potential. If $y \notin \Pi^{(L_t)}$, we simply modify equation \eqref{eq:r_to_s} by pulling out the factor $\sigma(y)$ and bounding the right-hand side by
\[
(2d)^{-n}\invsigma(y)(\lambda_t-L_t)^{-1} \left(1+\delta_\sigma(\lambda_t-L_t)\right)^{-n+1},
\]
and the claimed result follows. If $y \in \Pi^{(L_t)}$, we instead modify equation \eqref{eq:s_to_r} by using the \textit{second} bound in Lemma~\ref{lem:clusterexp} on the product factor for $\ell = 1$, which yields (abbreviating $s := s^\ell$)
\[
\EE_y[I_0^{\tau_{B(y,j)}}]\prod_{\ell=2}^{m-1}\EE_{ p_s} \left[I_0^{\tau_{B( p_s, j) } }\right]\le \invsigma(y)(\lambda_t-
{\lambda}^{(j)}(z))^{-1}\left(1+\frac{\delta^{-1}_\sigma |B(0,j)|}{\lambda_t- \lambda^{(j)}(z^{(P)})}\right)^{m-1} \, ,\]
and again the claimed result follows. 
\end{proof}

\begin{proposition}[Exponential decay of principal eigenfunction]
\label{prop:expdecay}
On the event $\mathcal{E}_{t, c}$ there exists a constant $C >0$ such that, for each $y \in B_t$ uniformly, as $t \to \infty$,
\[ \log v_t(y) +\log \sigma(Z_t) \le  - C |y -Z_t| \log \log t    \]
eventually almost surely.
\end{proposition}
\begin{proof}
As in Proposition~\ref{prop:U2neg}, we observe that there exists $\kappa>1$ such that
\begin{align*}
 v_t(y)  &=  \sum_{n, m} \sum_{ P \in \bar{\mathcal{P}}_{n,m}^{1} } v_t^{P}(y)  \le \max_{n,m} \max_{P \in \mathcal{P}^1_{n,m}} \left\{ \kappa^{2(n + m)}  v_t^{P}(y)  \right\}  \sum_{n,m} \kappa^{-n -m} \, .
 \end{align*} 
Suppose $y \in B_t\setminus \Pi^{(L_t)}$. Then for each $P \in \bar{\mathcal{P}}^1_{n,m}$, by Lemma~\ref{L:negupper2} there exist $c_1,c_2,c_3>0$ such that
\begin{align*}
 \kappa^{2(n + m)} \sigma(Z_t) v_t^{P}(y)
&\le (c_1 (\lambda_t-L_{t}))^{-n}
 (c_2 +c_3(\lambda_t-{\lambda}^{(j)}(z^{(P)}))^{-1})^{m-1}
\end{align*}
eventually almost surely. Note also that by Corollary~\ref{cor:pathoutsideballs} (similarly to \eqref{eq:boundonn}), eventually almost surely
\[ n > (m-1) (j+1) + c_4 |y - Z_t| \]
for any $c_4 < 1$. Then, for any $0 < \varepsilon < \theta$,
 \begin{align*}
  \kappa^{2(n + m)} \sigma(Z_t) v_t^{P}(y)
&\le  \left(c_1(L_{t, \varepsilon} -L_{t}) \right)^{-c_4 |y - Z_t| } \left((c_1(L_{t, \varepsilon} -L_{t}))^{-j-1} (c_2 +c_3( d_t e_t)^{-1}) \right)^{m-1} 
\end{align*}
eventually almost surely by monotonicity in $n$ and Corollary~\ref{cor:gap}, and so, applying equation~\eqref{eq:j}, there exists a $C > 0$ such that
 \begin{align*}
  2(n + m) \log \kappa +  \log v_t^{P}(y) +\log \sigma(Z_t) \le - C |y - Z_t| \log \log t
\end{align*}
eventually almost surely. Suppose then that $y\in \Pi^{(L_t)}$. Here we proceed similarly, but we now need the stronger bound $n > m(j+1) + c_4  |y - Z_t|$ for any $c_4 < 1$, valid eventually almost surely for $y \in \Pi^ {(L_t)}$ by Lemma~\ref{lem:assep}. Then,
\begin{align*}
  \kappa^{2(n + m)} \sigma(Z_t) v_t^{P}(y)
&\le  \left( (c_1(L_{t, \varepsilon} -L_{t}))^{-j-1} (d_t e_t)^{-1}  \right) \left(c_1(L_{t, \varepsilon} -L_{t}) \right)^{-c_4 |y - Z_t| } \\ 
 & \qquad \times \left((c_1 (L_{t, \varepsilon} -L_{t}))^{-j-1} (c_2 +c_3( d_t e_t)^{-1}) \right)^{m-1} \, ,
\end{align*}
and the rest of the proof follows as before.\end{proof}

\subsection{Completion of the proof of Theorem \ref{thm:main1}}
We are now in a position to establish Theorem~\ref{thm:main1}. First, remark that Proposition~\ref{prop:expdecay} implies that, as $t \to \infty$,
$$  \id_{\mathcal{E}_{t, c}} \, \sigma(Z_t) \sum_{z \in B_t \setminus \{Z_t\}} v_t(z)^2 \to 0 $$
almost surely, and so in particular $\id_{\mathcal{E}_{t, c}}\|v_t\|^2_{\ell_2}\to1$, since we know $\sigma(Z_t)>\delta_\sigma$.
Hence since
\begin{align}
\label{eq:loc}
\id_{\mathcal{E}_{t, c}} \, \sigma(Z_t) \, \|\sigma^{-\frac1{2}}v_t\|_{\ell_2}^{2} \, \sum_{z \in B_t \setminus \{Z_t\}}v_t(z) \ \le \ \id_{\mathcal{E}_{t, c}} \, \delta^{-1}_\sigma  \, \|v_t\|_{\ell_2}^{2} \,  \sigma(Z_t)  \sum_{z \in B_t \setminus \{Z_t\}}v_t(z) \, ,
\end{align}
the left-hand side of equation \eqref{eq:loc} also converges to zero almost surely. To finish the proof, we apply Proposition~\ref{prop:thm4.1}, which gives that
$$   \id_{\mathcal{E}_{t, c}}   \frac{1}{U(t)} \sum_{z \in B_t \setminus \{Z_t\}} u_1(t, z)  \to 0 $$
almost surely. Combining the above with the negligibility results already established in Corollary~\ref{cor:neg} on events $\mathcal{E}_{t, c}$ and $\Theta^d_t$, and the fact that the events $\mathcal{E}_{t, c}$ and  $\Theta^d_t$ hold eventually with overwhelming probability by Proposition~\ref{prop:existencepath}, we have established Theorem~\ref{thm:main1}.
\qed 
%

\bibliographystyle{plain}
\bibliography{paper}
\end{document}